\documentclass[11pt,twoside,reqno,psamsfonts]{amsart}

\usepackage[OT1]{fontenc}
\usepackage{type1cm}
\usepackage{amssymb}
\usepackage[left=2.3cm,top=3cm,right=2.3cm]{geometry}

\numberwithin{equation}{section}

\theoremstyle{plain}
\newtheorem{maintheorem}{Theorem}

\newtheorem{theorem}{Theorem}[section]

\newtheorem{proposition}[theorem]{Proposition}
\newtheorem{prop}[theorem]{Proposition}
\newtheorem*{prop*}{Proposition}
\newtheorem{lemma}[theorem]{Lemma}

\theoremstyle{remark}
\newtheorem{remark}[theorem]{Remark}

\newtheorem{example}[theorem]{Example}

\theoremstyle{definition}

\newcommand{\HH}{\mathcal{H}}

\newcommand{\MM}{\mathcal{M}}
\newcommand{\PP}{\mathcal{P}}

\newcommand{\LL}{\mathcal{L}}
\newcommand{\R}{\mathbb{R}}

\newcommand{\Z}{\mathbb{Z}}
\newcommand{\N}{\mathbb{N}}

\newcommand{\iii}{\mathtt{i}}

\newcommand{\eps}{\varepsilon}

\DeclareMathOperator{\diag}{diag}

\DeclareMathOperator{\inter}{int}

\DeclareMathOperator{\dimh}{dim_H}

\DeclareMathOperator{\GL}{{\it GL_d}(\R)}

\newcommand{\T}{\mathbb{T}}

\newcommand{\B}{\mathcal{B}}

\newcommand{\var}{\text{var}}

\newcommand{\D}{\mathcal{D}}

\newcommand{\A}{\mathcal{A}}

\newcommand{\M}{\mathcal{M}}
\newcommand{\F}{\mathcal{F}}

\newcommand{\supp}{\mbox{supp}}

\begin{document}

\title[Multifractal analysis of typical self-affine sets]{Multifractal analysis of Birkhoff averages for typical infinitely generated self-affine sets}

\author{Antti K\"{a}enm\"{a}ki}
\address{Antti K\"{a}enm\"{a}ki \\
Department of Mathematics and Statistics \\
P.O.\ Box 35 \\
FI-40014 University of Jyv\"{a}skyl\"{a} \\
Finland}
\email{antti.kaenmaki@jyu.fi}

\author{Henry WJ Reeve}
\address{Henry WJ Reeve \\
School of Mathematics \\
The University of Bristol \\
University Walk \\
Clifton \\
Bristol \\
BS8 1TW \\
UK}
\email{henrywjreeve@googlemail.com}

\thanks{HWJ Reeve would like to thank his supervisor Dr. Thomas Jordan for all of his help and guidance.}
\date{\today}

\begin{abstract}
We develop a thermodynamic formalism for quasi-multiplicative potentials on a countable symbolic space and apply these results to the dimension theory of infinitely generated self-affine sets. The first application is a generalisation of Falconer's dimension formula to include typical infinitely generated self-affine sets and show the existence of an ergodic invariant measure of full dimension whenever the pressure function has a root. Considering the multifractal analysis of Birkhoff averages of general potentials $\Phi$ taking values in $\R^{\N}$, we give a formula for the Hausdorff dimension of $J_\Phi(\alpha)$, the $\alpha$-level set of the Birkhoff average, on a typical infinitely generated self-affine set. We also show that for bounded potentials $\Phi$, the Hausdorff dimension of $J_\Phi(\alpha)$ is given by the maximum of the critical value for the pressure and the supremum of Lyapunov dimensions of invariant measures $\mu$ for which $\int\Phi\,d\mu=\alpha$. Our multifractal results are new in both the finitely generated and the infinitely generated setting.
\end{abstract}

\maketitle
\tableofcontents

\section{Introduction}

Let $F$ be the repeller of a piecewise smooth map $f \colon X \rightarrow X$. Given a continuous potential $\varphi \colon F \rightarrow \R^N$ and $\alpha \in \R^N$, we are interested in the set of points in the repeller for which the Birkhoff average converges to $\alpha$,
\begin{equation} \label{eq:central_birkhoff}
J_{\varphi}(\alpha) = \{ x \in F : \lim_{n \rightarrow \infty} \frac{1}{n} \sum_{i=0}^{n-1} \varphi(f^i(x))=\alpha \}.
\end{equation}
The central question in the multifractal analysis of Birkhoff averages is to determine the Hausdorff dimension of  the level sets $J_{\varphi}(\alpha)$. For conformal expanding maps on compact repellers the Hausdorff dimension is given by a well known conditional variational principle; see e.g.\ Pesin and Weiss \cite{PesinWeiss2001}, Fan, Feng and Wu \cite{FanFengWu2001}, Barreira and Saussol \cite{BarreiraSaussol2001}, Feng, Lau and Wu \cite{FengLauWu2002} and Olsen \cite{Olsen2003, Olsen2008}.

Situations in which either the map $f$ is non-conformal or the repeller $F$ is non-compact are far less well understood. Thus far most work on non-conformal systems has focused on maps which are obtained as skew products of conformal systems; see e.g.\ Barral and Mensi \cite{BarralMensi2008}, Barral and Feng \cite{BarralFeng2009} and Reeve \cite{Reeve2011, Reeve2012}. Jordan and Simon \cite{JordanSimon2007} have given a conditional variational principle for typical members of parameterizable families of self-affine iterated function systems with a simultaneously diagonalizable linear part.

Recently there has also been a great deal of work dealing with cases in which the repeller $F$ is a non-compact limit set of a countable collection of contractions; see e.g.\ \cite{FanLiaoMaWang2010, FanJordanLiaoRams2011, FanLiaoWangWu2009, FanLiaoMa2010, IommiJordan2010, JaerischKessebohmer2011, KessebohmerMundayStratmann2010, KessebohmerStratmann2007, Reeve2012}. All but one of these results have concerned situations in which the map $f$ is conformal. The only exception being \cite{Reeve2012} which deals with a family of skew products including the direct product of the Gauss map and the doubling map.

There are two facts concerning the the space of invariant measures for a continuous map of a compact metric space which make the dimension theory of compact systems a great deal easier to handle. The first is that if the space itself is compact, then the space of invariant probability measures is also compact. Thus given a sequence of invariant measures one can always extract a convergent subsequence. The second fact is that for compact systems entropy is an upper-semicontinuous function on the space of invariant measures, so given a sequence of invariant measures one may extract a weak star limit point with entropy equal to the limit superior of the entropies of the measures in the sequence. Since in our setting the underlying space $\N^\N$ is non-compact the main challenge comes from the lack of these two facts.

The article is organized as follows. In \S \ref{sec:preli}, we exhibit and motivate the results, and in \S \ref{sec:gibbs}--\ref{sec:conditional}, we provide the reader with all the necessary details.

\section{Preliminaries and statement of results} \label{sec:preli}

\subsection{Thermodynamic formalism for sub-multiplicative potentials} \label{sec:preli_thermo}

Define $\Sigma = \N^{\N}$ to be the set of all infinite words constructed from the integers. Let $\Sigma_n = \N^n$ for all $n \in \N$ and $\Sigma_* = \bigcup_{n \in \N} \Sigma_n$ be the collection of all finite word. If $\omega \in \Sigma_*$ and $\tau \in \Sigma_* \cup \Sigma$, then $\omega \tau$ denotes the concatenation of $\omega$ and $\tau$. Furthermore, if $\omega \in \Sigma_* \cup \Sigma$ and $n \in \N$, then $\omega|_n$ is the unique word in $\Sigma_n$ for which there is $\tau \in \Sigma$ so that $\omega|_n \tau = \omega$. If $\omega,\tau \in \Sigma_* \cup \Sigma$, then by $\omega \wedge \tau$ we mean the common beginning of $\omega$ and $\tau$. Given $n \in \N$ and $\omega \in \Sigma_n$ we set $|\omega|=n$ and define the cylinder set given by $\omega$ to be $[\omega] = \left\lbrace \omega\tau : \tau \in \Sigma \right\rbrace$. We denote the left shift operator by $\sigma$ and let $\M_{\sigma}(\Sigma)$ be the set of all $\sigma$-invariant Borel probability measures on $\Sigma$.

We equip $\Sigma$ with the discrete topology and call it a \emph{shift space}. If the shift space is constructed by using a finite alphabet, i.e.\ $\Sigma = I^\N$ for some finite set $I \subset \N$, then we say that the shift space is \emph{finitely generated}. The shift space is compact if and only if it is finitely generated.
Moreover, the cylinder sets are open and closed and they generate the Borel $\sigma$-algebra.

We shall consider maps $\varphi \colon \Sigma_* \rightarrow (0,\infty)$. We refer to such maps as \emph{potentials}. We say that a potential $\varphi$ is \emph{sub-multiplicative} if
\begin{equation*}
  \varphi(\omega \tau) \leq \varphi(\omega)\varphi(\tau).
\end{equation*}
for all $\omega, \tau \in \Sigma_*$. A sub-multiplicative $\varphi$ potential is said to be \emph{quasi-multiplicative} if there exist a constant $c \ge 1$ and a finite subset $\Gamma \subset \Sigma_*$ such that for any given pair $\omega, \tau \in \Sigma_*$ there exists $\kappa \in \Gamma$ with
\begin{equation} \label{eq:quasi_multi}
  \varphi(\omega) \varphi(\tau) \leq c \varphi(\omega \kappa \tau).
\end{equation}
We also define $K = \max\{ |\omega| : \omega \in \Gamma \} + 1$. A sub-multiplicative $\varphi$ potential is said to be \emph{almost-multiplicative} if there exists a constant $c>0$ such that
\begin{equation*}
  \varphi(\omega) \varphi(\tau) \leq c \varphi(\omega \tau).
\end{equation*}
for all $\omega, \tau \in \Sigma_*$. We note that quasi-multiplicativity is significantly less restrictive than the conditon of almost-multiplicativity which also appears in the literature; see e.g.\ Iommi and Yayama \cite{IommiYayama2012}.

If $\varphi$ is a sub-multiplicative potential, then we define the \emph{pressure} $P(\varphi)$ by setting
\begin{equation*}
  P(\varphi) = \lim_{n \rightarrow \infty} \tfrac{1}{n} \log Z_n(\varphi) = \inf_{n \in \N} \tfrac{1}{n} \log Z_n(\varphi),
\end{equation*}
where $Z_n(\varphi) = \sum_{\omega \in \Sigma_n}\varphi(\omega)$ for all $n \in \N$. Note that by the sub-multiplicativity, the pressure is well-defined, although it may not be finite. It is immediate that $P(\varphi)=\infty$ if and only if $Z_n(\varphi)=\infty$ for all $n \in \N$. Thus, if the shift space is finitely generated, then $P(\varphi)<\infty$. Observe that even if the shift space is finitely generated, the pressure can be negative infinity. Let $\psi \colon \Sigma_* \to (0,\infty)$ be a sub-multiplicative potential so that $P(\psi) < \infty$ and $Z_{n+m}(\psi) \ge cZ_n(\psi)Z_m(\psi)$ for some constant $c>0$. If the shift space is finitely generated, then the potential $\psi \equiv 1$ satisfies these assumptions. Now defining $\varphi \colon \Sigma_* \to (0,\infty)$ by setting $\varphi(\omega) = (cZ_n(\psi)n!)^{-1} \psi(\omega)$ for all $\omega \in \Sigma_*$, it is easy to see that $\varphi$ is sub-multiplicative with $P(\varphi) = -\lim_{n \to \infty} \tfrac{1}{n} \log n! = -\infty$.

We let $\MM_\sigma(\Sigma)$ denote the set of all $\sigma$-invariant Borel probability measures on $\Sigma$. Given $\mu \in \M_{\sigma}(\Sigma)$ along with a sub-multiplicative potential $\varphi$, we define the \emph{measure-theoretical pressure} $P_\mu(\varphi)$ by setting
\begin{equation} \label{eq:def_measure_pressure}
  P_\mu(\varphi) = \inf_{n \in \N} \tfrac{1}{n} \sum_{\omega \in \Sigma_n} \mu([\omega]) \log \frac{\varphi(\omega)}{\mu([\omega])}.
\end{equation}
We adopt the usual convention according to which $0\log(x/0) = 0\log 0 = 0$ for all $x > 0$.

\begin{lemma} \label{measure_theoretical_pressure}
  If $\varphi$ is a sub-multiplicative potential and $\mu \in \MM_\sigma(\Sigma)$, then
  \begin{equation*}
    P_\mu(\varphi) = \lim_{n \to \infty} \tfrac{1}{n} \sum_{\omega \in \Sigma_n} \mu([\omega]) \log \frac{\varphi(\omega)}{\mu([\omega])}.
  \end{equation*}
\end{lemma}

\begin{proof}
  The proof follows from the standard theory of sub-additive sequences
  by the sub-multi\-plicativity of $\varphi$, the concavity of the function $H(x)=-x\log x$, and the invariance of $\mu$.
\end{proof}

Furthermore, we define the \emph{Lyapunov exponent} for $\varphi$ and the \emph{entropy} of $\mu$ by setting
\begin{equation} \label{eq:def_entropy_lyapunov}
\begin{split}
  \Lambda_\mu(\varphi) &= \lim_{n \rightarrow \infty} \tfrac{1}{n}\sum_{\omega \in \Sigma_n} \mu([\omega]) \log \varphi(\omega) = \inf_{n \in \N} \tfrac{1}{n}\sum_{\omega \in \Sigma_n} \mu([\omega]) \log \varphi(\omega) \le \log\|\varphi\|, \\
  h_\mu &= \lim_{n \rightarrow \infty} \tfrac{1}{n} \sum_{\omega \in \Sigma_n} - \mu([\omega]) \log \mu([\omega]) = \inf_{n \in \N} \tfrac{1}{n} \sum_{\omega \in \Sigma_n} - \mu([\omega]) \log \mu([\omega]) \ge 0,
\end{split}
\end{equation}
respectively. Similarly as in the proof of Lemma \ref{measure_theoretical_pressure}, we see that the Lyapunov exponent and the entropy are well-defined by the sub-multiplicativity of $\varphi$ and the invariance of $\mu$.

\begin{lemma}\label{variational principal upper bound}
If $\varphi$ is a sub-multiplicative potential, then
\begin{equation*}
  P(\varphi) \ge P_\mu(\varphi)
\end{equation*}
for all $\mu \in \M_{\sigma}(\Sigma)$. Furthermore, if $h_\mu < \infty$ or $\Lambda_\mu(\varphi)$ is finite, then $P_\mu(\varphi) = h_\mu + \Lambda_\mu(\varphi)$.
\end{lemma}

\begin{proof}
  To show the first claim, we may assume that $P_\mu(\varphi)>-\infty$ and $P(\varphi)<\infty$. Thus $\sum_{\omega \in \Sigma_n} \mu([\omega]) \log\varphi(\omega)/\mu([\omega]) > -\infty$ for all $n \in \N$ and there is $n_0 \in \N$ so that $Z_n(\varphi) < \infty$ for all $n \ge n_0$. For each $n \ge n_0$ and $C_n \subset \Sigma_n$ we use the concavity of the function $H(x)=-x\log x$ to obtain
  \begin{equation} \label{eq:jensen_eq_calc}
  \begin{split}
    \sum_{\omega \in C_n} \mu([\omega]) \biggl( \log \frac{\varphi(\omega)}{\mu([\omega])} - \log \sum_{\omega \in C_n} \varphi(\omega) \biggr)
    &= \sum_{\omega \in C_n} \beta(\omega) H\bigl( \mu([\omega])/\beta(\omega) \bigr) \\ &\le H\biggl( \sum_{\omega \in C_n} \beta(\omega) \mu([\omega])/\beta(\omega) \biggr) \in [0,\tfrac{1}{e}],
  \end{split}
  \end{equation}
  where $\beta(\omega) = \varphi(\omega)/\sum_{\omega \in C_n}\varphi(\omega)$. Dividing by $n$ before letting $n \rightarrow \infty$ proves the first claim.

  To show the second claim, we first assume that $\Lambda_\mu(\varphi)$ is finite. Notice first that if $h_\mu < \infty$, then also $P_\mu(\varphi) = h_\mu + \Lambda_\mu(\varphi)$ is finite. On the other hand, if $P_\mu(\varphi) < \infty$, then there is $n_0 \in \N$ so that
  \begin{equation*}
    -\infty < \Lambda_\mu(\varphi) \le \tfrac{1}{n}\sum_{\omega \in \Sigma_n} \mu([\omega]) \log\varphi(\omega) \text{ and } \tfrac{1}{n}\sum_{\omega \in \Sigma_n} \mu([\omega]) \log\frac{\varphi(\omega)}{\mu([\omega])} \le P_\mu(\varphi) + 1 < \infty
  \end{equation*}
  for all $n \ge n_0$. Thus
  \begin{equation*}
    \tfrac{1}{n}\sum_{\omega \in \Sigma_n} -\mu([\omega]) \log\mu([\omega]) \le P_\mu(\varphi) -\Lambda_\mu(\varphi) + 1
  \end{equation*}
  for all $n \ge n_0$ and $h_\mu < \infty$. Therefore, if $h_\mu = \infty$, then $P_\mu(\varphi) = \infty$ and the desired equality holds.

  Finally, we notice that the proof of the second claim in the case $h_\mu < \infty$ is similar.
\end{proof}

Our first main result is the following variational principle. The proof of the result can be found in the end of \S \ref{sec:variational}.

\begin{maintheorem} \label{thm:main_gibbs}
  If $\varphi$ is a quasi-multiplicative potential, then
  \begin{equation*}
    P(\varphi) = \sup\{ P_\mu(\varphi) : \mu \in \MM_\sigma(\Sigma) \}.
  \end{equation*}
  Moreover, if $P(\varphi) < \infty$, then there exists a unique invariant measure $\mu$ for which $P(\varphi) = P_\mu(\varphi)$.
\end{maintheorem}

If the shift space is finitely generated, then we always have $P_\mu(\varphi) = h_\mu + \Lambda_\mu(\varphi)$. Moreover, the variational principle holds for all sub-multiplicative potentials; see K\"aenm\"aki \cite[Theorem 2.6]{Kaenmaki2004} and Cao, Feng, and Huang \cite[Theorem 1.1]{CaoFengHuang2008}. Quasi-multiplicativity has been a crucial property in the study of Lyapunov exponents for products of matrices; see e.g.\ Feng and Lau \cite{FengLau2002}, Feng \cite{Feng2009}, and Feng and K\"aenm\"aki \cite{FengKaenmaki2011}. It has also been used in connection with finitely generated self-affine sets; see Feng \cite{Feng2011} and Falconer and Sloan \cite{FalconerSloan2009}. Finally, we remark that in the infinitely generated setting, Iommi and Yayama \cite[Theorem 3.1]{IommiYayama2012} have recently verified the variational principle for almost-multiplicative potentials.

\subsection{Infinitely generated self-affine sets} \label{sec:preli_s-a}

Let $(T_i)_{i \in \N} \in \GL^\N$ be such that $\sup_{i \in \N} \| T_i \| < 1$. Define $\mathbf{A} = ([0,1]^d)^\N$ and note that by the Kolmogorov extension theorem $\mathbf{A}$ supports a natural probability measure $\LL_\mathbf{A} = (\LL^d|_{[0,1]^d})^\N$.
To each sequence $\mathbf{a} = (a_i)_{i \in \N} \in \mathbf{A}$ we associate a \emph{projection}
$\pi_{\mathbf{a}} \colon \Sigma \rightarrow \R^d$ defined by
\[
  \pi_{\mathbf{a}}(\omega) = \sum_{j=1}^\infty T_{\omega|_{j-1}} a_j.
\]
Here $T_\omega = T_{\omega_1} \cdots T_{\omega_n}$ for all $\omega = \omega_1 \cdots \omega_n \in \Sigma_n$ and $n \in \N$.
The set $F=F_{\mathbf{a}}=\pi_{\mathbf{a}}\left(\Sigma\right)$ is termed \emph{self-affine}.

The dimension theory of self-affine sets of this form was first investigated in the finitely generated setting by Falconer \cite{Falconer1988}. A central tool in Falconer's analysis was the singular value function $\varphi^s$. Given a matrix $T \in \GL$ we let $1 > \gamma_1(T) \geq \cdots \geq \gamma_d(T) > 0$ denote the singular values of $T$ (the square roots of the eigenvalues of $T^*T$), in non-increasing order of magnitude. Thus $\gamma_1(T) = \| T \|$ and $\gamma_d(T) = \| T^{-1} \|^{-1}$. If $0 \le s = m+\delta \le d$ with $m \in \Z$ and $0 < \delta \le 1$, then we define the \emph{singular value function} to be
\begin{equation*}
  \varphi^s(T) = \gamma_1(T) \cdots \gamma_m(T) \gamma_{m+1}(T)^\delta.
\end{equation*}
When $s \ge d$, we set $\varphi^s(t) = |\det(T)|^{s/d}$ for completeness. Given $(T_i)_{i \in \N} \in \GL^\N$ the singular value function introduces a potential by setting
\begin{equation} \label{eq:singular_value_function}
  \varphi^s(\omega) = \varphi^s(T_\omega)
\end{equation}
for all $\omega \in \Sigma_*$. Note that $\|\varphi^s\|\le 1$ for all $0\le s\le d$. Singular values $\gamma_i$ introduce potentials in a similar way. For example, if $s \ge 0$, then $\gamma_1^s$ is the sub-multiplicative potential $\omega \mapsto \| T_\omega \|^s$.

Falconer \cite[Lemma 2.1]{Falconer1988} showed that the singular value function is $\varphi^s$ is sub-multiplicative. It follows that the corresponding sub-multiplicative pressure $P(\varphi^s)$ is well-defined. Following the proof of \cite[Lemma 2.1]{KaenmakiVilppolainen2010}, we see that the function $s \mapsto P(\varphi^s)$ is strictly decreasing and thus finite on an interval $I$ of $[0,\infty)$. Furthermore, it is convex on connected components of $I \setminus \{ 1,\ldots,d \}$. Note that also the functions $s \mapsto P_\mu(\varphi^s)$ and $s \mapsto \Lambda_\mu(\varphi^s)$ are strictly decreasing and continuous for all $\mu \in \MM_\sigma(\Sigma)$.

Falconer \cite[Theorem 5.3]{Falconer1988} proved that given finitely many affine contractions with contraction ratios at most $\tfrac13$ the Hausdorff dimension of the corresponding self-affine set $F_{\mathbf{a}}$ is given by the unique zero of $s \mapsto P(\varphi^s)$ for almost every translation vector $\mathbf{a}$. Later Solomyak extended Falconer's proof to self-affine sets with the contraction ratios  up to $\tfrac12$; see \cite[Proposition 3.1(i)]{Solomyak1998}. See K\"aenm\"aki \cite[Theorem 4.5]{Kaenmaki2004} and Jordan, Pollicott and Simon \cite[Theorem 1.7]{JordanPollicottSimon2007} for corresponding results for measures. In order to extend Falconer's result to infinitely generated self-affine sets we have to assume that the singular value function is quasi-multiplicative. Let us next analyse the generality of this assumption.

\begin{proposition} \label{thm:sing_quasi}
  The singular value function $\varphi^s$ is quasi-multiplicative for all $0 \le s \le d$ if $(T_i)_{i \in \N} \in \GL^\N$ satisfies one of the following conditions:

  (1) Suppose that $d=2$ and for every line $\ell \in \R^2$ there is $i \in \N$ with $T_i(\ell) \ne \ell$.

  (2) Suppose that $d=2$ and the matrices $T_i$ have strictly positive entries so that the ratio of the smallest and largest entry of $T_i$ is uniformly bounded away from zero for all $i \in \N$.

  (3) Suppose that $d \in \N$ and $T_i = \diag(t^i_1,\ldots,t^i_d)$, where $1 > |t^i_1| > \cdots > |t^i_d| > 0$ for all $i \in \N$.
\end{proposition}

\begin{proof}
  Assuming (1), \cite[Proposition 2.8]{Feng2009} shows that the potential $\gamma_1$ is quasi-multiplicative. Observe that the proof of \cite[Proposition 2.8]{Feng2009} applies verbatim in the infinite case. Similarly, assuming (2), \cite[Lemma 7.1]{IommiYayama2012} shows that $\gamma_1$ is quasi-multiplicative. The claim in both of these cases follows now by recalling that the determinant is the product of singular values. Finally, assuming (3), the quasi-multiplicativity of the singular value function is immediate.
\end{proof}

\begin{remark}
  (1) The assumption (1) in Proposition \ref{thm:sing_quasi} is equivalent to the property that the matrices do not have a common eigenvector. Thus, if the $2 \times 2$ matrices cannot simultaneously be presented (in some coordinate system) as upper triangular matrices, then the singular value function $\varphi^s$ is quasi-multiplicative for all $0 \le s \le d$.

  (2) The set of $(T_i)_{i \in \N} \in {\it GL_2}(\R)^\N$ satisfying the assumption (1) in Proposition \ref{thm:sing_quasi} is open and dense set under the product topology. Indeed the set of pairs $(T_1,T_2) \in {\it GL_2}(\R)^2$ for which there is no common eigenvector is easily seen to be an open and dense set of full Lebesgue measure.

  (3) Falconer and Sloan \cite{FalconerSloan2009} have introduced a certain condition under which the singular value function is quasi-multiplicative also in higher dimensions; see \cite[Corollary 2.3]{FalconerSloan2009}.
\end{remark}

\begin{example}
  Two strictly positive $2 \times 2$ matrices having a common eigenvector show that strict positivity does not imply irreducibility. Furthermore, if
  \begin{equation*}
    T_1 = \begin{pmatrix}
            10 & 0 \\
            0 & 1
          \end{pmatrix}, \quad
    T_2 = \begin{pmatrix}
            0 & -1 \\
            10 & 11
          \end{pmatrix},
  \end{equation*}
  then $(T_1,T_2)$ is irreducible, but it is easy to see that there is no coordinate system in which the matrices are simultaneously strictly positive.
\end{example}

In our second main theorem, we generalise Falconer's dimension result to infinitely generated self-affine sets. The proof of the result can be found in \S \ref{sec:dim}.

\begin{maintheorem} \label{thm:dim_result}
If $(T_i)_{i \in \N} \in \GL^\N$ satisfies $\sup_{i \in \N} \| T_i \| < \tfrac12$ and the singular value function $\varphi^s$ is quasi-multiplicative for all $0 \le s \le d$, then
\[
  \dimh(F_{\mathbf{a}}) = \min\{d, \inf\{ s: P(\varphi^s) \leq 0 \} \} = \sup\{ \dimh(\pi_\mathbf{a}(I^\N)) : I \subset \N \text{ is finite} \}
\]
for $\LL_\mathbf{A}$-almost all $\mathbf{a} \in \mathbf{A}$.
\end{maintheorem}

\subsection{Multifractal analysis of Birkhoff averages} \label{sec:preli_multi}

We shall consider Birkhoff averages of functions $\Phi \colon \Sigma \rightarrow \R^{\N}$. The vector space $\R^{\N}$ is endowed with the product topology, so a sequence $(\alpha(n))_{n \in \N}$ with $\alpha(n)=(\alpha_i(n))_{i \in \N} \in \R^{\N}$ converges to $\alpha =(\alpha_i)_{i \in \N}$ if $\lim_{n \rightarrow \infty}\alpha_i(n)=\alpha_i$ for each $i \in \N$. Given a function $\phi \colon \Sigma \rightarrow \R$ we define the \emph{variation} $\var_n\phi$ for each $n \in \N$ by
\[
  \var_n\phi = \sup \{ |\phi(\omega) -\phi(\tau)|: [\omega|_n] = [\tau|_n] \}.
\]
A function $\phi \colon \Sigma \rightarrow \R$ is said to have \emph{summable variations} if $\sum_{n = 1}^\infty \var_n\phi < \infty$.

We take a sequence $\Phi = (\phi_i)_{i \in \N}$ of functions $\phi_i \colon \Sigma \rightarrow \R$, each with summable variations, which we think of as a function from $\Sigma$ to $\R^{\N}$. In this case, we just say that $\Phi \colon \Sigma \to \R^\N$ has \emph{summable variations}. Moreover, if each $\phi_i$ is bounded, then we say that $\Phi$ is \emph{bounded}. We define the \emph{Birkhoff sum} for each $n \in \N$ by
\begin{equation*}
S_n\Phi = \sum_{j=0}^{n-1} \Phi \circ \sigma^j
\end{equation*}
and the \emph{Birkhoff average} by $A_n\Phi = n^{-1}S_n\Phi$. We define $S_n\phi$ and $A_n\phi$ similarly when $\phi \colon \Sigma \to \R^k$ for some $k \in \N$. We let the \emph{symbolic level set} to be
\begin{equation*}
  E_\Phi(\alpha) = \{ \omega \in \Sigma : \lim_{n \to \infty} A_n\Phi(\omega) = \alpha \}
\end{equation*}
for all $\alpha = (\alpha_i)_{i \in \N} \in \overline{\R}^\N$, where $\overline{\R} = \R \cup \{-\infty,+\infty\}$.

Suppose we have a self-affine set $F_{\mathbf{a}}$, that is, $(T_i)_{i \in \N} \in \GL^\N$ is such that $\sup_{i \in \N} \| T_i \| < \tfrac12$, $\mathbf{a} = (a_i)_{i \in \N} \in A$, and $F_{\mathbf{a}} = \pi_{\mathbf{a}}(\Sigma)$ is the projection of the shift space. Let us denote the affine maps $x \mapsto T_ix + a_i$ by $f_i$. If the sequence $\mathbf{a}$ is such that there is a compact set $X \subset \R^d$ so that $f_i(X) \subset X$ for all $i \in \N$ and $f_i(X) \cap f_j(X) = \emptyset$ for $i \ne j$, then the projection $\pi_{\mathbf{a}}$ gives a conjugacy between the left shift $\sigma \colon \Sigma \to \Sigma$ and the well-defined map $g \colon \bigcup_{i \in \N} f_i(X) \to X$ for which $g(x) = f_i^{-1}(x) = T_i^{-1}x - a_i$ for $x \in f_i(X)$. Thus, in view of \eqref{eq:central_birkhoff}, this leads us to consider the projections of symbolic level sets,
\begin{equation*}
  J_\Phi^{\mathbf{a}}(\alpha) = \pi_{\mathbf{a}}(E_\Phi(\alpha)),
\end{equation*}
for as many $\mathbf{a} \in A$ as possible.

Next we state our main results concerning multifractal formalism in this paper. For each $k \in \N$ we let $\M_{\sigma^k}(\Sigma)$ denote the set of all $\sigma^k$-invariant Borel probability measures and define $\M^*_{\sigma^k}(\Sigma)$ to be the collection of all measures $\mu \in \MM_{\sigma^k}(\Sigma)$ which are compactly supported. If $k\in \N$ and $\mu \in \M^*_{\sigma^k}(\Sigma)$, then we let $D_k(\mu)$ to be the unique $s \geq 0$ satisfying
\[
\sum_{\omega \in \Sigma_k}\mu([\omega])\log\frac{\varphi^s(\omega)}{\mu([\omega])} = 0.
\]
The potential $\varphi^s$ here is the singular value function defined in \eqref{eq:singular_value_function}. We also set
\[
  D(\mu) = \inf\{ s: P_\mu(\varphi^s) \leq 0 \text{ and } \Lambda_\mu(\varphi^s) > -\infty \}
\]
for all $\mu \in \M_{\sigma}(\Sigma)$ and call it a \emph{Lyapunov dimension} of $\mu$. Given $\alpha \in \overline{\R}$ we define an indexed family of neighbourhoods by
\begin{equation*}
B_n(\alpha) = \begin{cases} \left(-\infty,-n\right), &\text{if } \alpha = - \infty, \\
\left( \alpha-\frac{1}{n},\alpha+\frac{1}{n}\right), &\text{if } \alpha \in \R, \\
\left(-\infty,-n\right), &\text{if } \alpha = \infty. \\
\end{cases}
\end{equation*}

We have two main results concerning multifractal analysis of Birkhoff averages. In the first one, we consider general potentials and in the second one, we restrict our analysis to bounded potentials.

\begin{maintheorem}\label{General theorem}
If $(T_i)_{i \in \N} \in \GL^\N$ is such that $\sup_{i \in \N} \| T_i \| < \tfrac12$, the singular value function $\varphi^s$ is quasi-multiplicative for all $0 \le s \le d$, $\Phi \colon \Sigma \to \R^{\N}$ has summable variations, and $\alpha \in \overline{\R}^\N$, then
\begin{align*}
  \dimh(J^{\mathbf{a}}_{\Phi}(\alpha)) = \min\bigl\{ d,\lim_{n \rightarrow \infty} \lim_{k \rightarrow \infty} \sup \bigl\{ D_k(\mu) : \;&\mu \in \M^*_{\sigma^k}(\Sigma) \text{ so that } \\ &\int A_k(\phi_i)d\mu \in B_n(\alpha_i) \text{ for all } i \in \{ 1,\ldots,n \}\bigr\} \bigr\}
\end{align*}
for $\LL_\mathbf{A}$-almost all $\mathbf{a} \in \mathbf{A}$.
\end{maintheorem}

The proof of Theorem \ref{General theorem} is given in \S \ref{sec:multifractal}. Theorem \ref{main theorem for bounded potentials}, our second result on multifractal formalism, generalises the theorem of Fan, Jordan, Liao and Rams \cite[Theorem 1.2]{FanJordanLiaoRams2011} to the self-affine setting. Define $s_{\infty} = \inf\left\lbrace s: P(\varphi^s) < \infty \right\rbrace$ and $\PP(\Phi) = \{ \int \Phi d\mu: \mu \in \M_{\sigma}(\Sigma) \}$, and let $\overline{\PP(\Phi)}$ be the closure of $\PP(\Phi)$ with respect to the pointwise topology.

\begin{maintheorem}\label{main theorem for bounded potentials}
If $(T_i)_{i \in \N} \in \GL^\N$ is such that $\sup_{i \in \N} \| T_i \| < \tfrac12$, the singular value function $\varphi^s$ is quasi-multiplicative for all $0 \le s \le d$, $\Phi \colon \Sigma \to \R^{\N}$ is bounded with summable variations, and $\alpha \in \overline{\PP(\Phi)}$, then
\begin{equation*}
  \dimh(J^\mathbf{a}_\Phi(\alpha)) = \min\bigl\{ d,\max\bigl\{ s_\infty, \sup\{ D(\mu) : \mu \in \MM_\sigma(\Sigma) \text{ so that } \int \Phi d\mu = \alpha \} \bigr\} \bigr\}
\end{equation*}
for $\LL_\mathbf{A}$-almost all $\mathbf{a} \in \mathbf{A}$. Furthermore, if $\alpha \notin \overline{\PP(\Phi)}$, then $J^\mathbf{a}_\Phi(\alpha)=\emptyset$ for all $\mathbf{a} \in \mathbf{A}$.
\end{maintheorem}

The proof of Theorem \ref{main theorem for bounded potentials} is presented in \S \ref{sec:conditional}.

\section{Thermodynamic formalism for quasi-multiplicative potentials} \label{sec:gibbs}

\subsection{Existence of Gibbs measures}
Suppose we have a sub-multiplicative potential $\varphi$ along with a subset $I \subset \N$. We define the \emph{pressure} $P(\varphi,I)$ by
\begin{equation*}
  P(\varphi,I) = \lim_{n \rightarrow \infty} \tfrac{1}{n} \log Z_n(\varphi, I) = \inf_{n \in \N} \tfrac{1}{n} \log Z_n(\varphi, I),
\end{equation*}
where $Z_n(\varphi, I) = \sum_{\omega \in I^n}\varphi(\omega)$ for all $n \in \N$. Thus $Z_n(\varphi, \N) = Z_n(\varphi)$ and $P(\varphi, \N) = P(\varphi)$. Observe that $Z_n(\varphi,I) \le Z_n(\varphi,I)$ and hence also $P(\varphi,J) \le P(\varphi,I)$ for all $J \subset I \subset \N$. If $C\ge 1$, then an invariant probability measure $\mu \in \M_{\sigma}(\Sigma)$ is said to be a \emph{$C$-Gibbs measure for the potential $\varphi$ on $I$} if it is supported on $I^\N$, the pressure $P(\varphi,I)$ is finite, and
\begin{equation*}
C^{-1}  \leq \frac{\mu([\omega]) }{\varphi(\omega)\exp(-n P(\varphi,I))} \leq C
\end{equation*}
for all $\omega \in I^n$ and $n \in \N$.
An invariant measure $\mu \in \M_{\sigma}(\Sigma)$ is said to be a \emph{Gibbs measure for the potential $\varphi$ on $I$} if there exists some $C\ge 1$ such that $\mu$ is a $C$-Gibbs measure for the potential $\varphi$ on $I$. Finally, $\mu \in \M_{\sigma}(\Sigma)$ is said to be a \emph{Gibbs measure for the potential $\varphi$} if $\mu$ is a Gibbs measure for the potential $\varphi$ on $\N$.

In this section, our main goal is to show that if $\varphi$ is a quasi-multiplicative potential with finite pressure, then $\varphi$ has a Gibbs measure. We remark that this is not the case for all sub-multiplicative potentials; see \cite[Example 6.4]{KaenmakiVilppolainen2010} for a counter-example in a finitely generated self-affine set. For a given quasi-multiplicative potential, throughout the section, we let $\Gamma \subset \Sigma_*$, $K \in \N$, and $c \ge 1$ be as in the definition of the quasi-multiplicative potential; see \eqref{eq:quasi_multi}.

\begin{lemma}\label{positive recurrence}
If $\varphi$ is a quasi-multiplicative potential and $I \subset \N$ is so that $\Gamma \subset \bigcup_{k = 1}^K I^k$, then
\begin{equation*}
  e^{n P(\varphi,I)} \leq Z_n(\varphi, I) \leq cK \max\{ 1 , e^{K P(\varphi)} \}  e^{n P(\varphi,I)}.
\end{equation*}
for all $n \in \N$. In particular, $P(\varphi,I) > -\infty$.
\end{lemma}

\begin{proof}
Since the left-hand side inequality follows immediately from the definition of the pressure, it suffices to show the right-hand side inequality. Fix $n,m \in \N$ and $\omega_i \in I^n$ for all $i \in \{ 1,\ldots,m \}$. By the quasi-multiplicativity, there are $\kappa_1,\ldots,\kappa_{m-1}$ so that
\begin{equation*}
  \varphi(\omega_1) \cdots \varphi(\omega_m) \leq c^{m-1} \varphi(\omega_1 \kappa_1 \omega_2 \kappa_2 \cdots \omega_{m-1} \kappa_{m-1} \omega_m).
\end{equation*}
Denoting $\xi_{n,m}(\omega_1\cdots \omega_m) = \omega_1 \kappa_1 \omega_2 \kappa_2 \cdots \omega_{m-1} \kappa_{m-1} \omega_m$ for all $\omega = \omega_1 \cdots \omega_m \in (I^{n})^m$ defines a mapping $\xi_{n,m} \colon (I^n)^m \rightarrow \bigcup_{\ell=1}^{K(m-1)} I^{nm+l}$ which is at most $K^{m-1}$ to one. Hence
\begin{align*}
  Z_n(\varphi, I)^m &= \biggl(\sum_{\omega \in I^n} \varphi(\omega)\biggr)^m
  = \sum_{\omega \in (I^{n})^m} \prod_{i=1}^m \varphi(\omega_i)
  \leq c^{m-1} \sum_{\omega \in (I^n)^m} \varphi(\xi_{n,m}(\omega)) \\
  &\leq (cK)^{m-1} \sum_{\ell=1}^{K(m-1)} \sum_{\omega \in I^{nm+l}} \varphi(\omega)
  = (cK)^{m-1}\sum_{\ell=1}^{K(m-1)} Z_{nm+l}(\varphi, I).
\end{align*}
Consequently, for each $m \in \N$ there is $\ell_m \in \N$ with $nm \leq \ell_m \leq (n+K)m$ satisfying $Z_n(\varphi, I)^m \leq m(cK)^{m} Z_{\ell_m}(\varphi, I)$.
Hence
\begin{equation*}
  Z_n(\varphi, I) \leq m^{1/m}cK \bigl (Z_{\ell_m}(\varphi, I)^{1/\ell_m} \bigr)^{\ell_m/m} \to
  \begin{cases}
    cK e^{(n+K) P(\varphi,I)}, &\text{if } P(\varphi,I) > 0, \\
    cK e^{n P(\varphi,I)}, &\text{if } P(\varphi,I) \leq 0,
  \end{cases}
\end{equation*}
by letting $m \rightarrow \infty$. The proof follows since $P(\varphi,I) \le P(\varphi)$.
\end{proof}

The following proposition is a finite approximation property for the pressure. It is a crucial property in our analysis since it makes it possible to construct a Gibbs measure on an infinitely generated shift space via its finitely generated sub-spaces.

\begin{prop}\label{finite approximation property}
If $(I_\ell)_{\ell \in \N}$ is a sequence of non-empty finite sets $I_\ell \subset \N$ with $I_\ell \subset I_{\ell+1}$ for all $\ell \in \N$ so that $\N = \bigcup_{\ell \in \N} I_\ell$, then
\begin{equation*}
  P(\varphi) = \lim_{\ell \to \infty} P(\varphi,I_\ell)
\end{equation*}
for all quasi-multiplicative potentials $\varphi$. In particular, $P(\varphi) = \sup \{ P(\varphi,I) : I \subset \N \text{ is finite} \}$.
\end{prop}

\begin{proof}
Recall that $P(\varphi,I_\ell) \le P(\varphi,I_{\ell+1}) \le P(\varphi)$ for all $\ell \in \N$. Fix $\varrho < P(\varphi)$, $n \in \N$, and let $P = \lim_{\ell \to \infty} P(\varphi,I_\ell)$. Since $\varrho < \tfrac{1}{n} \log Z_n(\varphi)$, we may choose $\ell \in \N$ so that $\Gamma \subset \bigcup_{k=1}^K I_\ell^k$ and $\varrho < \tfrac{1}{n} \log Z_n(\varphi, I_\ell)$. By Lemma \ref{positive recurrence}, we have $Z_n(\varphi, I_\ell) \leq cK \max\{ 1 , e^{K P(\varphi)} \}  e^{n P}$ and thus $\varrho < \tfrac{1}{n}\bigl( \log cK + K|P(\varphi)| \bigr) + P$. The proof is finished by letting $n \rightarrow \infty$.
\end{proof}

\begin{lemma}\label{key QM lemma}
If $\varphi$ is a quasi-multiplicative potential with $P(\varphi)<\infty$, then there exists a constant $C \ge 1$ such that for each $I \subset \N$ with $\Gamma \subset \bigcup_{k=1}^K I^k$ we have
\begin{equation*}
  C^{-1} e^{(n+m) P(\varphi,I)} \varphi(\omega) \leq \sum_{\kappa \in I^n} \sum_{\tau \in I^m} \varphi(\kappa \omega \tau) \leq C e^{(n+m) P(\varphi,I)} \varphi(\omega)
\end{equation*}
for all $m,n \in \N$ and $\omega \in \bigcup_{n \in \N} I^n$.
\end{lemma}

\begin{proof}
The right-hand side inequality follows immediately since
\begin{align*}
\sum_{\kappa \in I^n} \sum_{\tau \in I^m} \varphi(\kappa \omega \tau) &\leq \sum_{\kappa \in I^n} \sum_{\tau \in I^m} \varphi(\kappa)\varphi(\omega) \varphi(\tau) = \varphi(\omega) Z_n(\varphi,I) Z_m(\varphi,I) \\
&\leq \bigl( cK \max\{ 1 , e^{K P(\varphi)} \} \bigr)^2 e^{(n+m) P(\varphi,I)}  \varphi(\omega)
\end{align*}
by Lemma \ref{positive recurrence}.

To show the left-hand side inequality, we first notice that the quasi-multiplicativity implies
\begin{equation}\label{aggregation property}
\varphi(\omega) \varphi(\kappa) \leq c \sum_{k=1}^K \sum_{\alpha \in I^k} \varphi(\omega \alpha \kappa)
\end{equation}
for all $\omega, \kappa \in \Sigma_*$. Applying Lemma \ref{positive recurrence}, along with (\ref{aggregation property}), we obtain
\begin{align*}
e^{(n+m) P(\varphi,I)} \varphi(\omega) &\leq Z_n(\varphi,I) \sum_{\kappa \in I^m} \varphi(\omega) \varphi(\kappa)
\leq c Z_n(\varphi,I) \sum_{\tau \in I^m} \sum_{k=1}^K \sum_{\alpha \in I^k}  \varphi(\omega \alpha \tau)\\
&= c Z_n(\varphi,I) \sum_{k=1}^K \sum_{\alpha \in I^k} \sum_{\tau \in I^{m}} \varphi(\omega \tau \alpha)
\leq c Z_n(\varphi,I) \sum_{k=1}^K Z_k(\varphi,I) \sum_{\tau \in I^{m}} \varphi(\omega \tau)\\
&\leq c^2 \sum_{k=1}^K Z_k(\varphi,I) \sum_{k=1}^K \sum_{\alpha \in I^k} \sum_{\kappa \in I^n} \sum_{\tau \in I^{m}} \varphi(\kappa \alpha \omega \tau)\\
&\leq c^2 \biggl( \sum_{k=1}^K Z_k(\varphi,I)\biggr)^2 \sum_{\kappa \in I^n} \sum_{\tau \in I^{m}} \varphi(\kappa \omega \tau).
\end{align*}
The proof is now finished since
\begin{equation*}
\sum_{k=1}^K Z_k(\varphi,I) \leq cK \max\{ 1 , e^{K P(\varphi)} \}  \sum_{k=1}^K e^{k P(\varphi)} < \infty
\end{equation*}
by Lemma \ref{positive recurrence}.
\end{proof}

We are now ready to show that every finite sub-space carries a Gibbs measure. Observe that, to be able to extend the result into infinitely generated shift space, it is crucial to find a uniform constant.

\begin{prop}\label{finite Gibbs}
If $\varphi$ is a quasi-multiplicative potential with $P(\varphi)<\infty$, then there is $C\ge 1$ so that $\varphi$ has a $C$-Gibbs measure for $\varphi$ on $I$ for all finite subsets $I \subset \N$ with $\Gamma \subset \bigcup_{k=1}^K I^k$.
\end{prop}

\begin{proof}
Let $I \subset \N$ be a finite subset with $\Gamma \subset \bigcup_{k=1}^K I^k$. Given a finite word $\omega \in \bigcup_{n \in \N} I^n$ we choose $\tilde{\omega} \in  [\omega]\cap I^{\N}$ and let $\delta_{\omega}$ denote the point mass concentrated at $\tilde{\omega}$. For each $n\in\N$ we define a probability measure $\nu_n$ on $\Sigma$ by
\begin{equation*}
\nu_n = Z_{3n}(\varphi,I)^{-1} \sum_{\omega \in I^{3n}} \varphi(\omega) \delta_{\omega}.
\end{equation*}
Note that $\nu_n$ is supported on $I^{\N}$.
If $m,\ell \in \{ 1,\ldots,n \}$ and $\omega \in I^m$, then
\begin{equation*}
\nu_n \circ \sigma^{-\ell}([\omega]) = \sum_{\kappa \in I^\ell}\sum_{\tau \in I^{3n-\ell-m}} \nu_n([\kappa \omega \tau])
= Z_{3n}(\varphi,I)^{-1} \sum_{\kappa \in I^\ell}\sum_{\tau \in I^{3n-\ell-m}} \varphi(\kappa \omega \tau).
\end{equation*}
According to Lemmas \ref{positive recurrence} and \ref{key QM lemma} there exists a constant $C \ge 1$ so that
\begin{equation}\label{mu Gibbs}
C^{-1}e^{-m P(\varphi,I)}\varphi([\omega]) \leq \nu_n \circ \sigma^{-\ell}([\omega]) \leq C e^{-m P(\varphi,I)}\varphi(\omega)
\end{equation}
for all finite subsets $I \subset \N$ with $\Gamma \subset \bigcup_{k=1}^K I^k$. Observe that the above estimate remains true if we replace $\nu_n \circ \sigma^{-\ell}$ by the probability measure
\begin{eqnarray}\label{almost invariance}
\mu_n = \tfrac{1}{n} \sum_{\ell=1}^{n}\nu_n \circ \sigma^{-\ell}.
\end{eqnarray}
Since $I$ is finite and each $\mu_n$ is supported on the compact set $I^{\N}$, there is a convergent subsequence $(\mu_{n_k})_{k \in \N}$ converging to some limit $\mu$ in the weak$^*$ topology. It follows from \eqref{almost invariance} that $\mu$ is a $\sigma$-invariant probability measure. Moreover, by \eqref{mu Gibbs}, $\mu$ is a $C$-Gibbs measure for $\varphi$ on $I$.
\end{proof}

\begin{theorem} \label{thm:gibbs_exists}
If $\varphi$ is a quasi-multiplicative potential with $P(\varphi)<\infty$, then $\varphi$ has a Gibbs measure $\mu$. Moreover, there is $C \ge 1$ so that for each $\ell \in \N$ there are a finite set $I_\ell \subset \N$ and a $C$-Gibbs measure $\mu_\ell$ for $\varphi$ on $I_\ell$ such that $P(\varphi,I_\ell) \to P(\varphi)$ and $\mu_\ell \to \mu$ in the weak$^*$ topology.
\end{theorem}

\begin{proof}
Let $(I_\ell)_{\ell \in \N}$ be a sequence of non-empty finite sets $I_\ell \subset \N$ with $I_\ell \subset I_{\ell+1}$ and $\Gamma \subset \bigcup_{k=1}^K I_\ell^k$ for all $\ell \in \N$ such that $\N = \bigcup_{\ell \in \N} I_\ell$. Recalling Proposition \ref{finite approximation property}, we have $\lim_{\ell \rightarrow \infty} P(\varphi,I_\ell) = P(\varphi)$. By Proposition \ref{finite Gibbs}, there exist a constant $C \ge 1$ and for each $\ell \in \N$ a $\sigma$-invariant probability measure $\mu_\ell \in \M_{\sigma}(\Sigma)$ so that
\begin{equation}\label{l Gibbs}
C^{-1} \leq \frac{\mu_\ell([\omega]) }{\varphi(\omega)\exp(-n P(\varphi,I_\ell))} \leq C
\end{equation}
for all $\omega \in I_\ell^n$ and $\ell \in \N$.
It suffices to show that the sequence $(\mu_\ell)_{\ell \in \N}$ is tight, that is, for each $\varepsilon>0$ there exists a compact set $K \subset \Sigma$ for which $\mu_\ell(K) > 1 - \varepsilon$ for all $\ell \in \N$. Then the sequence $(\mu_\ell)_{\ell \in \N}$ has a converging subsequence and it follows from \eqref{l Gibbs} that the limit measure of that subsequence is a Gibbs measure for $\varphi$.

Fix $\varepsilon>0$ and notice that $\sum_{i \in \N}\varphi(i) = Z_1(\varphi) \leq C e^{P(\varphi)}< \infty$ by Lemma \ref{positive recurrence}.
Thus, for each $k \in \N$ there is a finite subset $I_k \subset \N$ so that
\begin{equation*}
\sum_{i \in \N \setminus I_k}\varphi(i) < \varepsilon 2^{-k} C^{-1} e^{P(\varphi, I_1)} \le \varepsilon 2^{-k} C^{-1} e^{P(\varphi, I_\ell)}
\end{equation*}
for all $\ell \in \N$.
We define $K = \left\lbrace \omega \in \Sigma : \omega_k \in I_k \text{ for all } k \in \N \right\rbrace$. It follows from \eqref{l Gibbs} that
\begin{align*}
\mu_\ell(K) &= \mu_\ell\biggl(\Sigma \setminus \bigcup_{k \in \N} \{ \omega \in \Sigma : \omega_k \notin I_k \} \biggr)
= 1- \sum_{k \in \N} \mu_\ell(\{ \omega \in \Sigma : \omega_k \notin I_k \})\\
&= 1- \sum_{k \in \N} \sum_{i \in \N \setminus I_k} \mu_\ell\bigl(\sigma^{-k}([i])\bigr)
= 1- \sum_{k \in \N} \sum_{i \in \N \setminus I_k} \mu_\ell([i])\\
&\geq 1- \sum_{k \in \N} \sum_{i \in \N \setminus I_k} Ce^{-P(\varphi, I_\ell)}\varphi(i)
> 1- \sum_{k \in \N} 2^{-k}\varepsilon = 1-\eps.
\end{align*}
for all $\ell \in \N$.
\end{proof}

\subsection{Variational principle} \label{sec:variational}
We shall study the properties of the Gibbs measure found in Theorem \ref{thm:gibbs_exists}. At the end of this section, we prove Theorem \ref{thm:main_gibbs}.

\begin{theorem}\label{Gibbs implies ergodicity}
If $\varphi$ is a quasi-multiplicative potential with $P(\varphi)<\infty$ and $\mu$ is a Gibbs measure for $\varphi$, then $\mu$ is ergodic. In particular, $\mu$ is the only Gibbs measure for $\varphi$.
\end{theorem}

\begin{proof}
  By \eqref{aggregation property}, it is straightforward to see that a $C$-Gibbs measure $\mu$ satisfies
  \[
    \sum_{p,q = 1}^K \mu\bigl( [\omega] \cap \sigma^{-(n+p+q)}([\tau]) \bigr) \ge c^{-2}C^{-1}e^{-2K|P(\varphi)|} \mu([\omega]) \mu([\tau])
  \]
  for all $\omega,\tau \in \Sigma$ and $n \ge |\omega|$.
  The proof follows now by standard arguments; see e.g.\ \cite[Theorem 3.2]{FengLau2002}.
\end{proof}

\begin{lemma} \label{thm:gibbs_equilibrium}
If $\varphi$ is a quasi-multiplicative potential with $P(\varphi)<\infty$ and $\mu$ is the Gibbs measure for $\varphi$ on a set $I \subset \N$, then $P(\varphi,I) = P_\mu(\varphi)$.
\end{lemma}

\begin{proof}
By the definition of a Gibbs measure, we get
\begin{equation*}
  P_\mu(\varphi) = \lim_{n \to \infty} \tfrac{1}{n} \sum_{\omega \in I^n} \mu([\omega])\log\frac{\varphi(\omega)}{\mu([\omega])} = \lim_{n \to \infty} \tfrac{1}{n} \sum_{\omega \in I^n} \mu([\omega])\log e^{nP(\varphi,I)} = P(\varphi,I)
\end{equation*}
as desired.
\end{proof}

\begin{lemma} \label{thm:eq_abs_cont}
If $\varphi$ is a quasi-multiplicative potential with $P(\varphi)<\infty$ and $\mu$ is the Gibbs measure for $\varphi$, then any measure $\nu \in \M_{\sigma}(\Sigma)$ with $P(\varphi) \le P_\nu(\varphi)$ is absolutely continuous with respect to $\mu$.
\end{lemma}

\begin{proof}
To prove the claim, we follow the ideas of \cite{Bowen1975} and \cite[Theorem 3.6]{KaenmakiVilppolainen2010}. Let $\mu$ be a $C$-Gibbs measure. Assume to the contrary that there exist a measure $\nu \in \M_{\sigma}(\Sigma)$ with $P(\varphi) \le P_\mu(\varphi)$ and a Borel set $B \subset \Sigma$ so that $\mu(B)=0$ and $\nu(B)>0$. Since the semi-algebra of cylinder sets generates the Borel $\sigma$-algebra we may choose a sequence of sets $(B_n)_{n \in \N}$ such that each $B_n$ is a union of cylinders of length $n$ with $(\mu+\nu)(B_n \triangle B)\rightarrow 0$ as $n \rightarrow \infty$. Let $B_n' = \{ \omega \in \Sigma_n : [\omega] \subset B_n \}$. Hence, by \eqref{eq:def_measure_pressure} and \eqref{eq:jensen_eq_calc}, we have
\begin{equation} \label{eq:boring_calculation}
\begin{split}
  0 &\le \sum_{\omega \in B_n'} \nu([\omega]) \log\frac{\varphi(\omega)}{\nu([\omega])} + \sum_{\omega \in \Sigma \setminus B_n'} \nu([\omega]) \log\frac{\varphi(\omega)}{\nu([\omega])} - nP(\varphi) \\
  &\le \nu(B_n) \log\sum_{\omega \in \B_n'} \varphi(\omega) + \nu(\Sigma \setminus B_n) \log\sum_{\omega \in \Sigma \setminus \B_n'} \varphi(\omega) - nP(\varphi) + \tfrac{2}{e} \\
  &\le \nu(B_n)\log\mu(B_n) + \nu(\Sigma \setminus B_n)\log\mu(\Sigma \setminus B_n) + \log C + \tfrac{2}{e}
\end{split}
\end{equation}
for all $n$ large enough. Since $\nu(B_n) \to \nu(B)$ and $\mu(B_n) \to 0$ the right-hand side of \eqref{eq:boring_calculation} tends to $-\infty$ as $n \to \infty$. This contradiction finishes the proof.
\end{proof}

We are now ready to prove Theorem \ref{thm:main_gibbs}.

\begin{proof}[Proof of Theorem \ref{thm:main_gibbs}]
  Let us first assume that $P(\varphi)=\infty$. Let $(I_\ell)_{\ell \in \N}$ be a sequence of non-empty finite sets with $I_\ell \subset \N$ and $\Gamma \subset \bigcup_{k=1}^K I_\ell^k$ for all $\ell \in \N$ such that $\N = \bigcup_{\ell \in \N} I_\ell$. Recalling Proposition \ref{finite Gibbs}, let $\mu_\ell$ be a Gibbs measure for $\varphi$ on $I_\ell$ for all $\ell \in \N$. Now
  \begin{equation*}
    P(\varphi) = \sup\{ P(\varphi,I_\ell) : \ell \in \N \} = \sup\{ P_{\mu_\ell}(\varphi) : \ell \in \N \}
    \le \sup\{ P_{\mu}(\varphi) : \mu \in \MM_\sigma(\Sigma) \}
  \end{equation*}
  by Proposition \ref{finite approximation property} and Lemma \ref{thm:gibbs_equilibrium}.

  If $P(\varphi)<\infty$, then it suffices to prove that a Gibbs measure $\mu$ is the only invariant measure for which $P(\varphi) = P_\mu(\varphi)$. Theorem \ref{Gibbs implies ergodicity} shows that $\mu$ is ergodic and Lemma \ref{thm:gibbs_equilibrium} shows that it satisfies $P(\varphi) = P_\mu(\varphi)$. If $\nu \in \MM_\sigma(\Sigma)$ is an invariant measure satisfying $P(\varphi) = P_\mu(\varphi)$, then $\nu$ is absolutely continuous with respect to $\mu$ by Lemma \ref{thm:eq_abs_cont}. It follows from the proof of \cite[Theorem 6.10(iii)]{Walters1982} that $\nu = \mu$.
\end{proof}

\subsection{Differentiation of pressure}

Given a pair of potentials $\varphi_1, \varphi_2 \colon \Sigma_* \rightarrow [0,\infty)$ we let $\varphi_1 \cdot \varphi_2$ denote the potential defined by $\omega \mapsto \varphi_1(\omega) \varphi_2(\omega)$ for all $\omega \in \Sigma_*$. Given a function $\phi \colon \Sigma \rightarrow \R$ we define an associated potential $e_\phi \colon \Sigma_* \rightarrow [0,\infty)$ by setting
\begin{equation*}
  e_\phi(\omega) = \exp\bigl(\sup\{ S_n\phi(\tau): \tau \in [\omega] \} \bigr)
\end{equation*}
for all $\omega \in \Sigma_n$ and $n \in \N$. Recall that $S_n\phi(\tau) = \sum_{j=0}^{n-1}\phi(\sigma^j(\tau))$ for all $\tau \in \Sigma$.

\begin{lemma}\label{compound quasi-multiplicative}
If $\varphi$ is a quasi-multiplicative potential and $\phi \colon \Sigma \to \R$ has summable variations, then the potential $\varphi \cdot e_\phi$ is quasi-multiplicative.
\end{lemma}

\begin{proof}
If $c=\exp(\sum_{n=1}^\infty \var_n(\phi))$, then
\begin{equation*}
  c^{-1}e_\phi(\omega) e_\phi(\kappa) \le e_\phi(\omega \kappa) \le e_\phi(\omega) e_\phi(\kappa)
\end{equation*}
for all $\omega,\kappa \in \Sigma_*$. The claim follows from the quasi-multiplicativity of $\varphi$.
\end{proof}

\begin{lemma}\label{convexity of the pressure function}
If $\varphi$ is a sub-multiplicative potential with $P(\varphi)<\infty$ and $\phi \colon \Sigma \to \R$ is bounded with summable variations, then the function $q \mapsto P\left(\varphi \cdot e(q\phi)\right)$ is convex.
\end{lemma}

\begin{proof}
If $q, p \in \R$ and $0 \le \lambda \le 1$, then
\begin{equation*}
 \varphi(\omega) e_{(\lambda q+(1-\lambda)p)\phi}(\omega) \le \bigl( \varphi(\omega) e_{q\phi}(\omega) \bigr)^\lambda \bigl( \varphi(\omega) e_{p\phi}(\omega) \bigr)^{1-\lambda}
\end{equation*}
for all $\omega \in \Sigma_*$.
Thus, by H\"{o}lder's inequality, we have
\begin{equation*}
\sum_{\omega \in \Sigma_n} \varphi(\omega) e_{(\lambda q+(1-\lambda)p)\phi}(\omega) \le \biggl( \sum_{\omega \in \Sigma_n} \varphi(\omega) e_{q\phi}(\omega) \biggr)^{\lambda} \biggl( \sum_{\omega \in \Sigma_n} \varphi(\omega) e_{p\phi}(\omega) \biggr)^{1-\lambda}.
\end{equation*}
Taking logarithms, dividing by $n$, and letting $n \rightarrow \infty$ gives the claim.
\end{proof}

\begin{lemma}\label{differentiation of pressure}
If $\varphi$ is a quasi-multiplicative potential with $P(\varphi)<\infty$, $\mu$ is the Gibbs measure for $\varphi$, and $\phi \colon \Sigma \to \R$ is bounded with summable variations, then the function $q \mapsto P\left(\varphi \cdot e_{q\phi}\right)$ is differentiable at zero with derivative
\begin{equation*}
\frac{\partial P\left(\varphi \cdot e_{q\phi}\right)}{\partial q}\bigg|_{q=0} = \int \phi d\mu.
\end{equation*}
\end{lemma}

\begin{proof}
To prove the claim, we use some of the ideas used in the proof of \cite[Thorem 4.4]{KaenmakiVilppolainen2010}. It suffices to show that the right derivative exists at zero and equals to $\int \phi d\mu$ since applying this result with $-\phi$ in place of $\phi$ gives
\begin{equation*}
\lim_{q \uparrow 0}\tfrac{1}{q} \bigl(P(\varphi \cdot e_{q\phi}) - P(\varphi)\bigr) =
-\lim_{q \downarrow 0}\tfrac{1}{q} \bigl(P(\varphi \cdot e_{q(-\phi)}) - P(\varphi)\bigr) = \int \phi d\mu.
\end{equation*}
Throughout the proof of the lemma, to simplify notation, we write $P(q)$ in place of $P\left(\varphi \cdot e_{q\phi}\right)$. By Lemma \ref{convexity of the pressure function}, the function $q \mapsto P(q)$ is convex and hence there is a well-defined right derivative at zero. We shall denote it by $P'_+(0)$.

To prove that $P'_+(0) \leq \int \phi d\mu$, take $\beta > \int \phi d\mu$. Define
\[
  C_n' = \{ \omega \in \Sigma_n : S_n\phi(\tau) > n\beta \text{ for some } \tau \in [\omega] \}
\]
for all $n \in \N$ and let $C_n = \bigcup_{\omega \in C_n'}[\omega]$. Since $\mu$ is a Gibbs measure for $\varphi$ there is $C \ge 1$ so that
\begin{equation} \label{eq:gibbs_at_zero}
  \varphi(\omega) \le Ce^{nP(0)} \mu([\omega])
\end{equation}
for all $\omega \in \Sigma_n$ and $n \in \N$. By Theorem \ref{Gibbs implies ergodicity}, $\mu$ is ergodic and thus, we may apply Birkhoff's ergodic theorem, Egorov's theorem, and the fact that $\phi$ has summable variations, to obtain $\lim_{n \rightarrow \infty} \mu(C_n) = 0$.

Fix $\gamma >0$. Since $P(q)$ is convex we have $P(\gamma/n) \geq P(0)+\gamma/n P_+'(0)$. Using the sub-multiplicativity of $\varphi \cdot e_{\gamma/n \phi}$ and \eqref{eq:gibbs_at_zero}, we have
\begin{align*}
e^{nP(0)+ \gamma P_+'(0)} &\leq e^{nP(\gamma/n)} \leq \sum_{\omega \in \Sigma_n} \varphi(\omega) \exp(\gamma/n \| S_n(\phi)|_{[\omega]}\|) \\
&\leq \sum_{\omega \in \Sigma_n \setminus C_n'} \varphi(\omega) e^{\gamma \beta} + \sum_{\omega \in C_n'} \varphi(\omega) e^{\gamma\|\phi\|}\\
&\leq  Ce^{\gamma \beta}e^{nP(0)}(1-\mu(C_n)) +  Ce^{\gamma\|\phi\|}e^{nP(0)}\mu(C_n).
\end{align*}
Dividing by $e^{nP(0)}$, letting $n \rightarrow \infty$, and then $\gamma \rightarrow \infty$ gives $P_+'(0) \leq \beta$ as desired.

To show that $P_+'(0) \geq \int \phi d\mu$, we use Lemma \ref{variational principal upper bound} for the sub-multiplicative potential $\varphi \cdot e_{q\phi}$ and Lemma \ref{thm:gibbs_equilibrium} for the quasi-multiplicative potential $\varphi$ to obtain
\begin{equation*}
P(q) \geq P_\mu(\varphi \cdot e_{q\phi}) \geq P_\mu(\varphi) + q \int \phi d\mu = P(0) + q \int \phi d\mu
\end{equation*}
for all $q \ge 0$. The proof follows.
\end{proof}

\section{Dimension of infinitely generated self-affine sets} \label{sec:dim}

In this section, we prove Theorem \ref{thm:dim_result}, that is, we show that the dimension of a typical infinitely generated self-affine set is a supremum of dimensions of its finitely generated subsets. We also examine when the projection of the Gibbs measure is a measure of maximal dimension. The reader is prompted to recall notation from \S \ref{sec:preli_s-a}.

\begin{proof}[Proof of Theorem \ref{thm:dim_result}]
  Define $s_0 = \inf\{ s : P(\varphi^s) \le 0 \}$ and let $(I_\ell)_{\ell \in \N}$ be a sequence of non-empty finite sets $I_\ell \subset \N$ with $I_\ell \subset I_{\ell+1}$ and $\Gamma \subset \bigcup_{k=1}^K I_\ell^k$ for all $\ell \in \N$ such that $\N = \bigcup_{\ell \in \N} I_\ell$. Fix $\ell \in \N$ and let $0 < s_\ell \le s_0$ be such that $P(\varphi^{s_\ell},I_\ell)=0$. To show that $s_0 \le \sup_{\ell \in \N} s_\ell$, take $s<s_0$. Since $P(\varphi^s)>0$ and $P(\varphi^s,I_\ell) \to P(\varphi^s)$ by Proposition \ref{finite approximation property}, we may choose $\ell_0 \in \N$ so that $P(\varphi^s,I_{\ell_0})>0$. Therefore $s_{\ell_0} > s$, and, consequently, $s_0 = \sup_{\ell \in \N} s_\ell$.

  Since $\dimh(\pi_{\mathbf{a}}(I_\ell^\N)) = \min\{ d,s_\ell \}$ for $\LL_\mathbf{A}$-almost all $\mathbf{a} \in \mathbf{A}$ by \cite[Theorem 5.3]{Falconer1988} and $\bigcup_{\ell \in \N} \pi_{\mathbf{a}}(I_\ell^\N) \subset F_\mathbf{a}$, we have $\min\{ d,s_0 \} \le \dimh(F_\mathbf{a})$. To show that $\dimh(F_\mathbf{a}) \le s_0$, take $s < \dimh(F_\mathbf{a})$. Choose $m \in \Z$ and $0<\delta\le 1$ so that $s = m+\delta$ and let $\Delta$ be a closed ball such that $f_i(\Delta) \subset \Delta$ for all $i \in \N$. It follows from the definition of singular values that for each $\omega \in \Sigma_*$ we may cover $f_\omega(\Delta)$ with at most a constant times
  \begin{equation*}
    \frac{\gamma_1(\omega)}{\gamma_{m+1}(\omega)} \frac{\gamma_2(\omega)}{\gamma_{m+1}(\omega)} \cdots \frac{\gamma_m(\omega)}{\gamma_{m+1}(\omega)}
  \end{equation*}
  balls of radius $\gamma_{m+1}(\omega)$. Thus there exists $c \ge 1$ so that
  \begin{equation*}
    \HH_{2^{-k}}^s(F_{\mathbf{a}}) \le \sum_{\omega \in \Sigma_k} \HH_{2^{-k}}^s(f_\omega(\Delta)) \le c\sum_{\omega \in \Sigma_k} \varphi^s(\omega)
  \end{equation*}
  for all $k \in \N$. It follows that $\sum_{\omega \in \Sigma_k} \varphi^s(\omega) \ge 1$ for all $k \in \N$ large enough. Thus $P(\varphi^s) \ge 0$ and $s \ge s_0$ which finishes the proof.
\end{proof}

Considering the projection $\pi_\mathbf{a}$, we denote the pushforward measure of $\mu \in \MM_\sigma(\Sigma)$ by $\pi_\mathbf{a}\mu$.

\begin{theorem} \label{thm:dim_result_for_gibbs}
  If $(T_i)_{i \in \N} \in \GL^\N$ is such that $\sup_{i \in \N} \| T_i \| < \tfrac12$, the singular value function $\varphi^s$ is quasi-multiplicative for all $0 \le s \le d$, there exists $0 \le s_0 \le d$ so that $P(\varphi^{s_0}) = 0$, $\mu$ is the Gibbs measure for $\varphi^{s_0}$ so that $\Lambda_\mu(\varphi^{s_0}) > -\infty$, and $F_\mathbf{a}' \subset F_{\mathbf{a}}$ with $\pi_\mathbf{a}\mu(F_\mathbf{a}')>0$, then
  \begin{equation*}
    \dimh(F_\mathbf{a}') = \dimh(F_\mathbf{a})
  \end{equation*}
  for $\LL_\mathbf{A}$-almost all $\mathbf{a} \in \mathbf{A}$.
\end{theorem}

\begin{proof}
  Let $s<t<s_0$ and recall that by Lemma \ref{thm:gibbs_equilibrium}, Theorem \ref{thm:main_gibbs}, and Lemma \ref{variational principal upper bound}, the measure $\mu$ is ergodic and satisfies $h_\mu + \Lambda_\mu(\varphi^{s_0}) = 0$. Hence, by Shannon-McMillan Theorem and Kingman's sub-additive ergodic theorem, we have
  \begin{equation*}
    \lim_{n \to \infty} \frac{\log\mu([\omega|_n])}{\log\varphi^t(\omega|_n)} > 1
  \end{equation*}
  for $\mu$-almost all $\omega \in \Sigma$. Applying Egorov's theorem, we find for each $\eps>0$ a compact set $C \subset \Sigma$ and $n_0 \in \N$ so that $\mu(C)>1-\eps$ and $\mu([\omega|_n]) \le \varphi^t(\omega|_n)$ for all $\omega \in C$ and $n \ge n_0$. Now, according to \cite[Proposition 3.1(i)]{Solomyak1998}, we have
  \begin{align*}
    \int_\mathbf{A}\int_C\int_\Sigma \frac{d\mu(\omega)d\mu(\tau)d\mathbf{a}}{|\pi_\mathbf{a}(\omega) - \pi_\mathbf{a}(\tau)|^s} &\le c'\int_C\int_\Sigma \varphi^s(\omega \wedge \tau)^{-1} d\mu(\omega)d\mu(\tau) \\
    &= c'\sum_{n=0}^\infty \sum_{\omega \in \Sigma_n} \varphi^s(\omega)^{-1}\mu([\omega]) \mu(C \cap [\iii]) \\
    &\le c\sum_{n=0}^\infty \sum_{\omega \in \Sigma_n} \varphi^s(\omega)^{-1} \varphi^t(\omega) \mu([\omega]) \\
    &\le c\sum_{n=0}^\infty 2^{-(t-s)n} < \infty
  \end{align*}
  for some constants $c,c'>0$. Observe that \cite[Proposition 3.1(i)]{Solomyak1998} is a refinement of \cite[Lemma 3.1]{Falconer1988} and it generalises immediately to the infinite case. It follows that
  \begin{equation*}
    \liminf_{r \downarrow 0} \frac{\log\pi_\mathbf{a}\mu(B(\pi_\mathbf{a}(\tau),r))}{\log r} = \sup\bigl\{ t \ge 0 : \int_\Sigma \frac{d\mu(\omega)}{|\pi_\mathbf{a}(\omega) - \pi_\mathbf{a}(\tau)|^t} < \infty \bigr\} \ge s
  \end{equation*}
  for $\mu$-almost all $\tau \in C$ and for $\LL_\mathbf{A}$-almost all $\mathbf{a} \in \mathbf{A}$. The proof is finished by recalling \cite[Proposition 2.3(a)]{Falconer1997} and Theorem \ref{thm:dim_result}.
\end{proof}

To finish this section, we provide the reader with a sufficient condition to guarantee the finiteness of the Lyapunov exponent in Theorem \ref{thm:dim_result_for_gibbs}. Recall that $s_\infty = \inf\{ s : P(\varphi^s) < \infty \}$.

\begin{lemma} \label{thm:lyapunov_finite_for_gibbs}
  If $(T_i)_{i \in \N} \in \GL^\N$ is such that $\sup_{i \in \N}\| T_i \| < 1$, the singular value function $\varphi^s$ is quasi-multiplicative for all $0 \le s \le d$, $s_0 > s_\infty$, and $\mu$ is the Gibbs measure for $\varphi^{s_0}$, then $\Lambda_\mu(\varphi^{s_0}) > -\infty$.
\end{lemma}

\begin{proof}
  Observe that since $P(\varphi^{s_0}) < \infty$ the Gibbs measure $\mu$ for $\varphi^{s_0}$ exists by Theorems \ref{thm:gibbs_exists} and \ref{Gibbs implies ergodicity}. To prove the claim, let $m \in \Z$ be so that $m < s_0 \le m+1$. By the Gibbs property there is a constant $C \ge 1$ so that
  \begin{equation*}
    \mu([\omega]) \le C\varphi^{s_0}(\omega) e^{-nP(\varphi^{s_0})}
  \end{equation*}
  for all $\omega \in \Sigma_n$ and $n \in \N$. Thus
  \begin{equation} \label{eq:log_lasku}
    \log\frac{\varphi^{s_0}(\omega)}{\mu([\omega])} \ge nP(\varphi^{s_0}) - \log C
  \end{equation}
  for all $\omega \in \Sigma_n$ and $n \in \N$.

  If $\max\{ s_\infty,m \} < t < s_0$, then $P(\varphi^t) < \infty$ and $Z_n(\varphi^t) < \infty$ for all $n \in \N$ by Lemma \ref{positive recurrence}. As in \eqref{eq:jensen_eq_calc}, Jensen's inequality gives
  \begin{align} \label{eq:log_lasku2}
    \sum_{\omega \in \Sigma_n} \mu([\omega]) \log\frac{\varphi^t(\omega)}{\mu([\omega])} \le \log \biggl(\sum_{\omega \in \Sigma_n} \varphi^t(\omega)\biggr)= \log Z_n(\varphi^t).
  \end{align}
  Since $\varphi^{s_0}(\omega) = \gamma_{m+1}(\omega)^{s_0-t} \varphi^t(\omega)$ for all $\omega \in \Sigma_*$ we have, by \eqref{eq:log_lasku} and \eqref{eq:log_lasku2}, that
  \begin{align*}
    \bigl( nP(\varphi^{s_0}) - \log C \bigr) &\le \sum_{\omega \in \Sigma_n} \mu([\omega]) \biggl( \log\gamma_{m+1}(\omega)^{s_0-t} + \log\frac{\varphi^t(\omega)}{\mu([\omega])} \biggr) \\
    &\le \sum_{\omega \in \Sigma_n} \mu([\omega])\log\gamma_{m+1}(\omega)^{s_0-t} + \log Z_n(\varphi^t).
  \end{align*}
  Hence,
  \begin{align*}
    \tfrac{1}{n}\sum_{\omega \in J^n} \mu([\omega])\log\varphi^{s_0}(\omega) &\ge \tfrac{1}{n}\sum_{\omega \in I^n} \mu([\omega])\log\gamma_{m+1}(\omega)^{m+1} \\
    &\ge \frac{(m+1)\bigl(nP(\varphi^{s_0}) - \log C - \log Z_n(\varphi^t) \bigr)}{n(s_0-t)}.
  \end{align*}
 Letting $n \to \infty$ we have
  \begin{align*}
    \Lambda_{\mu}(\varphi^{s_0}) \geq \frac{(m+1)\bigl(P(\varphi^{s_0}) - P(\varphi^{t})\bigr)}{s_0-t}>-\infty.
  \end{align*}
\end{proof}

\section{Multifractal analysis of Birkhoff averages} \label{sec:multifractal}

The aim of this section is to prove Theorem \ref{General theorem}. The upper bound is proved in Proposition \ref{General upper bound prop Mar15} and the lower bound in Theorem \ref{LB in main Thm Mar30}. It is worth mentioning that the upper bound in Theorem \ref{General theorem} holds for all $\mathbf{a} \in \mathbf{A}$.

\subsection{Proof of the upper bound in Theorem \ref{General theorem}} \label{sec:multifractal_upper}
In this subsection we shall prove the upper bound in Theorem \ref{General theorem}. The reader is prompted to recall notation from \S \ref{sec:preli_s-a} and \S \ref{sec:preli_multi}.
We begin with a lemma relating the dimension of $J_{\Phi}(\alpha)$ to the singular value function. Define
\begin{equation*}
  A_{\Phi}(\alpha,n,k) = \bigl\{ \omega \in \Sigma_k : A_k\phi_i(\tau) \in B_{n}(\alpha_i) \text{ for all } \tau \in [\omega] \text{ and } i \in \{ 1,\ldots,n \} \bigr\}
\end{equation*}
for all $n,k \in \N$.

\begin{lemma}\label{initial covering lemma Mar15}
  If $(T_i)_{i \in \N} \in \GL^\N$ is such that $\sup_{i \in \N}\| T_i \| < 1$, $\Phi \colon \Sigma \to \R^\N$ has summable variations, $\alpha \in \overline{\R}^\N$, $\mathbf{a} \in \mathbf{A}$, $s< \dimh(J^{\mathbf{a}}_{\Phi}(\alpha))$, and $n \in \N$, then there is $k_0 \in \N$ such that
\begin{equation*}
  \sum_{\omega \in A_{\Phi}(\alpha,n,k)} \varphi^s(\omega) > 1
\end{equation*}
for all $k \geq k_0$.
\end{lemma}

\begin{proof}
Let $\delta = \sup_{i \in \N}\| T_i \|$ and set
\begin{equation*}
D_{\Phi}(\alpha,n,k) = \bigl\{ \omega \in \Sigma_k : \text{there is } \tau \in [\omega] \text{ such that } A_k\phi_i(\tau) \in B_{2n}(\alpha_i) \text{ for all } i \in \{ 1,\ldots,n \} \bigr\}
\end{equation*}
for all $n,k \in \N$.
Fix $n \in \N$. Since $\lim_{k \rightarrow \infty} \var_k\left(A_k\phi_i\right) =0$ we may choose $k_1$ so that $\var_k\left(A_k\phi_i\right)< (2n)^{-1}$ for all $i \in \{ 1,\ldots,n \}$ and all $k \ge k_1$. Thus we have $D_{\Phi}(\alpha,n,k) \subset A_{\Phi}(\alpha,n,k)$ for all $k \geq k_1$.
Since
\begin{equation*}
J_{\Phi}(\alpha) \subset \bigcup_{l \in \N}\bigcap_{k=l}^\infty\bigcup_{\omega \in D_{\Phi}(\alpha,n,k)}\pi_{\mathbf{a}}\left([\omega]\right)
\end{equation*}
and $\dimh(J^{\mathbf{a}}_{\Phi}(\alpha)) > s$ there is $l \in \N$ with
\begin{equation*}
\dimh \biggl(\bigcap_{k=l}^\infty \bigcup_{\omega \in D_{\Phi}(\alpha,n,k)} \pi_{\mathbf{a}}\left([\omega]\right)\biggr) >s.
\end{equation*}
Hence, continuing as in the proof of Theorem \ref{thm:dim_result}, we find a constant $c \ge 1$ so that
\begin{equation*}
\HH_{\delta^k}^s \biggl(\bigcup_{\omega \in D_{\Phi}(\alpha,n,k)} \pi_{\mathbf{a}}\left([\omega]\right)\biggr) \le c\sum_{\omega \in D_{\Phi}(\alpha,n,k)} \varphi^s(\omega)
\end{equation*}
for all $k \in \N$. The claim follows.
\end{proof}

\begin{prop}\label{General upper bound prop Mar15}
If $(T_i)_{i \in \N} \in \GL^\N$ is such that $\sup_{i \in \N}\| T_i \| < 1$, $\Phi \colon \Sigma \to \R^\N$ has summable variations, and $\alpha \in \overline{\R}^\N$, then
\begin{align*}
  \dimh(J^{\mathbf{a}}_{\Phi}(\alpha)) \leq \lim_{n \rightarrow \infty} \lim_{k \rightarrow \infty} \sup \bigl\{ D_k(\mu) : \;&\mu \in \M^*_{\sigma^k}(\Sigma) \text{ so that } \\ &\int A_i\phi_i d\mu \in B_n(\alpha_i) \text{ for all } i \in \{ 1,\ldots,n \}\bigr\}
\end{align*}
for all $\mathbf{a} \in \mathbf{A}$.
\end{prop}

\begin{proof}
Fix $\mathbf{a} \in \mathbf{A}$, $s < \dimh(J^{\mathbf{a}}_{\Phi}(\alpha))$, and $n \in \N$. According to Lemma \ref{initial covering lemma Mar15}, there is $k_0 \in \N$ such that
\begin{equation*}
\sum_{\omega \in A_{\Phi}(\alpha,n,k)}\varphi^s(\omega) > 1
\end{equation*}
for all $k \geq k_0$.
Let $k \geq k_0$ and choose a finite subset $F_{\Phi}(\alpha,n,k) \subset A_{\Phi}(\alpha,n,k)$ with $F(k) = \sum_{\omega \in F_{\Phi}(\alpha,n,k)}\varphi^s(\omega) \ge 1$. Define a compactly supported $k$-th level Bernoulli measure $\mu \in \M_{\sigma^k}(\Sigma)$ by setting
\begin{equation*}
  \mu([\omega]) =
  \begin{cases}
    \varphi^s(\omega)/F(k), &\text{if } \omega \in F_{\Phi}(\alpha,n,k), \\
    0, &\text{if } \omega \in \Sigma_k \setminus F_{\Phi}(\alpha,n,k)
  \end{cases}
\end{equation*}
for all $\omega \in \Sigma_k$. It follows immediately that
\begin{equation*}
\sum_{\omega \in \Sigma_k} \mu([\omega])\log \frac{\varphi^s(\omega)}{\mu([\omega])} = \log F(k) \ge 0
\end{equation*}
yielding $s \le D_k(\mu)$.
Since $\mu$ is supported on $\bigcup_{\omega \in F_{\Phi}(\alpha,n,k)}[\omega]$ and $A_k\phi_i(\tau) \in B_{n}(\alpha_i)$ for all $\omega \in F_{\Phi}(\alpha,n,k)$, $\tau \in [\omega]$, and $i \in \{1,\ldots,n \}$ we also have
\begin{equation*}
\int A_k\phi_i d\mu  \in B_{n}(\alpha_i)
\end{equation*}
for all $i \in \{ 1,\ldots,n \}$. These observations imply the proof.
\end{proof}

\subsection{Symbolic tree structure in level sets} \label{sec:multifractal_tree_structure}
The following proposition contains the essence of the proof of the lower bound in Theorem \ref{General theorem}.

\begin{prop}\label{prop symbolic tree structure Mar30}
If $(T_i)_{i \in \N} \in \GL^\N$ is such that $\sup_{i \in \N} \| T_i \| < \tfrac12$, the singular value function $\varphi^s$ is quasi-multiplicative for all $0 \le s \le d$, $\Phi \colon \Sigma \to \R^\N$ has summable variations, $\alpha \in \overline{\R}^\N$, and
\begin{align*}
  s < \lim_{n \rightarrow \infty} \lim_{k \rightarrow \infty} \sup \bigl\{ D_k(\mu) : \;&\mu \in \M^*_{\sigma^k}(\Sigma) \text{ so that } \\ &\int A_i\phi_i d\mu \in B_n(\alpha_i) \text{ for all } i \in \{ 1,\ldots,n \}\bigr\},
\end{align*}
then there exists a set $S \subset E_{\Phi}(\alpha)$, a Borel probability measure $\mu$ supported on $S$, and a constant $C \ge 1$, such that $\mu([\omega]) \leq C\varphi^s(\omega)$ for all $\omega \in \Sigma_*$.
\end{prop}

In fact, with this proposition, the lower bound in Theorem \ref{General theorem} follows almost immediately.

\begin{theorem} \label{LB in main Thm Mar30}
  If $(T_i)_{i \in \N} \in \GL^\N$ is such that $\sup_{i \in \N} \| T_i \| < \tfrac12$, the singular value function $\varphi^s$ is quasi-multiplicative for all $0 \le s \le d$, $\Phi \colon \Sigma \to \R^\N$ has summable variations, and $\alpha \in \overline{\R}^\N$, then
  \begin{align*}
    \dimh(J^{\mathbf{a}}_{\Phi}(\alpha)) \ge \min\bigl\{ d,\lim_{n \rightarrow \infty} \lim_{k \rightarrow \infty} \sup \bigl\{ D_k(\mu) : \;&\mu \in \M^*_{\sigma^k}(\Sigma) \text{ so that } \\ &\int A_i\phi_i d\mu \in B_n(\alpha_i) \text{ for all } i \in \{ 1,\ldots,n \}\bigr\} \bigr\}
  \end{align*}
  for $\LL_\mathbf{A}$-almost all $\mathbf{a} \in \mathbf{A}$.
\end{theorem}

\begin{proof}
  Let $s>0$ be as in Proposition \ref{prop symbolic tree structure Mar30}. Applying the measure given by Proposition \ref{prop symbolic tree structure Mar30} in the proof of Theorem \ref{thm:dim_result_for_gibbs}, we get $\dimh(\pi_\mathbf{a}(S)) \ge s$ for $\LL_\mathbf{A}$-almost all $\mathbf{a} \in \mathbf{A}$, where $S \subset E_\Phi(\alpha)$ is as in Proposition \ref{prop symbolic tree structure Mar30}. Thus $\dimh(J_\Phi^\mathbf{a}(\alpha)) \ge s$ for $\LL_\mathbf{A}$-almost all $\mathbf{a} \in \mathbf{A}$.
\end{proof}

In the course of the proof of Proposition \ref{prop symbolic tree structure Mar30}, we shall rely on the concept of $\M$-trees. This approach is inspired by a similar notion discussed by Furstenberg in \cite{Furstenberg2008}. We shall now define all the required concepts.

If $\omega, \tau \in \Sigma_* \cup \Sigma$ so that $\omega \wedge \tau = \omega$, then we write $\omega \preccurlyeq \tau$. This defines a partial order on $\Sigma_*$. Let $\mathbb{X} \subset \Sigma_*$ be an antichain with respect to $\preccurlyeq$. If there is a function $\MM_\mathbb{X} \colon \mathbb{X} \to [0,1]$ so that
\begin{equation*}
  \sum_{\omega \in \mathbb{X}} \MM_\mathbb{X}(\omega) = 1,
\end{equation*}
then the ordered pair $(\mathbb{X},\MM_\mathbb{X})$ is called an \emph{$\MM$-tree}. An $\MM$-tree $(\mathbb{X},\MM_\mathbb{X})$ is said to be \emph{finite} if $\mathbb{X}$ is a finite set. If $(\mathbb{X},\MM_\mathbb{X})$ and $(\mathbb{Y},\MM_\mathbb{Y})$ are $\MM$-trees so that
\begin{equation*}
  \mathcal{M}_{\mathbb{X}}(\omega)=\sum_{\tau \in \{ \kappa \in \mathbb{Y} : \omega \preccurlyeq \kappa \}}\mathcal{M}_{\mathbb{Y}}(\tau)
\end{equation*}
for all $\omega \in \mathbb{X}$, then we write $(\mathbb{X},\MM_\mathbb{X}) \preccurlyeq (\mathbb{Y},\MM_\mathbb{Y})$. This defines a partial order on the collection of all $\MM$-trees.

Next we shall define a limit for certain $\MM$-tree sequences. If $((\mathbb{X}_n,\MM_{\mathbb{X}_n}))_{n \in \N}$ is a chain of finite $\MM$-trees so that $\lim_{n \to \infty} \min\{ |\omega| : \omega \in \mathbb{X}_n \} = \infty$, then the limit of that sequence is defined to be
\begin{equation*}
  \lim_{n \to \infty} (\mathbb{X}_n,\MM_{\mathbb{X}_n}) = (\mathbb{X}_\infty,\MM_\infty),
\end{equation*}
where
\begin{equation*}
  \mathbb{X}_{\infty} = \{ \tau \in \Sigma : \text{for each } n \in \N \text{ there is } \omega \in \mathbb{X}_n \text{ so that } \omega \preccurlyeq \tau \}
\end{equation*}
and $\mathcal{M}_{\infty}$ is a Borel probability measure supported on $\mathbb{X}_{\infty}$ defined as follows. Observe first that since each $\mathbb{X}_n$ is a finite antichain, it is readily checked that the collection
\begin{equation*}
  \mathcal{A}(\mathbb{X}_\infty) = \{\emptyset, X_{\infty}\} \cup \{ [\omega]\cap \mathbb{X}_{\infty}: \omega \in \bigcup_{n \in \N}\mathbb{X}_n \}
\end{equation*}
is a semi-algebra of subsets of $\mathbb{X}_\infty$.
Moreover, since $\lim_{n \rightarrow \infty} \min \left\lbrace  |\omega|: \omega \in \mathbb{X}_n\right\rbrace =\infty$, it is clear that this semi-algebra generates the Borel $\sigma$-algebra restricted to $\mathbb{X}_{\infty}$.
We define $\mathcal{M}_{\infty}$ on $\mathcal{A}(\mathbb{X}_\infty)$ by setting $\mathcal{M}_{\infty}(\emptyset)=0$, $\mathcal{M}_{\infty}(\mathbb{X}_{\infty})=1$, and
\[
  \mathcal{M}_{\infty}\left( [\omega]\cap \mathbb{X}_{\infty} \right)=\mathcal{M}_{\mathbb{X}_n}(\omega).
\]
for all $\omega \in \mathbb{X}_n$ and $n \in \N$.
It follows from the fact that $\left(\mathbb{X}_n,\mathcal{M}_{\mathbb{X}_n}\right) \preccurlyeq \left(\mathbb{X}_{n+1},\mathcal{M}_{\mathbb{X}_{n+1}}\right)$ for each $n \in \N$, that this set function is well-defined and countably additive. Thus $\MM_\infty$ extends to a measure on $\Sigma$.

Finally, given a subset $\Omega \subset \Sigma_* \cup \Sigma$ let $\mathbb{D}(\Omega) \subset \N$ be the collection of all digits contained within words from $\Omega$, that is, $\mathbb{D}(\Omega) = \{ l \in \N : \text{there is } \omega \in \Omega \text{ with } \omega_i = l \text{ for some } i \}$.

\begin{proof}[Proof of Proposition \ref{prop symbolic tree structure Mar30}]
We begin by noting that, without loss of generality, we may assume that if $\phi$ is in the sequence $\Phi$, then also $-\phi$ is in $\Phi$. Indeed, if $\Phi = (\phi_i)_{i \in \N}$, $\alpha = (\alpha_i)_{i \in \N}$, and the right-hand side of the inequality in the formulation of Proposition \ref{prop symbolic tree structure Mar30} is denoted by $D(\Phi,\alpha)$, then we clearly have $D(\Phi,\alpha) = D(\Phi',\alpha')$, where $\Phi' = (\phi_i')_{i \in \N}$ and $\alpha' = (\alpha_i')_{i \in \N}$ are defined so that $\phi_{2i}' = \phi_i$, $\phi_{2i-1}' = -\phi_i$, $\alpha_{2i}' = \alpha_i$, $\alpha_{2i-1}' = -\alpha_i$ for all $i \in \N$.

Choose $s<t<D(\Phi,\alpha)$ and define
\begin{equation} \label{eq:A_i def}
  A_i = \sup\{ |\phi(\omega)| : \omega \in [\tau] \text{ and } \tau \in \mathbb{D}(\Gamma \cup \{ 1 \}) \}
\end{equation}
for all $i \in \N$. For each $n \in \N$ we choose $k=k(n) \geq 4K n \left(\max_{i \leq n}A_i+1\right)$ so that $\var_{k(n)}A_{k(n)}\phi_i < \frac{1}{2n}$ for all $i \in \{ 1,\ldots,n \}$ and there exists $\nu_n \in \M^*_{\sigma^k}(\Sigma)$ with $D_k(\nu_n)>t$ and
\begin{equation} \label{interval 10th Jan 12}
\int A_k\phi_i d \nu_n \in B_{2n}(\alpha_i)
\end{equation}
for all $i \in \{ 1,\ldots,n \}$. Let $\rho_n \in \MM_{\sigma^k}^*(\Sigma)$ be the compactly supported $k(n)$-th level Bernoulli measure given by
\begin{equation*}
\rho_n([\omega_1 \cdots \omega_{k(n)q}]) = \prod_{j=0}^{q-1} \nu_n\left([\omega_{jk(n)+1} \cdots \omega_{(j+1)k(n)}]\right).
\end{equation*}
For each potential $\phi_i$ we define a $k$-th level locally constant potential $\underline{A}_k\phi_i$ by
\begin{equation*}
\underline{A}_k\phi_i(\omega) = \inf\bigl\{ \frac{1}{k}\sum_{l=0}^{k-1} \phi_i(\sigma^l\tau): \tau_j = \omega_j \text{ for } j \in \{ 1,\ldots,k \} \bigr\}.
\end{equation*}
Note that $A_k\phi_i(\omega)-\var_{k}A_{k}\phi_i \leq \underline{A}_k\phi_i(\omega) \leq A_k\phi_i(\omega)$ for all $\omega \in \Sigma$ and $i \in \N$. Since $\var_{k}A_{k}(\phi_i)<\frac{1}{2n}$ it follows from (\ref{interval 10th Jan 12}) that
\begin{equation*}
\int \underline{A}_k\phi_i d \rho_n \in B_n(\alpha_i)
\end{equation*}
for all $i \in \{ 1,\ldots,n \}$. Moreover, it is immediate from $D_k(\nu_n)>t$ that
\begin{equation*}
\sum_{\omega \in \Sigma_{k(n)}} \rho_n(\omega) \log \frac{\varphi^t(\omega)}{\rho_n(\omega)} >0.
\end{equation*}
We let $\D(n) = \left\lbrace \omega \in \Sigma_{k(n)}:\rho_n(\omega)>0 \right\rbrace$. Since $\rho_n \in \MM_{\sigma^k}^*(\Sigma)$ the number of words in $\D(n)$ is finite. Hence, for each $n$ there is a finite set of digits $\mathbb{D}^*(n) \subset \N$ defined by
\[
  \mathbb{D}^*(n) = \mathbb{D}\biggl( \bigcup_{l=1}^{n+2}\D(l)\cup\{1\}\cup \Gamma\biggr).
\]
Since we also have $\var_{1}\phi_i<\infty$ the quantities
\begin{equation}\label{ A and B }
\begin{split}
\A(n) &= \sup \bigl\{ |\phi_i(\tau)| : \tau \in [\omega]\text{ for some }\omega \in \mathbb{D}^*(n) \text{ and } i \in \{ 1,\ldots,n \} \bigr\}, \\
\B(n) &= \sup\{ \rho_n(\omega)^{-1} : \omega \in \D(n) \}
\end{split}
\end{equation}
are both finite. By Kolmogorov's strong law of large numbers, we have
\begin{equation*}
\lim_{N \rightarrow \infty} \frac{1}{N}\sum_{j=0}^{N-1}\log \frac{\varphi^t(\omega_{jk(n)+1} \cdots \omega_{(j+1)k(n)})}{\rho_n([\omega_{jk(n)+1} \cdots \omega_{(j+1)k(n)}])} = \sum_{\tau \in \Sigma_{k(n)}} \rho_n(\tau) \log \frac{\varphi^t(\tau)}{\rho_n([\tau])} > 0
\end{equation*}
for $\rho_n$-almost all $\omega \in \Sigma$, and
\begin{equation*}
\lim_{N \rightarrow \infty} \frac{1}{N}\sum_{j=0}^{N-1}\underline{A}_{k(n)}\phi_i(\sigma^{jk(n)}\omega) \in B_n(\alpha_i)
\end{equation*}
and for all $i \in \{ 1,\ldots,n \}$.
By Egorov's theorem we find $S_n \subset \supp(\rho_n)$ with $\rho_n(S_n)>1/2$ so that each of the above convergences are uniform upon $S_n$. Hence there is $L(n) \in \N$ such that
\begin{equation}\label{ERGODIC LIMITS}
\begin{split}
\prod_{j=0}^{N-1}\varphi^t(\omega_{jk(n)+1} \cdots \omega_{(j+1)k(n)}) &> \prod_{j=0}^{N-1}\rho_n([\omega_{jk(n)+1} \cdots \omega_{(j+1)k(n)}]), \\
\frac{1}{N}\sum_{j=0}^{N-1}\underline{A}_{k(n)}\phi_i(\sigma^{jk(n)}\omega) &\in B_n(\alpha_i)\text{ for } i \in \{ 1,\ldots,n \}.
\end{split}
\end{equation}
for all $\omega \in S_n$ and all $N \geq L(n)$. We also let $M(n) = \max \left\lbrace \varphi^t(\tau)^{-1}: [\tau] \cap \supp(\rho_n) \neq \emptyset \right\rbrace.$

For every $\alpha, \beta \in \Sigma_*$, according to the quasi-multiplicativity of $\varphi^t$, there exists $\omega \in \Gamma$ such that
\begin{equation*}
\varphi^t\left(\alpha\omega\beta \right) \geq c(t) \varphi^t(\alpha) \varphi^t(\beta),
\end{equation*}
where $c(t)>0$ is a constant depending only on $t$. We let $\alpha \star \beta$ denote the word $\alpha\omega\beta$, so $\varphi^t\left(\alpha \star \beta\right) \geq c(t)\varphi^t(\alpha)\varphi^t(\beta)$. Note that for any given $\alpha,\beta \in \Sigma_*$ there are at most $K = \max\left\lbrace |\omega|: \omega \in \Gamma \right\rbrace$ finite words $\beta' \in \Sigma_*$ with $\alpha\star\beta' = \alpha\star\beta$ (including $\beta$ itself). We also write $\alpha\star\beta\star\omega = (\alpha\star\beta)\star\omega$.

Our aim is to construct a sequence of $\M$-trees $((\T_n,\M_n))_{n \in \N\cup\{0\}}$ with $(\T_{n-1},\M_{n-1})  \preccurlyeq (\T_{n},\M_{n})$ for all $n \in \N$, along with functions $\left(\Psi_n\right)_{n \in \N}$ of the form $\Psi_n\colon \T_n \rightarrow \Sigma_*$, together with a sequence $(\gamma_n)_{n \in \N}$ with the property that every $\tau \in \T_{n}$ satisfies $\gamma_n-K \leq |\Psi_n(\tau)| \leq \gamma_n$.

We begin by letting $\T_0=\{\emptyset\}$, $\M_0(\emptyset)=1$, $\Psi_0 =\{\emptyset \mapsto \emptyset\}$, and $\gamma_0=0$. Suppose we have defined $(\M_{n-1},\T_{n-1})$, $\Psi_{n-1}\colon \T_{n-1} \rightarrow \Sigma_*$, and $\gamma_{n-1}$ with the required properties. For each $\omega \in \T_{n-1}$ we let
\begin{equation*}
\Z_{n-1}(\omega) = \left\lbrace \tau \in \T_{n-1}: \Psi_{n-1}(\tau) \preccurlyeq \Psi_{n-1}(\omega)\text{ or }\Psi_{n-1}(\omega) \preccurlyeq \Psi_{n-1}(\tau)\right\rbrace.
\end{equation*}
We shall construct $(\M_n,\T_n)$, $\Psi_n \colon \T_n \rightarrow \Sigma_*$, and $\gamma_n$ as follows. First take $q_n\in \N$ so that
\begin{equation}\label{q def}
\begin{split}
q_n &>  \frac{4 M(n)^{L(n)} M(n+1)^{L(n+1)}}{\min\{ \varphi^t(\Psi_{n-1}(\omega)): \omega \in \T_{n-1}\}}
+ \max \{ \#\Psi_{n-1}^{-1}(\omega): \omega \in \T_{n-1} \} \\
&\quad+ n L(n+1)(\A(n)+1) (\B(n)+1)(\gamma_{n-1}+4K+k(n+1)+k(n+2)+1)\\
&\quad+ \#\D(n) \#\D(n+1) \max\left\lbrace \#\Z_{n-1}(\omega): \omega \in \T_{n-1} \right\rbrace +q_{n-1}\\
&\quad+ \#\D(n+1) \#\bigl(\{1\}\cup \Gamma \cup \mathbb{D}(\D(n)) \cup \mathbb{D}(\D(n+1))\bigr)^{5K+k(n+1)+k(n)+2},
\end{split}
\end{equation}
where $\A(n)$ and $\B(n)$ are as in (\ref{ A and B }).

Let $\F_n = \{ \tau \in \Sigma_{k(n)q_n} : [\tau] \cap S_n \neq \emptyset \}$ and $\eta(n) = \sum_{\tau \in \F_n} \rho_n ([\tau]) \geq \rho_n(S_n) >1/2$. Define
\begin{equation*}
\F_n^l = \{  \tau \in \Sigma_{k(n)l} : \text{there is } \beta \in \Sigma_{k(n)(q_n-l)} \text{ with }\tau \beta \in \F_n \}
\end{equation*}
for all $l \in \{ 1,\ldots,q_n \}$. In the process of constructing $(\T_n,\M_n)$, $\Psi_n$, and $\gamma_n$, we shall construct a sequence of intermediary $\M$-trees $((\T^l_n,\M^l_n))_{l=0}^{q_n}$ so that $(\T^l_n,\M^l_n) \preccurlyeq (\T^{l+1}_n,\M^{l+1}_n)$ and
\begin{equation} \label{eq:tree_goal1}
  (\T_{n-1},\M_{n-1}) \preccurlyeq (\T^l_n, \M^l_n) \preccurlyeq (\T_n,\M_n)
\end{equation}
for all $l \in \{1, \ldots, q_n\}$. In addition, we construct intermediary maps $\Psi^l_n\colon \T^l_n \rightarrow \Sigma_*$ and $\left(\gamma_n^l\right)_{l=0}^{q_n}$ so that $\gamma_n^l-K \leq |\Psi_n^l(\tau)| \leq \gamma_n^l$ for all $\tau \in \T^l_n$ and $\Psi_n^l(\omega^l) \preccurlyeq \Psi_n^{l+1}(\omega^{l+1})$ for all $\omega^l \in \T_n^l$, $\omega^{l+1} \in \T_n^{l+1}$, and $\omega^l \preccurlyeq \omega^{l+1}$.

First take $(\T^0_n,\M^0_n) = (\T_{n-1},\M_{n-1})$, $\Psi_n^0 = \Psi_{n-1}$, and $\gamma_n^0=\gamma_{n-1}$. Clearly $(\T^0_n,\M^0_n)$, $\Psi_n^0$, and $\gamma_n^0$ satisfy the required properties. For each $l \in \{1, \ldots, q_n\}$ we let
\begin{equation*}
\T_n^l = \{ \kappa\tau : \kappa \in \T_{n-1} \text{ and } \tau \in \F_n^l \}.
\end{equation*}
For each $\omega \in \T_n^l$ we take the (unique) pair $\kappa \in \T_{n-1}$ and $\tau \in \F_n^l$ with $\omega = \kappa\tau$ and let
\begin{equation*}
\M_n^l(\omega) = \frac{\M_{n-1}(\kappa)}{\eta(n)} \sum_{\beta \in \{ \alpha : \tau\alpha \in \F_n\}} \rho_n([\tau \beta]).
\end{equation*}
It is clear that $(\T^l_n,\M^l_n) \preccurlyeq (\T^{l+1}_n,\M^{l+1}_n)$ and if we let $(\T_n,\M_n) = (\T^{q_n}_n,\M^{q_n}_n)$, we have shown \eqref{eq:tree_goal1}. We shall construct the functions $(\Psi_n^l)_{l=1}^{q_n}$ and numbers $(\gamma_n^{l})_{l=1}^{q_n}$ recursively.

Suppose $l\in \{1, \ldots, q_n\}$ and we have constructed $\Psi_n^{l-1}$ and $\gamma_n^{l-1}$ satisfying the required properties. Define $\gamma_n^l = \gamma_n^{l-1}+2K+k(n)$ and let $\omega = \kappa\tau \in \T_n^l$ so that $\kappa \in \T_{n-1}$ and $\tau \in \F_n^l$. Choose $\tau' \in \F_n^{l-1}$ so that $\tau' \preccurlyeq \tau$ and set $\omega' = \kappa\tau' \in \T_n^{l-1}$. Thus there exists $\tau_l \in \Sigma_{k(n)}$ so that $\omega = \omega'\tau_l$. The function $\Psi_n^l$ is defined by setting
\begin{equation*}
  \Psi_n^l(\omega) = \Psi_n^{l-1}(\omega')\star\tau_l\star 1\cdots 1,
\end{equation*}
where the length of $1 \cdots 1$ is $\gamma_n^l - K - |\Psi_n^{l-1}(\omega') \star \tau_l|$. Since $\gamma_n^{l-1}-K \leq |\Psi_n^{l-1}(\omega')| \leq \gamma_n^{l-1}$ we have $\gamma_n^{l-1}-K+k(n) \leq |\Psi_n^l(\omega')\star\tau_l| \leq \gamma_n^{l-1}+k(n)+K$. Thus $\Psi_n^l(\omega)$ is well-defined, the length of $1 \cdots 1$ is at most $2K$, and $\gamma_n^l-K \leq |\Psi_n^l(\omega)| \leq \gamma_n^l$. Moreover, if we let $c_0 = c(t)^{2}\varphi^t(1\cdots 1)$, where the length of $1\cdots 1$ is $2K$, then a simple induction gives
\begin{equation}\label{singular value lower bound at l stage}
  \varphi^t(\Psi_{n-1}(\kappa)) \prod_{j=1}^l\varphi^t(\tau_j) \leq c_0^{-l}\varphi^t(\Psi_n^l(\kappa\tau)),
\end{equation}
where each $\tau_j$ has length $k(n)$.
We emphasise that $c_0$ is independent of $n$ and $l$. Recalling that $\T_n=\T_n^{q_n}$, we set $\Psi_n = \Psi_n^{q_n}$.

To finish the construction of $\MM$-trees $((\T_n,\MM_n))_{n \in \N \cup \{0\}}$, functions $(\Psi_n)_{n \in \N}$, and the sequence $(\gamma_n)_{n \in \N}$, we shall show that
\begin{equation} \label{count of phi overlaps}
  \max \{ \#\Z_n^l(\omega) : \omega \in \T_n^l \} \leq q_{n-1}(2K)^{l+q_{n-1}}
\end{equation}
for all $l \in \{ 1,\ldots,q_n \}$, and $n \in \N$, where
\begin{equation*}
  \Z_n^l(\omega) = \{ \omega' \in \T_n^l : \Psi_n^l(\omega) \preccurlyeq \Psi_n^l(\omega') \text{ or } \Psi_n^l(\omega') \preccurlyeq \Psi_n^l(\omega) \}
\end{equation*}
for all $\omega \in \T_n^l$, $l \in \{ 1,\ldots,q_n \}$, and $n \in \N$.

Fix $\omega \in \T_n^l$ and choose $\kappa \in \T_{n-1}$ and $\tau \in \F_n^l$ so that $\omega = \kappa\tau$. Write $\tau=\tau_1\cdots\tau_l$, where each $\tau_j \in \Sigma_{k(n)}$. Now suppose $\omega' \in \Z_n^l(\omega)$ and similarly take $\kappa' \in \T_{n-1}$ and $\tau' = \tau'_1\cdots\tau'_l \in \F_n^l$ so that $\omega' = \kappa'\tau'$ and $\tau'_j \in \Sigma_{k(n)}$. It is clear that either $\Psi_{n-1}(\kappa) \preccurlyeq \Psi_{n-1}(\kappa')$ or $\Psi_{n-1}(\kappa') \preccurlyeq \Psi_{n-1}(\kappa)$. Thus we have $\kappa' \in \Z_{n-1}(\kappa)$. Now since either $\Psi_n^l(\omega)\preccurlyeq \Psi_n^l(\omega')$ or $\Psi_n^l(\omega')\preccurlyeq  \Psi_n^l(\omega)$ and $|\Psi_n^{l-1}(\omega'')|\leq \gamma_n^{l-1}< \gamma_n^l-K \leq |\Psi_n^{l}(\omega)|$, we have $\Psi_n^{l-1}(\omega'')\preccurlyeq  \Psi_n^l(\omega)$, where $\omega'' = \kappa' \tau'_1 \cdots \tau_{l-1}' \in \T_n^{l-1}$. Thus, for $j \leq l-1$ there is a subword of $\Psi_n^l(\omega)$ which starts between positions $\gamma^{j-1}_n-K$ and $
\gamma^{j-1}_n+K$ (or between positions $\gamma_{n-1}-K$ and $
\gamma_{n-1}+K$ if $j=1$). Also, $\tau'_l \in \D(n)$. This shows that
\begin{equation*}
  \#\Z_n^l(\omega) \leq \#\Z_{n-1}(\kappa) (2K)^{l-1} \#\D(n).
\end{equation*}
Hence
\begin{align*}
\max \{ \#\Z_n^l(\omega): \omega \in \T_n^l \} &\leq \max \{ \#\Z_{n-1}(\kappa): \kappa \in \T_{n-1} \} (2K)^{l-1} \#\D(n)\\
&= \max \{ \#\Z_{n-1}^{q_{n-1}}(\kappa): \omega \in \T_{n-1}^{q_{n-1}} \}  (2K)^{l-1} \#\D(n).
\end{align*}
Iterating this inequality and applying the definition of $q_{n-1}$ we obtain
\begin{align*}
\max \{ \#\Z_n^l(\omega): \omega \in \T_n^l \} &\leq \max \{ \#\Z_{n-2}(\kappa): \kappa \in \T_{n-2} \} (2K)^{q_{n-1}+l-2} \#\D(n-1)\#\D(n) \\
&\leq q_{n-1} (2K)^{q_{n-1}+l}.
\end{align*}

Let $\left(T, \nu \right) = \lim_{n \rightarrow \infty}\left(\T_n, \M_n\right)$. Then $T$ consists of all $\omega \in \Sigma$ such that there is a sequence $\left(\omega_n\right)_{n \in \N} \in \prod_{n \in \N}\T_n$ such that $\omega_{n}\preccurlyeq\omega_{n+1}\preccurlyeq\omega$ for all $n \in \N$. It follows from the construction of $\left(\Psi_n\right)_{n \in \N}$ that $\Psi_n(\omega_{n})\preccurlyeq\Psi_{n+1}(\omega_{n+1})$ and $|\Psi_n(\omega_{n})|< |\Psi_{n+1}(\omega_{n+1})|$. Thus, there is a unique infinite word $\Psi(\omega) \in \Sigma$ with $\Psi_n(\omega_{n})\preccurlyeq\Psi(\omega)$ for all $n \in \N$. This defines a map $\Psi \colon T \rightarrow \Sigma$. We let $S = \Psi(T)$ and $\mu = \nu \circ \Psi^{-1}$ and let $S^* = \{ \omega \in \Sigma_* : [\omega] \cap S \neq \emptyset \}$.

In Lemmas \ref{ first nu estimate }--\ref{S subset of E} we shall verify that $S$ and $\mu$ defined above have the required properties. The proof of Proposition \ref{prop symbolic tree structure Mar30} thus follows.
\end{proof}

\begin{lemma}\label{ first nu estimate }
In the setting of the proof of Proposition \ref{prop symbolic tree structure Mar30}, let $n \in \N$, $l \in \{ 1,\ldots,q_n \}$, and $\omega=\kappa\tau \in \T_n^l$ so that $\kappa \in \T_{n-1}$ and $\tau \in \F_n^l$. Then
\begin{equation*}
\nu([\omega]) \leq 2\nu([\kappa]) M(n)^{L(n)}C_0^l \biggl(\frac{\varphi^t\left(\Psi_n^l(\omega) \right)}{\varphi^t(\Psi_{n-1}(\kappa))}\biggr).
\end{equation*}
\end{lemma}

\begin{proof}
Since $\omega \in \T_n^l$ and $\kappa \in \T_{n-1}$ we have $\nu([\kappa])=\M_{n-1}(\kappa)$ and
\begin{align*}
\nu([\omega]) &= \M_n^l(\omega) = \frac{\M_{n-1}(\kappa)}{\eta(n)} \sum_{\beta \in \{ \alpha : \tau\alpha \in \F_n \} }\rho_n([\tau \beta]) \\
&= \frac{\nu([\kappa])}{\eta(n)} \sum_{\beta \in \{ \alpha : \tau\alpha \in \F_n \} }\rho_n([\tau])\rho_n([\beta])
\leq 2 \nu([\kappa]) \rho_n([\tau]).
\end{align*}
So it suffices to show that
\begin{equation*}
\rho_n([\tau]) =\prod_{j=1}^l \rho_n([\tau_j]) \leq M(n)^{L(n)}C_0^l \biggl(\frac{\varphi^t\left(\Psi_n^l(\kappa\tau) \right)}{\varphi^t(\Psi_{n-1}(\kappa))}\biggr).
\end{equation*}
Thus, by (\ref{singular value lower bound at l stage}), it suffices to show that
\begin{equation} \label{eq:maali}
\prod_{j=1}^l \rho_n([\tau_j]) \leq M(n)^{L(n)} \prod_{j=1}^l \varphi^t([\tau_j]).
\end{equation}
Now either $l\geq L(n)$, in which case it follows from $\tau \in \F_n^l$ and (\ref{ERGODIC LIMITS}) that
\begin{equation*}
\prod_{j=1}^l \rho_n([\tau_j]) \leq \prod_{j=1}^l \varphi^t([\tau_j]),
\end{equation*}
or $l<L(n)$, in which case we have
\begin{equation*}
\prod_{j=1}^l \varphi^t([\tau_j])^{-1} \leq M(n)^l \leq M(n)^{L(n)}.
\end{equation*}
This shows \eqref{eq:maali} and thus completes the proof of the lemma.
\end{proof}

\begin{lemma}\label{second nu est}
In the setting of the proof of Proposition \ref{prop symbolic tree structure Mar30}, let $n \in \N$, $l \in \{ 1,\ldots,q_n \}$, and $\omega \in \T_n^l$. Then
\begin{equation*}
\nu([\omega]) \leq q_{n-1}C_0^{q_{n-1}+l} \varphi^t\bigl(\Psi_n^l(\omega) \bigr).
\end{equation*}
\end{lemma}

\begin{proof}
Since $\omega \in \T_n^l$ we may take $\kappa \in \T_{n-1}$ and $\tau \in \F_n^l$ so that $\omega= \kappa\tau$, so by Lemma \ref{ first nu estimate }, we have
\begin{equation*}
\nu([\omega]) \leq 2\nu([\kappa]) M(n)^{L(n)}C_0^l \biggl(\frac{\varphi^t\left(\Psi_n^l(\omega) \right)}{\varphi^t(\Psi_{n-1}(\kappa))}\biggr).
\end{equation*}
Moreover, since $\kappa \in \T_{n-1}=\T_{n-2}^{q_{n-1}}$ there exists $\kappa_- \in \T_{n-2}$ and $\tau_- \in \F_{n-1}^{q_{n-1}}$ so that $\kappa=\kappa_-\tau_-$. Applying Lemma \ref{ first nu estimate } once more, we obtain
\begin{equation*}
\nu([\kappa]) \leq 2\nu([\kappa_-]) M(n-1)^{L(n-1)}C_0^{q_{n-1}} \biggl(\frac{\varphi^t\left(\Psi_{n-1}^{q_{n-1}}(\kappa) \right)}{\varphi^t(\Psi_{n-2}(\kappa_-))}\biggr).
\end{equation*}
Combining these two estimates we get
\begin{equation*}
\nu([\omega]) \leq  \frac{4 M(n-1)^{L(n-1)} M(n)^{L(n)}}{\varphi^t(\Psi_{n-2}(\kappa_-))} C_0^{q_{n-1}+l} \varphi^t\bigl(\Psi_n^l(\omega) \bigr).
\end{equation*}
Noting that the definition of $q_{n-1}$ implies
\begin{equation*}
q_{n-1} \geq  \frac{4 M(n-1)^{L(n-1)} M(n)^{L(n)}}{\min\left\lbrace \varphi^t(\Psi_{n-2}(\kappa')): \kappa' \in \T_{n-2}\right\rbrace}
\end{equation*}
completes the proof.
\end{proof}

\begin{lemma}\label{first mu est}
In the setting of the proof of Proposition \ref{prop symbolic tree structure Mar30}, let $n \in \N$, $l \in \{ 1,\ldots,q_n \}$, and $\omega \in \T_n^l$. Then,
\begin{equation*}
\mu([\Psi_n^l(\omega)]) \leq q_{n-1}^2\#\D(n)(2KC_0)^{q_{n-1}+l} \varphi^t\bigl(\Psi_n^l(\omega) \bigr).
\end{equation*}
\end{lemma}

\begin{proof}
Note that $\mu([\Psi_n^l(\omega)]) =\nu \circ \Psi^{-1}\left([\Psi_n^l(\omega)]\right)$. Moreover,
$\Psi^{-1}\left([\Psi_n^l(\omega)]\right) \subset \bigcup[\eta]$,
where the union is taken over all $\eta \in \T_n^{l+1}$ satisfying $\Psi_n^{l}(\omega) \preccurlyeq \Psi_n^{l+1}(\eta)$. This follows from the fact that every $\eta \in \T_n^{l+1}$ maps to a string $\Psi_n^{l+1}(\eta)$ of length
$|\Psi_n^{l+1}(\eta)| \geq \gamma_n^{l+1}-K \geq \gamma_n^l \geq |\Psi_n^{l}(\omega)|$.
Since
\begin{equation*}
\nu([\eta]) \leq q_{n-1}C_0^{q_{n-1}+l+1} \varphi^t\bigl(\Psi_n^{l+1}(\eta) \bigr)
\leq q_{n-1}C_0^{q_{n-1}+l+1} \varphi^t\bigl(\Psi_n^{l}(\omega) \bigr)
\end{equation*}
for all $\eta \in \T_n^{l+1}$ satisfying $\Psi_n^{l}(\omega) \preccurlyeq \Psi_n^{l+1}(\eta)$, it suffices to show that
\begin{equation*}
\#\{ \eta  \in \T_n^{l+1}:\Psi_n^{l}(\omega) \preccurlyeq \Psi_n^{l+1}(\eta)\} \leq \#\D(n) q_{n-1}(2K)^{q_{n-1}+l}.
\end{equation*}
Now if $\eta  \in \T_n^{l+1}$ satisfies $\Psi_n^{l}(\omega) \preccurlyeq \Psi_n^{l+1}(\eta)$, then there exists some $\eta_- \in \T_n^l$ with $\eta_-\preccurlyeq \eta$ and $\Psi_n^{l}(\omega) \preccurlyeq \Psi_n^{l}(\eta_-)$ or $\Psi_n^{l}(\eta_-) \preccurlyeq \Psi_n^{l}(\omega)$. By Lemma \ref{count of phi overlaps}, there are at most $q_{n-1}(2K)^{q_{n-1}+l}$ such $\eta_-$. Moreover, each such $\eta_-$ is continued by at most $\#\D(n)$ strings in $\T_n^{l+1}$. This finishes the proof.
\end{proof}

\begin{lemma} \label{thm:previous_lemma}
In the setting of the proof of Proposition \ref{prop symbolic tree structure Mar30}, let $\tau \in \Sigma_*$. If $n=n(\tau)$ is minimal so that $|\tau|\leq \gamma^{l}_n-K$ for some $l \in \{ 1,\ldots,q_n \}$ and let $l=l(\tau)$ be the least such $l$. Then $|\tau|> \gamma^{l}_n/2$ and
\begin{equation*}
\mu([\tau]) \leq q_{n-1}^3(2KC_0)^{q_{n-1}+l} \varphi^t\left(\tau \right).
\end{equation*}
\end{lemma}

\begin{proof}
Since $|\tau| \leq \gamma^{l}_n-K$ every $\omega \in \T_n^l$ satisfies $|\tau| \leq |\Psi_n^{l}(\omega)|$. Hence
$\mu([\tau]) \leq \sum \mu\left(\left[\Psi_n^{l}(\omega)\right]\right)$,
where the sum is taken over all $\omega \in \T_n^l$ with $\tau  \preccurlyeq \Psi_n^{l}(\omega)$. By Lemma \ref{first mu est}, for each $\omega \in \T_n^l$ with $\tau  \preccurlyeq \Psi_n^{l}(\omega)$, we have
\begin{equation*}
\mu([\Psi_n^l(\omega)]) \leq q_{n-1}^2\#\D(n)(2KC_0)^{q_{n-1}+l} \varphi^t\bigl(\Psi_n^l(\omega) \bigr)
\leq q_{n-1}^2\#\D(n)(2KC_0)^{q_{n-1}+l} \varphi^t\left(\tau \right).
\end{equation*}
As such, we must estimate the number of $\omega \in \T_n^l$ with $\tau  \preccurlyeq \Psi_n^{l}(\omega)$. Either $l>1$, in which case $|\tau|>\gamma^{l-1}_n-K$, or $l=1$, in which case $|\tau|>\gamma^{q_{n-1}-1}_{n-1}-K$. In either case $|\tau| >\gamma_n^l- (5K+k(n)+k(n-1)+2)\geq \gamma_n^l/2$.
Now each $\omega \in \T_n^l$ with $\tau  \preccurlyeq \Psi_n^{l}(\omega)$ is of length no more than $\gamma_n^l$ and each of the final $|\Psi_n^{l}(\omega)|-|\tau|$ digits is chosen from, $\{1\}\cup \Gamma \cup \mathbb{D}(\D(n-1)) \cup \mathbb{D}(\D(n))$. Thus, there are at most
\begin{equation*}
\#\left(\{1\}\cup \Gamma \cup \mathbb{D}(\D(n-1)) \cup \mathbb{D}(\D(n))\right)^{5K+k(n)+k(n-1)+2}
\end{equation*}
words $\omega \in \T_n^l$ with $\tau  \preccurlyeq \Psi_n^{l}(\omega)$.

By the choice of $q_{n-1}$, we have
\begin{equation*}
q_{n-1}>\#\D(n)\#\left(\{1\}\cup \Gamma \cup \mathbb{D}(\D(n-1)) \cup \mathbb{D}(\D(n))\right)^{5K+k(n)+k(n-1)+2}
\end{equation*}
and the claim follows.
\end{proof}

\begin{lemma}
In the setting of the proof of Proposition \ref{prop symbolic tree structure Mar30}, there exists a constant $C \ge 1$ with $\mu([\tau]) \leq C \varphi^s(\tau)$ for all $\tau \in \Sigma_*$.
\end{lemma}

\begin{proof}
Clearly we may assume $\tau \in S^*$ since otherwise $\mu([\tau])=0$. Since each $\T_n$ consists of finitely many elements all of length at least $\gamma_n$, the set $S^* \cap \Sigma_m$ is finite for all $m \in \N$. As such, it suffices to show that there is $N\in \N$ so that $\mu([\tau]) \leq \varphi^s(\tau)$ for all $\tau \in S^*$ with $|\tau|> N$. Choose $N$ so that $(2KC_0)^{2/(N-1)}< (3/2)^{(t-s)}$ and $N^3 (3/4)^{N(t-s)}<1$ and let $M = \max\{N, \gamma_N\}$.

Given $\tau \in S^*$ with $|\tau|>N$ we let $n=n(\tau)$ to be minimal so that $|\tau|\leq \gamma^{l}_n-K$ for some $l \in \{ 1,\ldots,q_n \}$ and take $l=l(\tau)$ to be the least such $l$. Then, by Lemma \ref{thm:previous_lemma}, we have $|\tau|> \gamma^{l}_n/2$ and
\begin{equation*}
\mu([\tau]) \leq q_{n-1}^3(2KC_0)^{q_{n-1}+l} \varphi^t\left(\tau \right).
\end{equation*}
Hence $2|\tau|\geq \gamma_n^l \geq lk(n)+q_{n-1}k(n+1) \geq (l+q_{n-1})(n-1)$ and
\begin{equation*}
\mu([\tau]) \leq |\tau|^3(2KC_0)^{|\tau|/(n-1)} \varphi^t\left(\tau \right).
\end{equation*}
Since $\sup_{i \in \N}\| T_i \| < \tfrac12$ we have $\varphi^t\left(\tau \right) \leq 2^{-|\tau|(t-s)}\varphi^s\left(\tau \right)$ and so
\begin{equation*}
\mu([\tau]) \leq |\tau|^3\bigl((2KC_0)^{2/(n-1)}2^{-(t-s)}\bigr)^{|\tau|} \varphi^s\left(\tau \right).
\end{equation*}
Since $\gamma_n \geq \gamma_l^n \geq |\tau| >M \geq \gamma_N$ we have $n> N$, so $(2KC_0)^{2/(n-1)}(1/2)^{(t-s)}< (3/4)^{(t-s)}$. Moreover, since $|\tau|>N$ we have $|\tau|^3(3/4)^{|\tau|(t-s)}<1$ and
\begin{equation*}
\mu([\tau]) \leq |\tau|^3\bigl((2KC_0)^{2/(n-1)}(1/2)^{(t-s)}\bigr)^{|\tau|} \varphi^s\left(\tau \right)
\leq |\tau|^3(3/4)^{|\tau|(t-s)} \varphi^s\left(\tau \right) \leq \varphi^s\left(\tau \right)
\end{equation*}
finishing the proof.
\end{proof}

\begin{lemma}\label{S subset of E}
In the setting of the proof of Proposition \ref{prop symbolic tree structure Mar30}, $S \subset E_{\Phi}(\alpha)$.
\end{lemma}

\begin{proof}
Recall that we previously made the assumption that if $\phi_i$ is in the sequence $\Phi$, then also $-\phi_i$ is in $\Phi$. As such it suffices to fix $\phi_i$ and show that for each $\tau \in S$ we have
\begin{equation} \label{eq:S_maali}
\liminf_{m \rightarrow \infty}\frac{1}{m}\sum_{j=0}^{m-1}\phi_i(\sigma^j(\tau)) \geq \alpha_i.
\end{equation}
Given $m \geq \gamma_{i}$ we choose $n=n(m)$ to be maximal so that $\gamma_{n}\leq m$ and choose $l=l(m) \leq q_{n+1}$ to be maximal so that $\gamma_{n+1}^{l}\leq m$. Since $\tau \in S =\Psi(T)$ there is $\omega \in \T^l_{n+1}$ with $\Psi^l_{n+1}(\omega) \preccurlyeq \tau$. It follows from the construction of $\T^l_{n+1}$ that $\omega=\kappa\omega^1\omega^2$, where $\kappa \in \T_{n-1}$, $\omega^1 \in \T^{q_n}_n$ and $\omega^2 \in \T^l_n$. We deal with these three segments seperately.

Since $l$ is maximal we have $m<\gamma_{n+1}^{l+1}$, or $m<\gamma_{n+2}^{1}$ if $l=q_{n+1}$. This implies that $m-|\Psi_{n+1}^l(\omega)|\leq 2K+k(n+1)+k(n+2)+1$. It also follows that $\tau|_m$, the initial segment of $\tau$ of length $m$, consists entirely of digits from $\mathbb{D}^*(n)=\mathbb{D}\bigl( \bigcup_{l=1}^{n+2} \D(l)\cup\{1\}\cup \Gamma\bigr)$. Hence, for all $j \in \{ 1,\ldots,m \}$ we have
\begin{equation}\label{all j less m}
-\phi_i(\sigma^j(\omega)) \leq \A(n),
\end{equation}
where $\A(n)$ is as in \eqref{ A and B }.
Since $m \geq \gamma_n \geq q_n > n \A(n)(\gamma_{n-1}+2K+k(n+1)+k(n+2)+1)$, we thus get
\begin{align}
\sum_{j=0}^{|\Psi_{n-1}(\kappa)|-1}\phi_i(\sigma^j(\tau))+ \sum_{j=|\Psi_{n+1}^l(\omega)|}^{m-1}\phi_i(\sigma^j(\tau)) &\geq -\gamma_{n-1}\A(n) - (2K+k(n+1)+k(n+2)+1) \A(n) \notag \\ &\geq -\frac{2m}{n}. \label{error terms at start and end}
\end{align}

Observe that $q_n>L(n)$ and $\omega^1 \in \T^{q_n}_n$ imply
\begin{equation*}
\frac{1}{q_n}\sum_{j=0}^{q_n-1}\underline{A}_{k(n)}\phi_i(\sigma^{jk(n)}(\tilde{\omega}^1)) \in B_n(\alpha_i)
\end{equation*}
for all $\tilde{\omega}^1 \in [\omega^1]$. Here we have used the fact that $\underline{A}_{k(n)}\phi_i$ is constant on cylinders of length $k(n)$. Write $\omega^1$ in the form $\omega^1=\omega^1_1\cdots\omega^1_{q_n}$ where each $\omega^1_\nu \in \D(n) \subset \Sigma_{k(n)}$. It follows from the construction of $\Psi_n$, along with the fact that $\Psi_{n}(\kappa\omega^1) \preccurlyeq \tau$, that some set $\tilde{\mathbb{A}} \subset [\gamma_{n-1},\gamma_{n}]\cap \N$ of cardinality $q_n$ has the property that for each $j \in \tilde{\mathbb{A}}$ there is $\nu \in \{1, \ldots, q_n\}$ such that $\sigma^j\tau \in [\omega_\nu^1]$. Let $\mathbb{A}=\tilde{\mathbb{A}}+\{0,1,\ldots,k(n)-1\}$. We may choose $\tilde{\mathbb{A}}$ so that $\tau_{\nu} \in \{1\}\cup \mathbb{D}\left( \Gamma\right)$ for all integers $\nu \in [\gamma_{n-1}+1,\gamma_{n}] \setminus (\mathbb{A}+1)$. Since $\sigma^j(\tau) \in [\omega_{\nu}^1]$ for each $j \in \tilde{\mathbb{A}}$, we have
\begin{equation*}
\frac{1}{q_n k(n)} \sum_{j \in \mathbb{A}}\phi_i(\sigma^{j}(\tau)) \geq \frac{1}{q_n}\sum_{j\in \mathbb{A}}^{q_n-1}A_{k(n)}\phi_i(\sigma^{j}(\tau)) \\
\geq \frac{1}{q_n}\sum_{\nu=0}^{q_n-1}\underline{A}_{k(n)}\phi_i(\sigma^{\nu k(n)}(\tilde{\omega}^1)) \in B_n(\alpha_i).
\end{equation*}
By the construction of $\Psi_n$, the cardinality of $[\gamma_{n-1},\gamma_{n}-1] \setminus \mathbb{A}$ is at most $4Kq_n$, $k(n) \geq 4Kn$, and $m \geq \gamma_n \geq k(n)q_n \geq 4Kn q_n$. Thus $\#\mathbb{A} \geq \gamma_n -\gamma_{n-1}- 4Kq_n \geq \gamma_n - \frac{2m}{n}$ and
\begin{equation*}
\sum_{j\in \mathbb{A}}\phi_i(\sigma^j(\tau)) >
\begin{cases}
  \gamma_n  \left(\alpha_i-\frac{1}{n}\right), &\text{if }\alpha_i \leq \frac{1}{n},\\
  \left(\gamma_n - \frac{2m}{n}\right) \left(\alpha_i-\frac{1}{n}\right), &\text{if }\alpha_i > \frac{1}{n},\\
  \left(\gamma_n - \frac{2m}{n}\right) n, &\text{if } \alpha_i=\infty.
\end{cases}
\end{equation*}
As noted, for each $j \in [\gamma_{n-1},\gamma_{n}-1] \setminus \mathbb{A}$ there is $\eta \in \{1\}\cup \mathbb{D}\left( \Gamma\right)$ so that $\sigma^j\tau \in [\eta]$. Thus, for all such $j$, we have
$\phi_i(\sigma^j(\tau)) \geq -A_i$, where $A_i$ is as in \eqref{eq:A_i def}.
Moreover, since $k(n) \geq 4K n A_i$ for $i \in \{ 1,\ldots,n \}$ and $m \geq \gamma_n \geq k(n)q_n$, we get
\begin{equation*}
\sum_{j \in [\gamma_{n-1},\gamma_{n}-1] \setminus \mathbb{A}}\phi_i(\sigma^j(\tau)) \geq -4K A_i q_n \geq -\frac{m}{n}.
\end{equation*}
Putting these inequalities together we have
\begin{equation}\label{ main gamma n chunk }
\sum_{j=\gamma_{n-1}}^{\gamma_n-1}\phi_i(\sigma^j(\tau)) >
\begin{cases}
  \gamma_n  \left(\alpha_i-\frac{1}{n}\right)-\frac{m}{n}, &\text{if }\alpha_i \leq \frac{1}{n},\\
  \left(\gamma_n - \frac{2m}{n}\right) \left(\alpha_i-\frac{1}{n}\right)-\frac{m}{n}, &\text{if }\alpha_i > \frac{1}{n},\\
  \left(\gamma_n - \frac{2m}{n}\right) n-\frac{m}{n}, &\text{if } \alpha_i=\infty.
\end{cases}
\end{equation}

For the sum $\sum_{j=\gamma_n}^{|\Psi_{n+1}^l(\omega)|-1} \phi_i(\sigma^j(\tau))$, there are two cases. Either $l \geq L(n+1)$, in which case there exists $\tilde{\omega}^2 \in [\omega^2]$ with
\begin{equation}
\frac{1}{l}\sum_{j=0}^{l-1}\underline{A}_{k(n+1)}\phi_i(\sigma^{jk(n+1)}\tilde{\omega}^2) \in B_n(\alpha_i),
\end{equation}
so we may proceed as in the previous case to deduce
\begin{equation*}
\sum_{j=\gamma_n}^{|\Psi_{n+1}^l(\omega)|-1}\phi_i(\sigma^j(\tau)) >
\begin{cases}
  \left(|\Psi_{n+1}^l(\omega)|-\gamma_n \right) \left(\alpha_i-\frac{1}{n}\right)-\frac{m}{n}, &\text{if }\alpha_i \leq \frac{1}{n},\\
  \left(|\Psi_{n+1}^l(\omega)|-\gamma_n  - \frac{2m}{n}\right) \left(\alpha_i-\frac{1}{n}\right)-\frac{m}{n}, &\text{if }\alpha_i > \frac{1}{n},\\
  \left(|\Psi_{n+1}^l(\omega)|-\gamma_n  - \frac{2m}{n}\right) n-\frac{m}{n}, &\text{if } \alpha_i=\infty.
  \end{cases}
 \end{equation*}
Recall that $m-|\Psi_{n+1}^l(\omega)|\leq 2K+k(n+1)+k(n+2)+1 \leq \gamma_{n}/n \leq m/n$, so we may combine the above inequalities to obtain
\begin{equation}\label{ l large term }
\sum_{j=\gamma_n}^{|\Psi_{n+1}^l(\omega)|-1}\phi_i(\sigma^j(\tau)) >
\begin{cases}
  \left(m-\gamma_n \right) \left(\alpha_i-\frac{1}{n}\right)-\frac{m}{n}, &\text{if }\alpha_i \leq \frac{1}{n},\\
  \left(m-\gamma_n  - \frac{3m}{n}\right) \left(\alpha_i-\frac{1}{n}\right)-\frac{m}{n}, &\text{if }\alpha_i > \frac{1}{n},\\
  \left(m-\gamma_n  - \frac{3m}{n}\right) n-\frac{m}{n}, &\text{if } \alpha_i=\infty.
  \end{cases}
\end{equation}
Combining (\ref{ l large term }) with (\ref{error terms at start and end}) and (\ref{ main gamma n chunk }) we conclude that whenever $l \geq L(n+1)$, we have
\begin{equation}\label{first case}
\frac{1}{m}\sum_{j=0}^{m-1}\phi_i(\sigma^j(\tau)) >
\begin{cases}
  \left(\alpha_i-\frac{1}{n}\right)-\frac{3}{n}, &\text{if }\alpha_i \leq \frac{1}{n},\\
  \left(1  - \frac{3}{n}\right) \left(\alpha_i-\frac{1}{n}\right)-\frac{3}{n}, &\text{if }\alpha_i > \frac{1}{n},\\
  \left(1  - \frac{3}{n}\right) n-\frac{3}{n}, &\text{if } \alpha_i=\infty.
\end{cases}
\end{equation}
If $l \leq L(n+1)$, then we apply (\ref{all j less m}) once more to obtain
\begin{equation} \label{l small term}
\sum_{j=\gamma_n}^{|\Psi_{n+1}^l(\omega)|-1}\phi_i(\sigma^j(\tau)) \geq - k(n+1)L(n+1)\mathcal{A}(n) \geq -\frac{\gamma_n}{n}\geq - \frac{m}{n}.
\end{equation}
Notice that if $l \leq L(n+1)$, then $|\Psi_{n+1}^l(\omega)|-\gamma_n \leq (4K+k(n+1))L(n+1) \leq m/n$. We also have $m-|\Psi_{n+1}^l(\omega)|\leq m/n$, so $m-\gamma_n \leq 2m/n$ which combined with (\ref{ main gamma n chunk }) gives
\begin{equation}\label{second main gamma n chunk }
\sum_{j=\gamma_{n-1}}^{\gamma_n-1}\phi_i(\sigma^j(\tau)) >
\begin{cases}
  m \left(\alpha_i-\frac{1}{n}\right)-\frac{m}{n}, &\text{if }\alpha_i \leq \frac{1}{n},\\
  \left(m - \frac{4m}{n}\right) \left(\alpha_i-\frac{1}{n}\right)-\frac{m}{n}, &\text{if }\alpha_i > \frac{1}{n},\\
  \left(m - \frac{4m}{n}\right) n-\frac{m}{n}, &\text{if } \alpha_i=\infty.
\end{cases}
\end{equation}
Combining (\ref{second main gamma n chunk }) with (\ref{l small term}) and (\ref{error terms at start and end}) gives
\begin{equation}\label{second case}
\frac{1}{m}\sum_{j=0}^{m-1}\phi_i(\sigma^j(\tau)) >
\begin{cases}
  \left(\alpha_i-\frac{1}{n}\right)-\frac{3}{n}, &\text{if }\alpha_i \leq \frac{1}{n},\\
  \left(1 - \frac{4}{n}\right) \left(\alpha_i-\frac{1}{n}\right)-\frac{3}{n}, &\text{if }\alpha_i > \frac{1}{n},\\
  \left(1 - \frac{4}{n}\right) n-\frac{3}{n}, &\text{if } \alpha_i=\infty.
\end{cases}
\end{equation}
Since either (\ref{first case}) or (\ref{second case}) holds for all $m \ge \gamma_i$ and $n(m) \to \infty$ as $m \to \infty$ we have shown \eqref{eq:S_maali} and thus finished the proof.
\end{proof}

\section{Conditional variational principle for bounded potentials} \label{sec:conditional}

In this section we shall prove Theorem \ref{main theorem for bounded potentials}. The progression from Theorem \ref{General theorem} to Theorem \ref{main theorem for bounded potentials} relies upon the thermodynamic formalism developed in \S \ref{sec:gibbs}. The main challenge is to prove the upper bound which is given in \S \ref{sec:conditional_upper}--\ref{sec:conditional_completeupperbound}.

We begin by proving some elementary lemmas in \S \ref{sec: space of integrals}. The upper bound for the interior points of the spectrum for finitely many potentials is given in \S \ref{sec:conditional_upper}. In \S \ref{sec: uppersemicont lemma}, we prove an upper-semicontinuity lemma which is a crucial technical tool in forthcoming sections. In Section \S \ref{sec:conditional_finitely}, we prove a lemma which allows us to extend the upper bound to the boundary of the spectrum. The proof of the upper bound, for all points of the spectrum and for a countable infinity of potentials, is given in \S \ref{sec:conditional_completeupperbound}. The lower bound in Theorem \ref{main theorem for bounded potentials} follows reasonably straightforwardly from Theorem \ref{General theorem} and it is proved in \S \ref{sec:conditional_lower}.

\subsection{Space of integrals with respect to invariant measures}\label{sec: space of integrals}
We restrict our attention to potentials $\Phi \colon \Sigma \rightarrow \R^N$ taking values in some finite dimensional vector space. We begin by recalling an elementary lemma concerning convex sets of $\R^N$. If $\kappa \in \{-1,1\}^N$, then we define the open \emph{$\kappa$-orthant} $O(\kappa)$ to be the set
\begin{equation*}
  O(\kappa) = \{ (x_i)_{i=1}^N : \kappa_i \cdot x_i > 0 \text{ for each } i \}.
\end{equation*}

\begin{lemma}
\label{Convexity Lemma B}
  If $C$ is a convex set, then $a \in \R^N$ lies in the interior of $C$ if and only if $C \cap \left(a+O(\kappa)\right) \neq \emptyset$ for all $\kappa \in \{-1,1\}^N$.
\end{lemma}

Suppose $\Phi = (\phi_i)_{i=1}^N$ is bounded with summable variations. If $\nu \in \MM(\Sigma)$, then we write $\int \Phi d\nu = (\int \phi_1 d\nu,\ldots,\int \phi_N d\nu)$. The space of integrals with respect to invariant measures is $\mathcal{P}(\Phi) = \{ \int \Phi d\nu : \nu \in \M_{\sigma}(\Sigma) \} \subset \R^N$.

\begin{lemma}\label{first lemma in The space of integrals with respect to invariant measures}
  The set $\mathcal{P}(\Phi)$ is bounded and convex. Moreover, either $\mathcal{P}(\Phi)$ is contained within some $(N-1)$-dimensional hyperplane or $\mathcal{P}(\Phi) \subset \overline{\inter \left(\mathcal{P}(\Phi)\right)}$.
\end{lemma}

\begin{proof}
The first statement follows immediately from the fact that the mapping $\nu \mapsto \int \Phi d\nu$ defined on $\M_{\sigma}(\Sigma)$ is bounded and affine and $\M_{\sigma}(\Sigma)$ is convex. The second statement follows from elementary properties of convex sets in Euclidean spaces.
\end{proof}

If $I \subset \N$ is finite, then we define $\mathcal{P}(\Phi,I) \subset \mathcal{P}(\Phi)$ by
\begin{align*}
\mathcal{P}(\Phi, I) &= \{ \Phi(\nu) : \nu \in \M_{\sigma}(\Sigma)\text{ and }\nu(I^{\N}) = 1  \}, \\
\mathcal{P}_e(\Phi, I) &= \{ \Phi(\nu) : \nu \in \M_{\sigma}(\Sigma) \text{ is ergodic and } \nu(I^{\N}) = 1 \}.
\end{align*}

\begin{lemma}\label{second lemma in The space of integrals with respect to invariant measures}
It holds that
\begin{align*}
  \mathcal{P}(\Phi) &\subset \overline{\bigcup_{I} \mathcal{P}_e(\Phi,I)}, \\
  \inter \left(\mathcal{P}(\Phi)\right) &\subset \bigcup_{I}\inter \left( \mathcal{P}(\Phi,I)\right),
\end{align*}
where the unions are taken over all finite subsets $I \subset \N$.
\end{lemma}

\begin{proof}
Let $\nu \in \MM_\sigma(\Sigma)$, $\alpha = \int\Phi d\nu \in \mathcal{P}(\Phi)$ and $\varepsilon>0$. Since each $\phi_i$ has summable variations we may choose $n \in \N$ with $\var_n(A_n\phi_i)< \eps$. For each $\omega \in \N^n$ we let $\tilde{\omega} \in \Sigma$ denote the unique periodic point with $\sigma^{qn}(\tilde{\omega}) \in [\omega]$ for all $q \in \N \cup \{0\}$.

Note that since $\nu$ is $\sigma$-invariant, $\int A_n\phi_i d \nu=\int \phi_i d \nu =\alpha_i$ for each $i \in \{ 1,\ldots,N \}$. Hence, as $\var_n(A_n\phi_i)< \eps$ for each $i$, we have
\begin{equation}\label{ineq Mar19 1}
\begin{split}
  \sum_{\omega \in \Sigma_n} \nu([\omega])\inf_{\tau}\left\lbrace A_n\phi_i(\tau)\right\rbrace &> \alpha_i - \eps, \\
  \sum_{\omega \in \Sigma_n} \nu([\omega])\sup_{\tau}\left\lbrace A_n\phi_i(\tau)\right\rbrace &< \alpha_i+\eps.
\end{split}
\end{equation}
Given a finite set $I \subset \N$ we let $c(\nu,I) = \sum_{\omega \in I^n}\nu([\omega])$. Note that $I$ may be chosen so that $c(\nu,I)$ is arbitrarily close to one. Hence, (\ref{ineq Mar19 1}) implies that there exists a finite set $I \subset \N$ such that
\begin{equation}\label{ineq Mar19 2}
\begin{split}
c(\nu,I)^{-1} \sum_{\omega \in I^n} \nu([\omega])\inf_{\tau}\left\lbrace A_n\phi_i(\tau)\right\rbrace &> \alpha_i - \eps, \\
c(\nu,I)^{-1} \sum_{\omega \in I^n} \nu([\omega])\sup_{\tau}\left\lbrace A_n\phi_i(\tau)\right\rbrace &< \alpha_i+\eps.
\end{split}
\end{equation}
for all $i \in \{ 1,\ldots,N \}$.
Let $\mu'$ be the unique $n$th level Bernoulli measure which satisfies $\mu'([\omega]) = c(\nu,I)^{-1} \nu([\omega])$ for all $\omega \in I^n$. By (\ref{ineq Mar19 2}), we have
\[
  \int A_n\phi_i d \mu' \in  \left(\alpha_i-\eps, \alpha_i+\eps\right).
\]
Now let $\mu = \frac{1}{n}\sum_{j=0}^{n-1}\mu' \circ \sigma^{-j}$. Since $\mu'$ is $\sigma^n$-invariant and ergodic with respect to $\sigma^n$, the measure $\mu$ is $\sigma$-invariant and ergodic with respect to $\sigma$. It is also clear that $\mu$ is supported on $I^{\N}$. Moreover, since
\begin{equation*}
\int \phi_i d \mu = \frac{1}{n} \sum_{j=0}^{n-1} \int \phi_i d \mu' \circ \sigma^{-j} = \int A_n\phi_i d \mu' \in \left(\alpha_i-\eps, \alpha_i+\eps\right)
\end{equation*}
for all $i \in \{ 1,\ldots,N \}$, we have shown the first claim.

To prove the second claim, we apply Lemma \ref{Convexity Lemma B}. Indeed, if $\alpha \in \inter \left(\mathcal{P}(\Phi)\right)$, then $\mathcal{P}(\Phi) \cap \left(\alpha+O(\kappa)\right) \neq \emptyset$ for all $\kappa \in \{-1,1\}^N$. Since each set $\alpha+O(\kappa)$ is open it follows from the first claim that for each $\kappa \in \{-1,1\}^N$ there is a finite set $I(\kappa) \subset \N$ with  $\mathcal{P}(\Phi,I(\kappa)) \cap \left(\alpha+O(\kappa)\right) \neq \emptyset$. Letting $I = \bigcup_{\kappa \in  \{-1,1\}^N} I(\kappa)$, we obtain a finite set with
$\mathcal{P}(\Phi,I) \cap \left(\alpha+O(\kappa)\right) \neq \emptyset$ for all $\kappa \in \{-1,1\}^N$. Moreover, since $\mathcal{P}(\Phi, I)$ is convex it follows from Lemma \ref{Convexity Lemma B} that $\alpha \in \inter \left( \mathcal{P}(\Phi, I)\right)$. This completes the proof.
\end{proof}

\subsection{Upper bound for interior points of the spectrum} \label{sec:conditional_upper}

In this section we give the proof of the upper bound in Theorem \ref{main theorem for bounded potentials} for interior points of the spectrum in the special case where there are only finitely many potentials.

\begin{prop}\label{special case of main theorem for bounded potentials - finite, interior}
If $(T_i)_{i \in \N} \in \GL^\N$ is such that $\sup_{i \in \N} \| T_i \| < 1$, the singular value function $\varphi^s$ is quasi-multiplicative for all $0 \le s \le d$, $\Phi \colon \Sigma \to \R^{N}$ is bounded with summable variations, and $\alpha \in \inter(\PP(\Phi))$, then
\begin{equation*}
  \dimh(J^\mathbf{a}_\Phi(\alpha)) \leq \min\bigl\{ d,\max\bigl\{ s_\infty, \sup\{ D(\mu) : \mu \in \MM_\sigma(\Sigma) \text{ so that } \int \Phi d\mu = \alpha \} \bigr\} \bigr\}
\end{equation*}
for all $\mathbf{a} \in \mathbf{A}$.
\end{prop}
The proof of Proposition \ref{special case of main theorem for bounded potentials - finite, interior} requires two lemmas. Lemma \ref{relate dimension and pressure lemma} uses Lemma \ref{initial covering lemma Mar15} to relate the dimension of $J_{\Phi}(\alpha)$ to the pressure. Then Lemma \ref{thm:upper_bound_for_inter_points} proves the upper bound by showing the existence of an appropriate maximising measure.

\begin{lemma}\label{relate dimension and pressure lemma}
  If $(T_i)_{i \in \N} \in \GL^\N$ is such that $\sup_{i \in \N}\| T_i \| < 1$, $\mathbf{a} \in \mathbf{A}$, $\Phi \colon \Sigma \to \R^N$ is bounded with summable variations, $\alpha \in \R^N$, and $s< \dimh(J^{\mathbf{a}}_{\Phi}(\alpha))$, then
  \[
    P(\varphi^s \cdot e_{\langle q, \Phi - \alpha \rangle} ) \geq 0
  \]
  for all $q \in \R^N$.
\end{lemma}

\begin{proof}
Fix $s < \dimh(J_{\Phi}(\alpha))$ and $q \in \R^N$. By Lemma \ref{initial covering lemma Mar15}, for each $n \in \N$ there exists
$k(n) \in \N$ such that
\begin{eqnarray*}
\sum_{\omega \in A_{\Phi}(\alpha,n,k)} \varphi^s(\omega) > 1
\end{eqnarray*}
for all $k \geq k(n)$. Thus
\begin{align*}
  \sum_{\omega \in \Sigma_k}\varphi^s(\omega) \sup_{\tau \in [\omega]}\exp\left(S_n\left\langle q, \Phi - \alpha \right\rangle(\tau)\right)
  &= \sum_{\omega \in \Sigma_k} \varphi^s(\omega) \sup_{\tau \in [\omega]} \exp\biggl(\sum_{i=1}^N q_i\left(S_n\phi_i - n\alpha_i\right)\biggr) \\
  &\geq \sum_{\omega \in A_{\Phi}(\alpha,n,k)}\varphi^s(\omega) \sup_{\tau \in [\omega]} \exp\biggl(\sum_{i=1}^N q_i\left(S_n\phi_i - n\alpha_i\right)\biggr) \\
  &\geq \sum_{\omega \in A_{\Phi}(\alpha,n,k)}\varphi^s(\omega) \cdot e^{-N\overline{q}k/n}>e^{-N\overline{q}k/n},
\end{align*}
where $\overline{q} = \max_{i \in \{ 1,\ldots,N \}} q_i$.
Hence
\begin{equation*}
P(\varphi^s \cdot e_{\left\langle q, \Phi - \alpha \right\rangle} ) = \lim_{k \rightarrow \infty} \tfrac{1}{k} \log \biggl(
\sum_{\omega \in \Sigma_k}\varphi^s(\omega) \sup_{\tau \in [\omega]} \exp\left(S_n\left\langle q, \Phi - \alpha \right\rangle(\tau)\right)\biggr)
\geq -\frac{N\overline{q}}{n}.
\end{equation*}
Letting $n\rightarrow \infty$ completes the proof of the lemma.
\end{proof}

\begin{lemma} \label{thm:product_gibbs_lyapunov}
  If $(T_i)_{i \in \N} \in \GL^\N$ is such that $\sup_{i \in \N}\| T_i \| < 1$, the singular value function $\varphi^s$ is quasi-multiplicative for all $0\le s \le d$, $\Phi \colon \Sigma \to \R^N$ is bounded with summable variations, $\alpha \in \R^N$, $q \in \R^N$, and $s > s_\infty$, then the potential $\varphi^s \cdot e_{\langle q,\Phi-\alpha \rangle}$ is quasi-multiplicative and $P(\varphi^s \cdot e_{\langle q,\Phi-\alpha \rangle}) < \infty$. Moreover, if $\mu$ is the Gibbs measure for $\varphi^s \cdot e_{\langle q,\Phi-\alpha \rangle}$, then $\Lambda_\mu(\varphi^s) > -\infty$.
\end{lemma}

\begin{proof}
  Observe that the quasi-multiplicativity follows immediately from Lemma \ref{compound quasi-multiplicative}. Since $\Phi$ is bounded we have $B = \sup\{ |\langle q,\Phi(\omega)-\alpha \rangle| : \omega \in \Sigma \} < \infty$. This together with $P(\varphi^s) < \infty$ gives $P(\varphi^s \cdot e_{\langle q,\Phi-\alpha \rangle}) = P < \infty$. Thus, by Theorems \ref{thm:gibbs_exists} and \ref{Gibbs implies ergodicity}, the Gibbs measure $\mu$ for $\varphi^s \cdot e_{\langle q,\Phi-\alpha \rangle}$ exists.

  To prove the last claim, let $m \in \Z$ be so that $m < s \le m+1$. By the Gibbs property of $\mu$ there is a constant $C \ge 1$ so that
  \begin{equation*}
    \mu([\omega]) \le C\varphi^s(\omega) e_{\langle q,\Phi-\alpha \rangle}(\omega) e^{-nP} \le C\varphi^s(\omega) e^{n(B-P)}
  \end{equation*}
  for all $\omega \in \Sigma_n$ and $n \in \N$. Now, following the proof of Lemma \ref{thm:lyapunov_finite_for_gibbs}, we get $\Lambda_\mu(\varphi^s) > -\infty$.
\end{proof}

\begin{lemma} \label{thm:upper_bound_for_inter_points}
  If $(T_i)_{i \in \N} \in \GL^\N$ is such that $\sup_{i \in \N}\| T_i \| < 1$, the singular value function $\varphi^s$ is quasi-multiplicative for all $0\le s \le d$, $\Phi \colon \Sigma \to \R^N$ is bounded with summable variations, $\alpha \in \inter(\PP(\Phi)) \subset \R^N$, and $s > s_\infty$ satisfies
  \begin{equation*}
    \inf_{q \in \R^N}P(\varphi^s \cdot e_{\left\langle q, \Phi - \alpha \right\rangle}) \geq 0,
  \end{equation*}
  then there exists an ergodic invariant measure $\mu \in \M_{\sigma}(\Sigma)$ with $\int \Phi d \mu=\alpha$ and $D(\mu) \geq s$.
\end{lemma}

\begin{proof}
We shall consider the function $F \colon \R^N \rightarrow \R$ defined by $F(q) = P(\varphi^s \cdot e_{\left\langle q, \Phi - \alpha \right\rangle})$.
Since $\alpha \in \inter\left(\mathcal{P}(\Phi)\right)$ we may apply the latter claim of Lemma \ref{second lemma in The space of integrals with respect to invariant measures} to obtain a finite subset $I \subset \N$ with $\alpha \in \inter\left(\mathcal{P}(\Phi,I)\right)$. Since $I$ is finite and all of the matrices $T_i$ are non-singular we have $c = \min\{ \gamma_d(T_i) : i \in I \} > 0$, $\varphi^s(\omega) \ge c^n$ for all $\omega \in \Sigma_n$, and $\Lambda_\mu(\varphi^s) \ge \log c > -\infty$ for all $\mu \in \MM_\sigma(\Sigma)$ with $\mu(I^\N)=1$. Since $\alpha \in \inter\left(\mathcal{P}(\Phi,I)\right)$ there exists $\eps >0$ with $\overline{B_{\eps}(\alpha)} \subset \mathcal{P}(\Phi,I)$. Hence for each $q \in \R^N \setminus \{0\}$ we have $\alpha + \eps q/\|q\| \in \mathcal{P}(\Phi,I)$ and there exists a measure $\nu_q \in \M_{\sigma}(\Sigma)$ with $\Lambda_{\nu_q}(\varphi^s) \geq \log c$ satisfying
\begin{equation*}
\int \Phi d\nu_q = \alpha + \eps q/\|q\|.
\end{equation*}
By Lemma \ref{variational principal upper bound} and the boundedness of $\Phi$, we have
\begin{equation*}
  F(q) \ge h_{\nu_q} + \bigl\langle q, \int\Phi d\nu_q - \alpha \bigr\rangle + \Lambda_{\nu_q}(\varphi^s) \ge \langle q, \eps q/\|q\| \rangle + \log c = \eps\|q\| + \log c.
\end{equation*}
Hence $F(q) \rightarrow \infty$, as $\|q\| \rightarrow \infty$. It follows from Lemma \ref{convexity of the pressure function} that $F$ attains a global minimum on a bounded set. Let $q(\alpha)$ denote a point at which this global minimum is attained.

By Lemma \ref{thm:product_gibbs_lyapunov}, the Gibbs measure $\mu_q$ for $\varphi^s \cdot e_{\langle q,\Phi-\alpha \rangle}$ satisfies $\Lambda_{\mu_q}(\varphi^s) > -\infty$. Thus, applying Lemmas \ref{thm:gibbs_equilibrium} and \ref{variational principal upper bound}, we obtain
\begin{equation*}
  F(q) = h_{\mu_q} + \bigl\langle q, \int\Phi d\mu_q - \alpha \bigr\rangle + \Lambda_{\mu_q}(\varphi^s).
\end{equation*}
Moreover, by Lemma \ref{differentiation of pressure} for each $i \in \{1, \ldots, N \}$ we have
\begin{align*}
\frac{\partial F(q)}{\partial q_i}\bigg|_{q=q_*} &= \lim_{q_i \rightarrow q_i^{*}}\frac{P\bigl(\varphi^s\cdot e_{\left\langle q_*,\Phi-\alpha\right\rangle}\cdot e_{(q_i-q_i^*)(\phi_i-\alpha_i)}\bigr)-P\left(\varphi^s\cdot e_{\left\langle q_*,\Phi-\alpha\right\rangle}\right)}{q_i-q_i^*}\\
&= \frac{ \partial P\left((\varphi^s\cdot e_{\left\langle q_*,\Phi-\alpha\right\rangle})\cdot e_{q_i(\phi_i-\alpha_i)}\right)}{\partial q_i}\bigg|_{q_i=0} = \int \phi_i d \mu_{q_*}- \alpha_i.
\end{align*}
Since $F$ attains a minimum at $q(\alpha)$ it follows that $\int \phi_i d \mu_{q(\alpha)}= \alpha_i$ for all $i \in \{ 1,\ldots, N \}$. Thus, denoting $\mu = \mu_{q(\alpha)}$, we have $\int \Phi d\mu = \alpha$,
\begin{equation*}
  h_\mu + \Lambda_\mu(\varphi^s) = h_\mu + \bigl\langle q(\alpha), \int \Phi d \mu - \alpha \bigr\rangle + \Lambda_\mu(\varphi^s) = P(\varphi^s \cdot e_{\left\langle q, \Phi - \alpha \right\rangle}) \geq 0,
\end{equation*}
and $D(\mu)\geq s$.
\end{proof}

\begin{proof}[Proof of Proposition \ref{special case of main theorem for bounded potentials - finite, interior}]
Take $\alpha \in \inter(\PP(\Phi))$. Either $\dimh \left(J_{\Phi}^{\mathbf{a}}(\alpha)\right)\leq s_{\infty}$, in which case the upper bound holds, or $\dimh \left(J_{\Phi}^{\mathbf{a}}(\alpha)\right)> s_{\infty}$. If $\dimh \left(J_{\Phi}^{\mathbf{a}}(\alpha)\right) > s_{\infty}$, then we may choose $s_\infty < s < \dimh \left(J_{\Phi}^{\mathbf{a}}(\alpha)\right)$. By Lemmas \ref{relate dimension and pressure lemma} and \ref{thm:upper_bound_for_inter_points} there exists $\mu \in \M_{\sigma}(\Sigma)$ with $D(\mu) \geq s$. This completes the proof of the proposition.
\end{proof}

\subsection{Quasi upper-semicontinuity lemma}\label{sec: uppersemicont lemma}
Recall that because of the non-compactness of the shift space, the space of invariant probability measures is non-compact and the entropy is not upper semi-continuous. Nonetheless we do have the following proposition.

\begin{prop}\label{Upper-semicont prop Mar15}
  If $(T_i)_{i \in \N} \in \GL^\N$ is such that $\sup_{i \in \N}\| T_i \| < 1$, $(\mu_n)_{n \in \N}$ is a sequence with $\mu_n \in \MM_\sigma(\Sigma)$ for all $n \in \N$ and $\limsup_{n \rightarrow \infty} D(\mu_n) > s_{\infty}$, then there exists a sub-sequence $(\mu_{n_j})_{j \in \N}$ and a measure $\mu \in \MM_\sigma(\Sigma)$ which is a weak$^*$ limit point of $\left(\mu_{n_j}\right)_{j \in\N}$ and satisfies $D(\mu) \geq \limsup_{n \rightarrow \infty}D(\mu_n)$.
\end{prop}

Before proving the proposition, we first prove a few elementary lemmas. Let $\mathbb{P}$ be the set of all infinite probability vectors, that is, $\mathbb{P} = \left\lbrace (q_i)_{i \in \N} \in [0,1]^{\N}: \sum_{i=1}^\infty q_i=1 \right\rbrace$, and equip it with the usual product topology. If $a = (a_i)_{i \in \N}$ is a sequence of numbers in $(0,1)$ and $C>0$, then we set $\mathbb{P}(a,C) = \{ (q_i)_{i \in \N} \in \mathbb{P} : \sum_{i \in \N} q_i \log a_i \ge -C \}$.

\begin{lemma} \label{thm:P_closed}
  If $a = (a_i)_{i \in \N}$ is a sequence of numbers in $(0,1)$ and $C>0$, then $\mathbb{P}(a,C)$ is closed.
\end{lemma}

\begin{proof}
  Let $p=(p_i)_{i \in \N}$ be an accumulation point of $\mathbb{P}(a,C)$ and let $(p(n))_{n \in \N}$ be a sequence so that $p(n) = (p_i(n))_{i \in \N} \in \mathbb{P}(a,C)$ for all $n \in \N$ and $\lim_{n \to \infty} p(n) = p$. Suppose for a reductio that $\sum_{i=1}^\infty p_i \log a_i < -C$. Then there exists $k \in \N$ so that $\sum_{i=1}^k p_i \log a_i < -C$. Choosing now $n_0 \in \N$ so that $\sum_{i=1}^k p_i(n) \log a_i < -C$ for all $n \ge n_0$, we have arrived at a contradiction since $\sum_{i=1}^\infty p_i(n) \log a_i \le \sum_{i=1}^k p_i(n) \log a_i < -C$ for $n \ge n_0$.
\end{proof}

\begin{lemma}\label{sequences limit lemma Mar25}
  If $a = (a_i)_{i \in \N}$ is a non-increasing sequence of numbers in $(0,1)$ with $\sum_{i=1}^\infty a_i < \infty$ and $C>0$, the the function $F \colon \mathbb{P}(a,C) \to \overline{\R}$, defined by $F(q) = \sum_{i=1}^\infty q_i\log(a_i/q_i)$, is upper semi-continuous.
\end{lemma}

\begin{proof}
Let $\mathbb{P}_k = \bigl\{ (q_i)_{i \in \N} \in \mathbb{P} : q_i=0 \text{ for all } i \in \{ k+1,k+2,\ldots \} \bigr\}$ for all $k \in \N$ and define a map $\xi_k \colon \mathbb{P} \rightarrow \mathbb{P}_k$ by setting
\begin{equation*}
  \xi_k(q)_i =
  \begin{cases}
    q_i,                   &\text{if } i \in \{ 1,\ldots,k-1 \}, \\
    \sum_{j=k}^\infty q_j, &\text{if } i=k, \\
    0,                     &\text{if } i \in \{ k+1,k+2,\ldots \},
\end{cases}
\end{equation*}
for all $q = (q_i)_{i \in \N} \in \mathbb{P}$. Take a sequence of vectors $(p(n))_{n \in \N}$ with each $p(n) = (p_i(n))_{i \in \N} \in \mathbb{P}(a,C)$ along with $p = (p_i)_{i \in \N} \in \mathbb{P}$ such that for each $i \in \N$ we have $\lim_{n \to \infty} p_i(n) = p_i$. In particular, we have $\sum_{i=1}^\infty p_i(n)\log a_i \ge -C$. Our goal is to show that $\limsup_{n \to \infty} F(p(n)) \le F(p)$.

Since $\lim_{n \rightarrow \infty}p_i(n)=p_i$ for all $i \in \{ 1,\ldots,k-1 \}$ and $\sum_{i=k}^\infty q_i= 1-\sum_{i=1}^{k-1} q_i$ for all $q \in \mathbb{P}$ we have $\lim_{n \rightarrow \infty} \xi_k(p(n)) = \xi_k(p)$ with respect to the supremum metric. Hence $\lim_{n \rightarrow \infty}F(\xi_k(p(n))) = F(\xi_k(p))$. Similarly, as in \eqref{eq:jensen_eq_calc}, we see that
\begin{equation*}
  \sum_{i=k}^\infty q_i\biggl( \log\frac{a_i}{q_i} - \log\sum_{j=k}^\infty a_j \biggr) \le -\sum_{i=k}^\infty q_i \log\sum_{j=k}^\infty q_j,
\end{equation*}
and, consequently,
\begin{equation*}
  F(q) \le \sum_{i=1}^{k-1} q_i\log\frac{a_i}{q_i} + \sum_{i=k}^\infty q_i\log\frac{\sum_{j=k}^\infty a_j}{\sum_{j=k}^\infty q_j} = F(\xi_k(q)) - \sum_{i=k}^\infty q_i\log a_k + \sum_{i=k}^\infty q_i\log\sum_{j=k}^\infty a_j
\end{equation*}
for all $q \in \mathbb{P}(a,C)$ and $k \in \N$. Choosing $k_0 \in \N$ so that $\sum_{j=k}^\infty a_j < 1$ for all $k \ge k_0$, this implies
\begin{align*}
  F(\xi_k(p)) &= \lim_{n \to \infty} F(\xi_k(p(n))) \ge \limsup_{n \to \infty}\bigl( F(p(n)) + \sum_{i=k}^\infty p_i(n)\log a_k \bigr) \\
  &\ge \limsup_{n \to \infty} F(p(n)) + \liminf_{n \to \infty} \sum_{i=k}^\infty p_i(n)\log a_k \ge \limsup_{n \to \infty} F(p(n)) + \sum_{i=k}^\infty p_i\log a_k
\end{align*}
for all $k \ge k_0$ by Fatou's lemma. Note that since the sequence $(a_i)_{i \in \N}$ is non-increasing we have $\sum_{i=k}^\infty p_i\log a_k \ge \sum_{i=k}^\infty p_i\log a_i$ for all $k \in \N$. Moreover, let $\eps>0$ and, by recalling Lemma \ref{thm:P_closed}, choose $k_\eps \in \N$ so that $\sum_{i=k}^\infty p_i\log a_i > -\eps$ and $\sum_{i=k}^\infty p_i\log\sum_{j=k}^\infty p_j > -\eps$ for all $k \ge k_\eps$. Since
\begin{equation*}
  F(p)-F(\xi_k(p)) = \sum_{i=k}^\infty p_i\log\frac{a_i}{p_i} - \sum_{i=k}^\infty p_i\log\frac{a_k}{\sum_{j=k}^\infty p_j} \ge \sum_{i=k}^\infty p_i\log a_i + \sum_{i=k}^\infty p_i\log\sum_{j=k}^\infty p_j > -2\eps
\end{equation*}
for all $k \ge k_\eps$ we have
\begin{equation*}
  \limsup_{n \to \infty} F(p(n)) \le \limsup_{k \to \infty}\bigl( F(\xi_k(p)) - \sum_{i=k}^\infty p_i\log a_k \bigr) \le \limsup_{k \to \infty} F(\xi_k(p)) + \eps \le F(p) + 3\eps.
\end{equation*}
Letting $\eps \downarrow 0$ finishes the proof.
\end{proof}

\begin{proof}[Proof of Proposition \ref{Upper-semicont prop Mar15}]
We begin by showing that $(\mu_n)_{n \in \N}$ has a convergent subsequence. Let $\delta = \limsup_{n \to \infty} D(\mu_n)$ and choose $m<\delta\le m+1$. If $\max\{ s_\infty,m \} < t_0<t_1<\delta$, then there exists a subsequence $(n_j)_{j \in \N}$ with $D(\mu_{n_j}) > t_1$ for all $j \in \N$ and $\lim_{j \to \infty} D(\mu_{n_j}) = \delta$. It follows that $\Lambda_{\mu_{n_j}}(\varphi^{t_0}) \ge \Lambda_{\mu_{n_j}}(\varphi^{t_1}) > -\infty$ and
\begin{equation*}
  0 \le P_{\mu_{n_j}}(\varphi^{t_1}) \le \tfrac{1}{k}\sum_{\omega \in \Sigma_k} \mu_{n_j}([\omega]) \log\frac{\varphi^{t_1}(\omega)}{\mu_{n_j}([\omega])}
\end{equation*}
for all $k,j \in \N$. Furthermore, recalling \eqref{eq:jensen_eq_calc} and Lemma \ref{positive recurrence}, we have
\begin{equation*}
  \sum_{\omega \in \Sigma_k} \mu_{n_j}([\omega]) \log\frac{\varphi^{t_0}(\omega)}{\mu_{n_j}([\omega])} \le \log Z_k(\varphi^{t_0}) \le k\log Z_1(\varphi^{t_0}) < \infty
\end{equation*}
for all $k,j \in \N$. Since $\varphi^{t_1}(\omega) = \varphi^{t_0}(\omega)\gamma_{m+1}(\omega)^{t_1-t_0}$ we get
\begin{equation} \label{eq:ala_arvio1}
  \sum_{\omega \in \Sigma_k} \mu_{n_j}([\omega]) \log\gamma_{m+1}(\omega) \ge -\frac{k\log Z_1(\varphi^{t_0})}{t_1-t_0}
\end{equation}
for all $k,j \in \N$. Note that for every $\eps>0$ there is $M \in \N$ so that $\sum_{i=M}^\infty \gamma_{m+1}(T_i)^{m+1} \le \sum_{i=M}^\infty \varphi^{t_0}(T_i) < \eps$. Thus for each $\eps>0$ there are only finitely many $i$'s so that $\log\gamma_{m+1}(T_i) \ge (m+1)^{-1}\log\eps$. Therefore, \eqref{eq:ala_arvio1} implies that the sequence $(\mu_{n_j})_{j \in \N}$ is tight and thus has a converging subsequence. We keep denoting the subsequence by $(\mu_{n_j})_{j \in \N}$ and let $\mu \in \MM_\sigma(\Sigma)$ be its weak$^*$ limit.

Let $\max\{ s_\infty,m \} < s < \delta$. Since $\varphi^s(\omega) \ge \gamma_{m+1}(\omega)^{m+1}$ for all $\omega \in \Sigma_k$ it follows from \eqref{eq:ala_arvio1} that
\begin{equation*}
  \sum_{\omega \in \Sigma_k} \mu_{n_j}([\omega]) \log\varphi^s(\omega) \ge -\frac{k(m+1)\log Z_1(\varphi^{t_0})}{t_1-t_0}
\end{equation*}
for all $k,j \in \N$. According to Lemma \ref{thm:P_closed}, the same estimate holds when the measure $\mu_{n_j}$ is replaced by $\mu$. Thus $\Lambda_\mu(\varphi^s) \ge -(m+1)(t_1-t_0)^{-1}\log Z_1(\varphi^{t_0}) > -\infty$. Furthermore, since $s<\delta$ there is $j_0 \in \N$ so that $D(\mu_{n_j})>s$ for all $j \ge j_0$. Therefore, Lemma \ref{sequences limit lemma Mar25} implies
\begin{equation*}
  \sum_{\omega \in \Sigma_k} \mu([\omega]) \log\frac{\varphi^s(\omega)}{\mu([\omega])} \ge \limsup_{j \to \infty} \sum_{\omega \in \Sigma_k} \mu_{n_j}([\omega]) \log\frac{\varphi^s(\omega)}{\mu_{n_j}([\omega])} \ge 0
\end{equation*}
and $D(\mu) \ge s$. The proof is finished since $\max\{ s_\infty,m \} < s < \limsup_{n \to \infty} D(\mu_n)$ was arbitrary.
\end{proof}

\subsection{Finitely many potentials lemma} \label{sec:conditional_finitely}
In this section we prove a technical lemma which, together with Proposition \ref{Upper-semicont prop Mar15}, allows us to prove the upper bound in Theorem \ref{main theorem for bounded potentials} for boundary points of the spectrum.

\begin{lemma}\label{finitely many potentials lemma}
  If $(T_i)_{i \in \N} \in \GL^\N$ is such that $\sup_{i \in \N}\| T_i \| < \tfrac12$, the singular value function $\varphi^s$ is quasi-multiplicative for all $0\le s \le d$, $\Phi \colon \Sigma \to \R^N$ is bounded with summable variations, $\PP(\Phi)$ is not contained within any $(N-1)$-dimensional hyperplane, and $\alpha \in \overline{\PP(\Phi)}$, then for each $\eps>0$ there is $\gamma \in \inter(\PP(\Phi))$ with $|\alpha - \gamma| < \eps$ and $\dimh(J_\Phi^\mathbf{a}(\gamma)) \ge \dimh(J_\Phi^\mathbf{a}(\alpha)) - \eps$ for $\LL_\mathbf{A}$-almost all $\mathbf{a} \in \mathbf{A}$.
\end{lemma}

\begin{proof}
Fix $\eps>0$ and let $\dimh(J_{\Phi}^{\mathbf{a}}(\alpha)) - \eps < s < t < \dimh(J_{\Phi}^{\mathbf{a}}(\alpha))$. By Lemma \ref{first lemma in The space of integrals with respect to invariant measures} and the first part of Lemma \ref{second lemma in The space of integrals with respect to invariant measures}, we may choose $\beta \in \text{int} \left( \mathcal{P}(\Phi)\right)\cap \mathcal{P}_e(\Phi,I)$, with respect to some finite subset $I \subset \N$, satisfying $|\beta-\alpha| < 1/n$.
Since $\beta \in \mathcal{P}_e(\Phi,I)$ there is an ergodic invariant measure $\nu \in \M_{\sigma}(\Sigma)$ with $\nu(I^{\N})=1$ and $\int \phi_i d\nu = \beta_i$ for all $i \in \{ 1,\ldots,N \}$. Since $\nu(I^{\N}) = 1$ we also have $\Lambda_\nu(\varphi^s) > -\infty$. By the sub-additive ergodic theorem there exist $\tau \in \Sigma$, a constant $\theta(\tau)>0$, and $L(\tau) \in \N$ such that
\begin{equation*}
  |A_l\phi_i(\tau)-\beta_i| < l^{-1} \text{ and } \varphi^s(\tau|_l) \geq \theta(\tau)^l
\end{equation*}
for all $l\geq L(\tau)$ and $i \in \{ 1,\ldots,N\}$. Choose $0 < \rho < \min\{ 1,\eps/|\beta-\alpha| \}$ so that
\begin{equation*}
  2^{(1-\rho)(t-s)}\theta(\tau)^{\rho} > 1
\end{equation*}
and let $\gamma = \rho \beta + (1-\rho) \alpha$. Since $\beta \in  \inter(\mathcal{P}(\Phi))$ and $\alpha \in \overline{\mathcal{P}(\Phi)}$ we have $\gamma \in  \inter(\mathcal{P}(\Phi))$ by the elementary properties of convex sets in $\R^N$. Moreover, since $\rho < \eps/|\beta-\alpha|$ we have $|\alpha - \gamma| < \eps$. We shall now show that $\dimh(J_{\Phi}^\mathbf{a}(\gamma)) \geq \dimh(J_{\Phi}^\mathbf{a}(\alpha)) - \eps$.

Since $t < \dimh(J_{\Phi}^\mathbf{a}(\gamma))$, it follows from Lemma \ref{initial covering lemma Mar15} that for all $l\in \N$ there exists $q(l) \in \N$ such that
\begin{equation*}
  \sum_{\kappa \in A_{\Phi}(\alpha,l,q)} \varphi^t(\kappa) > 1.
\end{equation*}
for all $q \geq q(l)$. Since $\sup_{i \in \N}\| T_i \| < \tfrac12$ and $s<t$ it follows that
\begin{equation*}
  \sum_{\kappa \in A_{\Phi}(\alpha,l,q)}\varphi^s(\kappa) > 2^{l(t-s)}.
\end{equation*}
for all $q \geq q(l)$. For every $\alpha, \beta \in \Sigma_*$, according to the quasi-multiplicativity of $\varphi^s$, there exists $\omega \in \Gamma$ such that
\begin{equation*}
  \varphi^s\left(\alpha\omega\beta \right) \geq c\varphi^s(\alpha) \varphi^s(\beta),
\end{equation*}
where $c>0$ is a constant depending only on $s$. As in \S \ref{sec:multifractal_tree_structure}, we write $\alpha \star \beta$ for $\alpha\omega\beta$. Note that for any given $\alpha,\beta \in \Sigma_*$ there are at most $K = \max\left\lbrace |\omega|: \omega \in \Gamma \right\rbrace$ finite words $\beta' \in \Sigma_*$ with $\alpha\star\beta' = \alpha\star\beta$

Our choice of $\rho$ implies that for each $l \in \N$ there exists $A(l) > \max\{ (1-\rho)^{-1}q(l), \rho^{-1}  L(\tau), l\}$ such that
\[
  c\theta(\tau) \bigl(2^{(1-\rho)(t-s)}\theta(\tau)^{\rho }\bigr)^k > 1
\]
for all $k\geq A(l)$. It follows that
\begin{equation}\label{Inqual 1 Mar2012}
\begin{split}
  \sum_{\kappa \in A_{\Phi}(\alpha,l,\lceil k(1-\rho)\rceil)} \varphi^s(\kappa \star (\tau|_{\lceil k\rho \rceil})) &\geq c \sum_{\kappa \in A_{\Phi}(\alpha,l,\lceil k(1-\rho)\rceil)}\varphi^s(\kappa) \varphi^s(\tau|_{\lceil k\rho \rceil}) \\
  &> c 2^{\lceil k(1-\rho)\rceil(t-s)}\theta(\tau)^{\lceil \rho k \rceil}
  > c \theta(\tau) \bigl(2^{(1-\rho)(t-s)}\theta(\tau)^{\rho }\bigr)^k > 1
\end{split}
\end{equation}
for all for all $k \geq A(l)$ and $l\in \N$.
We shall temporarily fix $l \geq d$ and $k \geq A(l)$, and for each $\kappa \in A_{\Phi}(\alpha,l,\lceil k(1-\rho)\rceil)$ we let $r(\kappa)$ be $|\kappa \star (\tau|_{\lceil k\rho \rceil})|-\lceil k\rho \rceil$. Since $[\hat{\kappa}] \subset [\kappa]$ for $\hat{\kappa} = (\kappa \star (\tau|_{\lceil k\rho \rceil}))|_k$ we have
\begin{equation*}
\lceil k(1-\rho)\rceil(\alpha_i-l^{-1}) < S_{\lceil k(1-\rho)\rceil}\phi_i(\omega) < \lceil k(1-\rho)\rceil(\alpha_i+l^{-1})
\end{equation*}
for each $i \in \{ 1,\ldots,N\}$ by the definition of $A_{\Phi}(\alpha,l,\lceil k(1-\rho)\rceil)$. We also have
\begin{equation*}
\lceil k\rho\rceil(\beta_i-l^{-1})-(K+1)\|\phi_i\| < S_{k-r}\phi_i(\sigma^r(\omega))
< \lceil k \rho \rceil(\beta_i+l^{-1})+(K+1)\|\phi_i\|.
\end{equation*}
Since $r-\lceil k(1-\rho)\rceil \leq K$ it follows that
\begin{align*}
\lceil k(1-\rho)\rceil(\alpha_i-l^{-1}) &+ \lceil k\rho\rceil(\beta_i-l^{-1})-(2K+1)\|\phi_i\| < S_{k}\phi_i(\omega) \\
&< \lceil k \rho \rceil(\beta_i+l^{-1})+(2K+1)\|\phi_i\| + \lceil k(1-\rho)\rceil(\alpha_i+l^{-1}).
\end{align*}
Furthermore, since $|\alpha_i|,|\beta|< \|\phi_i\|$, $\gamma_i=(1-\rho)\alpha_i+\rho \beta_i$, and $k\geq l$ we have
\begin{equation*}
\gamma_i-l^{-1}\left((2K+3)\|\phi_i\|+1\right)< A_{k}\phi_i(\omega) < \gamma_i+l^{-1}\left((2K+3)\|\phi_i\|+1\right).
\end{equation*}
Hence, if $Q = ((2K+3)\max_{i \in \{ 1,\ldots,N \}} \|\phi_i\| + 1)^{-1}$, then we have $\hat{\kappa} \in A_{\Phi}(\gamma,n, k)$ for all $n \leq \lfloor Ql\rfloor$.
Since $\hat{\kappa}$ is an initial substring of $\kappa\star(\tau|_{\lceil k\rho \rceil})$ for any given $\kappa \in A_{\Phi}(\alpha,l,\lceil k(1-\rho)\rceil)$ it follows from (\ref{Inqual 1 Mar2012}) that
\begin{equation*}
\sum_{\hat{\kappa} \in  A_{\Phi}(\gamma,n, k)} \varphi^s(\hat{\kappa}) \geq \sum_{\kappa \in A_{\Phi}(\alpha,l,\lfloor k(1-\rho)\rfloor)}\varphi^s(\kappa\star(\tau|_{\lceil k\rho \rceil}))>1
\end{equation*}
for all $n \leq \lfloor Ql\rfloor$ and $k \geq A(l)$.
For each $n$ we choose $l(k)\in \N$ so that $n \leq \lfloor Ql(n)\rfloor$ and let $B(n) = A(l(n))$. It follows that
\begin{equation*}
\sum_{\hat{\kappa} \in  A_{\Phi}(\gamma,n, k)} \varphi^s(\hat{\kappa}) >1
\end{equation*}
for all $n \in \N$ and for all $k \geq B(n)$.
As in the proof of Proposition \ref{General upper bound prop Mar15}, we get
\begin{align*}
 \dimh(J_{\Phi}^\mathbf{a}(\alpha)) - \eps < s
 \leq \lim_{n \rightarrow \infty} \lim_{k \rightarrow \infty} \sup \bigl\{ D_k(\mu) : \;&\mu \in \M^*_{\sigma^k}(\Sigma) \text{ so that } \\ &\int A_k(\phi_i) d\mu \in B_n(\gamma_i) \text{ for all } i \in \{ 1,\ldots,n \}\bigr\}
\end{align*}
for all $\mathbf{a} \in \mathbf{A}$. Theorem \ref{General theorem} finishes the proof.
\end{proof}

\subsection{Proof of the upper bound in Theorem \ref{main theorem for bounded potentials}}\label{sec:conditional_completeupperbound}
For potentials $\Phi = (\phi_i)_{i \in \N}$ taking values in $\R^\N$ we similarly write $\int \Phi d\nu = (\int \phi_1 d\nu, \int \phi_2 d\nu, \ldots)$ for all $\nu \in \MM_\sigma(\Sigma)$ and $\PP(\Phi) = \{ \int \Phi d\nu : \nu \in \MM_\sigma(\Sigma) \} \subset \R^\N$. The closure of $\PP(\Phi)$ with respect to the product topology is denoted by $\overline{\PP(\Phi)}$.

The following proposition proves the upper bound in Theorem \ref{main theorem for bounded potentials}. In Lemma \ref{ergodic limit is an int of inv measure}, we show that if $\alpha \notin \overline{\PP(\Phi)}$, then $J_\Phi^\mathbf{a}(\alpha) = \emptyset$ for all $\mathbf{a} \in \mathbf{A}$.

\begin{prop}\label{upper bound for main theorem for bounded potentials}
  If $(T_i)_{i \in \N} \in \GL^\N$ is such that $\sup_{i \in \N}\| T_i \| < \tfrac12$, the singular value function $\varphi^s$ is quasi-multiplicative for all $0\le s \le d$, $\Phi \colon \Sigma \to \R^\N$ is bounded with summable variations, and $\alpha \in \overline{\PP(\Phi)}$, then
  \begin{equation*}
    \dimh(J^\mathbf{a}_\Phi(\alpha)) \leq \min\bigl\{ d,\max\bigl\{ s_\infty, \sup\{ D(\mu) : \mu \in \MM_\sigma(\Sigma) \text{ so that } \int \Phi d\mu = \alpha \} \bigr\} \bigr\}
  \end{equation*}
  for $\LL_\mathbf{A}$-almost all $\mathbf{a} \in \mathbf{A}$.
\end{prop}

Dealing first with the special case in which $\Phi$ takes values in $\R^N$, we extend the upper bound for the interior points of the spectrum found in Proposition \ref{special case of main theorem for bounded potentials - finite, interior} to the closure of the spectrum. The proof of Proposition \ref{upper bound for main theorem for bounded potentials} is given after this.

\begin{prop}\label{closure of interior}
  If $(T_i)_{i \in \N} \in \GL^\N$ is such that $\sup_{i \in \N}\| T_i \| < \tfrac12$, the singular value function $\varphi^s$ is quasi-multiplicative for all $0\le s \le d$, $\Phi \colon \Sigma \to \R^N$ is bounded with summable variations, and $\alpha \in \overline{\inter(\PP(\Phi))}$, then
  \begin{equation*}
    \dimh(J^\mathbf{a}_\Phi(\alpha)) \leq \min\bigl\{ d,\max\bigl\{ s_\infty, \sup\{ D(\mu) : \mu \in \MM_\sigma(\Sigma) \text{ so that } \int \Phi d\mu = \alpha \} \bigr\} \bigr\}
  \end{equation*}
  for $\LL_\mathbf{A}$-almost all $\mathbf{a} \in \mathbf{A}$.
\end{prop}

\begin{proof}
If $\inter(\PP(\Phi))=\emptyset$, then there is nothing to prove, so we may assume that $\inter(\PP(\Phi))\neq \emptyset$. Note that in this case, $\PP(\Phi)$ cannot be contained within any $(N-1)$-dimensional hyperplane. In addition, if $\dimh(J^\mathbf{a}_\Phi(\alpha)) \leq s_{\infty}$, then the conclusion of the proposition holds, so we may as well assume that $\dimh(J^\mathbf{a}_\Phi(\alpha)) > s_{\infty}$.

Fix $\alpha \in \overline{\inter(\PP(\Phi))}$. By Lemma \ref{finitely many potentials lemma}, for each $n \in \N$ we may choose $\gamma_n \in \inter(\PP(\Phi))$ with $|\alpha-\gamma_n| < 1/n$ and $\dimh (J^\mathbf{a}_\Phi(\gamma_n)) \geq \dimh (J^\mathbf{a}_\Phi(\alpha)) - 1/n$ for $\LL_\mathbf{A}$-almost all $\mathbf{a} \in \mathbf{A}$.  Since $\dimh(J^\mathbf{a}_\Phi(\alpha)) > s_{\infty}$ we have $\dimh(J^\mathbf{a}_\Phi(\gamma_n)) - 1/n > s_{\infty}$ for all sufficiently large $n$. By Proposition \ref{special case of main theorem for bounded potentials - finite, interior} it follows that for all such $n$ there is a measure $\mu_n \in \MM_\sigma(\Sigma)$ so that $\int \Phi d\mu_n=\gamma_n$ and
\[
  D(\mu_n) > \dimh(J^\mathbf{a}_\Phi(\gamma_n)) - 1/n > s_\infty.
\]
Now by Proposition \ref{Upper-semicont prop Mar15} this implies that the sequence $(\mu_n)_{n \in \N}$ has a weak$^*$ limit $\mu \in \M_{\sigma}(\Sigma)$ with $D(\mu) \geq \limsup_{n\rightarrow \infty} D(\mu_n)$. That is, $D(\mu)= \dimh(J^\mathbf{a}_\Phi(\alpha))$ for $\LL_\mathbf{A}$-almost all $\mathbf{a}$. Moreover, since $\lim_{n \rightarrow \infty} \gamma_n= \alpha$ we have $\int \Phi d\mu= \alpha$.
\end{proof}

\begin{prop}\label{upper bound - all finite valued}
  If $(T_i)_{i \in \N} \in \GL^\N$ is such that $\sup_{i \in \N}\| T_i \| < \tfrac12$, the singular value function $\varphi^s$ is quasi-multiplicative for all $0\le s \le d$, $\Phi \colon \Sigma \to \R^N$ is bounded with summable variations, and $\alpha \in \overline{\PP(\Phi)}$, then
  \begin{equation*}
    \dimh(J^\mathbf{a}_\Phi(\alpha)) \leq \min\bigl\{ d,\max\bigl\{ s_\infty, \sup\{ D(\mu) : \mu \in \MM_\sigma(\Sigma) \text{ so that } \int \Phi d\mu = \alpha \} \bigr\} \bigr\}
  \end{equation*}
  for $\LL_\mathbf{A}$-almost all $\mathbf{a} \in \mathbf{A}$.
\end{prop}

\begin{proof}
We let $(\phi_i)_{i=1}^N$ denote the collection of real-valued maps with $\Phi(\omega)=(\phi_i(\omega))_{i=1}^N$ for each $\omega \in \Sigma$. We begin by taking the smallest possible integer $M \leq N$ so that there is an $M$-dimensional affine subspace of $\R^N$ which contains $\PP(\Phi)$. Then there exist $(j_{l})_{l=1}^M$ with $j_l \in \{ 1,\ldots,N \}$ and for each $i \in \{ 1,\ldots,N \}$ a tuple of reals $(\gamma_{il})_{l=0}^{M}$ such that
\[
  \int \phi_i d\mu= \gamma_{i0}+ \sum_{l=1}^M \gamma_{il}\int \phi_{j_l}d\mu
\]
for all $\mu \in \M_{\sigma}(\Sigma)$ and all $i \in \{1,\ldots,N\}$.
Now define $\Phi' \colon \Sigma \rightarrow \R^M$ by setting
\[
  \Phi'(\omega) = \left(\phi_{j_l}(\omega)\right)_{l=1}^M
\]
for all $\omega \in \Sigma$.
Given $\alpha=(\alpha_i)_{i=1}^N \in \overline{\PP(\Phi)}$ we let $\alpha' = (\alpha_{j_l})_{l=1}^M$. It follows that $J^\mathbf{a}_\Phi(\alpha)\subset J^\mathbf{a}_{\Phi'}(\alpha')$. Moreover, by our choice of $M$, $\PP(\Phi')\subset \R^M$ cannot be contained within any proper $(M-1)$-dimensional affine space. Thus, by Lemma \ref{first lemma in The space of integrals with respect to invariant measures} we have $\overline{\PP(\Phi')}= \overline{\inter(\PP(\Phi'))}$. Moreover, since $\alpha \in \overline{\PP(\Phi)}$ we have $\alpha' \in \overline{\PP(\Phi')}$. Consequently, by Proposition \ref{closure of interior}, we have
\begin{align*}
\dimh(J^\mathbf{a}_\Phi(\alpha)) &\leq \dimh(J^\mathbf{a}_{\Phi'}(\alpha')) \\
&\leq \min\bigl\{ d,\max\bigl\{ s_\infty, \sup\{ D(\mu) : \mu \in \MM_\sigma(\Sigma) \text{ so that } \int \Phi' d\mu = \alpha' \} \bigr\} \bigr\}
\end{align*}
for $\LL_\mathbf{A}$-almost all $\mathbf{a} \in \mathbf{A}$. Now given any $\mu \in \M_{\sigma}(\Sigma)$ with $\int \Phi' d\mu = \alpha'$ we have $\int \phi_{j_l}=\alpha_{j_l}$ for $l \in \{ 1,\ldots, M \}$. Thus
\begin{align*}
\int \phi_i d\mu =\gamma_{i0}+ \sum_{l=1}^M \gamma_{il}\int \phi_{j_l}d\mu =  \gamma_{i0}+\sum_{l=1}^M \gamma_{il}\alpha_{j_l}=\alpha_i
\end{align*}
for all $i \in \{ 1,\ldots, N \}$ and, consequently, $\int \Phi d\mu=\alpha$. The proof is finished.
\end{proof}

\begin{proof}[Proof of Proposition \ref{upper bound for main theorem for bounded potentials}]
Take a bounded potential $\Phi \colon \Sigma \to \R^{\N}$ with summable variations and fix $\alpha \in \overline{\PP(\Phi)}$. We shall apply Proposition \ref{Upper-semicont prop Mar15} in a similar way to the proof of Proposition \ref{closure of interior}. Again, if $\dimh(J^\mathbf{a}_\Phi(\alpha)) \leq s_{\infty}$ then the conclusion of the proposition holds trivially, so we may as well assume that $\dimh(J^\mathbf{a}_\Phi(\alpha)) > s_{\infty}$.

We take $\phi_i \colon \Sigma \rightarrow \R$ and $\alpha_i \in \R$ so that $\Phi=(\phi_i)_{i \in \N}$ and $\alpha=(\alpha_i)_{i \in \N}$. For each $n \in \N$ we define $\Phi_n =(\phi_i)_{i=1}^n$ and $\alpha_n=(\alpha_i)_{i=1}^n$. Then for each $n \in \N$ we have $J^\mathbf{a}_\Phi(\alpha)\subset J^\mathbf{a}_{\Phi_n}(\alpha_n)$. Thus, by applying Proposition \ref{upper bound - all finite valued} we have
\begin{align*}
  \dimh(J^\mathbf{a}_\Phi(\alpha)) & \leq  \dimh(J^\mathbf{a}_{\Phi_n}(\alpha_n)) \\
  &\leq \min\bigl\{ d,\max\bigl\{ s_\infty, \sup\{ D(\mu) : \mu \in \MM_\sigma(\Sigma) \text{ so that } \int \Phi_n d\mu = \alpha_n \} \bigr\} \bigr\}
\end{align*}
for all $n \in \N$.
Since $\dimh(J^\mathbf{a}_\Phi(\alpha)) > s_{\infty}$ we see that for each $n \in \N$ we may choose $\mu_n \in \MM_\sigma(\Sigma)$ so that
\[
  D(\mu_n) > \max\{ \dimh(J^\mathbf{a}_\Phi(\alpha)) - 1/n, s_{\infty} \}
\]
and $\int \Phi_n d\mu_n = \alpha_n$. By applying Proposition \ref{Upper-semicont prop Mar15} we see that the sequence $(\mu_n)_{n \in \N}$ has a limit $\mu \in \MM_\sigma(\Sigma)$ with
\[
  D(\mu) \geq \limsup_{n \rightarrow \infty} D(\mu_n) \geq \dimh(J^\mathbf{a}_\Phi(\alpha)).
\]
Moreover, since $\int \Phi_n d\mu= \alpha_n$ for each $n \in \N$ we have $\int \Phi d\mu = \alpha$. The proof is finished.
\end{proof}

We finish this section by showing that $E_\Phi(\alpha) = \emptyset$ outside of the closure of the spectrum.

\begin{lemma}\label{ergodic limit is an int of inv measure}
  If $\Phi \colon \Sigma \to \R^\N$ is bounded with summable variations and $\alpha \in \R^{\N}$ satisfies $E_{\Phi}(\alpha)\neq \emptyset$, then $\alpha \in \overline{\mathcal{P}(\Phi)}$.
\end{lemma}

\begin{proof}
It suffices to show that for each $q \in \N$ there exists a measure $\mu \in \M_{\sigma}(\Sigma)$ such that
\[
  \int \phi_i d \mu \in \left( \alpha_i - 1/q, \alpha_i + 1/q \right).
\]
for all $i \in \{ 1,\ldots,q \}$.
Now since each $\phi_i$ is uniformly continuous there exists $N \in \N$ for which $\var_n(A_n\phi_i)< (2q)^{-1}$ for all $i \in \{ 1,\ldots,q \}$ and $n \geq N_0$. Moreover, since $E_{\Phi}(\alpha)\neq \emptyset$ we may take $\omega \in E_{\Phi}(\alpha)$. In particular, there exists $N_1 \in \N$ such that for all $i \in \{ 1,\ldots,q \}$ and $n \geq N_1$ we have $A_n\phi_i(\omega) \in (\alpha_i-(2q)^{-1}, \alpha_i+(2q)^{-1})$. From these two facts it follows that if we take $N=\max\left\lbrace N_0,N_1\right\rbrace$ and let $\tau \in \Sigma$ denote the $\sigma^N$ fixed point with $\sigma^{lN}(\tau)\in [\omega|_N]$ for all $l \in \N\cup \{0\}$, then
\[
  A_N\phi_i(\tau) \in \left( \alpha_i - 1/q, \alpha_i + 1/q \right)
\]
for all $i \in \{ 1,\ldots,q \}$.
Thus, if $\mu = N^{-1} \sum_{i=0}^{N-1}\delta_{\sigma^i(\tau)}$, then
\[
  \int\phi_i d\mu \in \left( \alpha_i - 1/q, \alpha_i + 1/q \right)
\]
for all $i \in \{ 1,\ldots,q \}$.
Moreover, since $\tau$ is a fixed point for $\sigma^N$ we conclude that $\mu$ is $\sigma$-invariant.
\end{proof}

\subsection{Proof of the lower bound in Theorem \ref{main theorem for bounded potentials}} \label{sec:conditional_lower}
In this section, we shall prove the lower bound in Theorem \ref{main theorem for bounded potentials}.

\begin{prop} \label{thm:bounded_lower_bound}
  If $(T_i)_{i \in \N} \in \GL^\N$ is such that $\sup_{i \in \N} \| T_i \| < \tfrac12$, the singular value function $\varphi^s$ is quasi-multiplicative for all $0 \le s \le d$, $\Phi \colon \Sigma \to \R^\N$ if bounded with summable variations, and $\alpha \in \overline{\PP(\Phi)}$, then
  \begin{equation*}
    \dimh(J_\Phi^\mathbf{a}(\alpha)) \ge \min\bigl\{ d, \max\bigl\{ s_\infty, \sup\{ D(\mu) : \mu \in \MM_\sigma(\Sigma) \text{ so that } \int\Phi d\mu = \alpha \} \bigr\} \bigr\}
  \end{equation*}
  for $\LL_\mathbf{A}$-almost all $\mathbf{a} \in \mathbf{A}$.
\end{prop}

By Theorem \ref{LB in main Thm Mar30}, to prove Proposition \ref{thm:bounded_lower_bound}, it suffices to show that
\begin{align*}
\lim_{n \rightarrow \infty} \lim_{k \rightarrow \infty} \sup \bigl\{ D_k(\mu) : \;&\mu \in \M^*_{\sigma^k}(\Sigma) \text{ so that } \int A_i\phi_i d\mu \in B_n(\alpha_i) \text{ for all } i \in \{ 1,\ldots,n \}\bigr\} \\
&\ge \max\bigl\{ s_\infty, \sup\{ D(\mu) : \mu \in \MM_\sigma(\Sigma) \text{ so that } \int\Phi d\mu = \alpha \} \bigr\}.
\end{align*}
This inequality is shown in the following two lemmas.

\begin{lemma}
  If $(T_i)_{i \in \N} \in \GL^\N$ is such that $\sup_{i \in \N} \| T_i \| < 1$, the singular value function $\varphi^s$ is quasi-multiplicative for all $0 \le s \le d$, $\Phi \colon \Sigma \to \R^\N$ is bounded with summable variations, and $\alpha \in \R^\N$, then
  \begin{equation*}
    D(\mu) \le \lim_{n \rightarrow \infty} \lim_{k \rightarrow \infty} \sup \bigl\{ D_k(\mu) : \mu \in \M^*_{\sigma^k}(\Sigma) \text{ so that } \int A_i\phi_i d\mu \in B_n(\alpha_i) \text{ for all } i \in \{ 1,\ldots,n \}\bigr\}
  \end{equation*}
  for all $\mu \in \MM_\sigma(\Sigma)$ with $\int\Phi d\mu = \alpha$.
\end{lemma}

\begin{proof}
Fix $\mu \in \M_{\sigma}(\Sigma)$ with $\int \Phi d\mu=\alpha$ and let $0 \le s < D(\mu)$. It follows that $P_\mu(\varphi^s) \ge 0$. Thus
\begin{equation*}
  \sum_{\omega \in \Sigma_k} \mu([\omega]) \log\frac{\varphi^s(\omega)}{\mu([\omega])} \ge 0
\end{equation*}
and $D_k(\mu) \ge s$ for all $k \in \N$. Moreover, since $\mu$ is $\sigma$-invariant we have
\begin{equation*}
\int A_k\Phi d\mu = \int \Phi d\mu = \alpha \in B_n(\alpha)
\end{equation*}
for all $k,n \in \N$. Hence
\begin{equation*}
s < \lim_{k \rightarrow \infty} \sup \bigl\{ D_k(\nu) : \nu \in \M^*_{\sigma^k}(\Sigma) \text{ so that } \int A_i\phi_i d\nu \in B_n(\alpha_i) \text{ for all } i \in \{ 1,\ldots,n \}\bigr\}
\end{equation*}
Letting $n \rightarrow \infty$ and $s \uparrow D(\mu)$ completes the proof of the lemma.
\end{proof}

\begin{lemma}
  If $(T_i)_{i \in \N} \in \GL^\N$ is such that $\sup_{i \in \N} \| T_i \| < 1$, the singular value function $\varphi^s$ is quasi-multiplicative for all $0 \le s \le d$, $\Phi \colon \Sigma \to \R^\N$ is bounded with summable variations, then
  \begin{equation*}
    s_\infty \le \lim_{n \rightarrow \infty} \lim_{k \rightarrow \infty} \sup \bigl\{ D_k(\mu) : \mu \in \M^*_{\sigma^k}(\Sigma) \text{ so that } \int A_i\phi_i d\mu \in B_n(\alpha_i) \text{ for all } i \in \{ 1,\ldots,n \}\bigr\}
  \end{equation*}
  for all $\alpha \in \overline{\PP(\Phi)}$.
\end{lemma}

\begin{proof}
It suffices to show that for any $s< s_{\infty}$ and $n\in \N$ there exists $k(n) \in \N$ such that for all $n \geq k(n)$ there exists $\mu \in \M^*_{\sigma^k}(\Sigma)$ with $D_k(\mu) \geq s$ and $\int A_k\phi_i d\mu \in B_n(\alpha_i)$ for $i \in \{ 1,\ldots,n \}$.

First take $k_0(n)$ so that $\var_k\left(A_k\phi_i\right)< (4n)^{-1}$ for all $i \in \{ 1,\ldots,n \}$. Let $\Phi^n \colon \Sigma \rightarrow \R^n$ denote the potential $\left(\phi_i \right)_{i=1}^n$. Since $\alpha \in \overline{\mathcal{P}(\Phi)}$ we have $(\alpha_i)_{i=1}^n \in \overline{\mathcal{P}(\Phi^n)} \subset \overline{\bigcup_{I} \mathcal{P}_e(\Phi^n,I)}$ by Lemma \ref{second lemma in The space of integrals with respect to invariant measures}. Thus there exists an ergodic invariant measure $\nu$ with $\int \phi_i d\nu \in B_{4n}(\alpha_i)$ for all $i \in \{ 1,\ldots,n \}$. Since $\nu$ is ergodic we obtain $\tau \in \Sigma$ and $k(n) \geq k_0(n)$ such that for all $k \geq k(n)$ we have $A_k\phi_i(\tau) \in B_{4n}(\alpha_i)$. Since $k(n) \geq k_0(n)$ we have $A_k\phi_i(\kappa) \in B_{2n}(\alpha_i)$ for all $k \geq k(n)$ and $\kappa \in [\tau|_k]$. Now choose $\rho \in (0,1)$ sufficiently large that
\begin{equation*}
 \rho\bigl(\alpha_i-\tfrac{1}{2n}\bigr)-(1-\rho) \|\phi_i\| > \alpha_i-\tfrac{1}{n} \text{ and }
 \rho \bigl(\alpha_i+\tfrac{1}{2n}\bigr)+(1-\rho) \|\phi_i\| < \alpha_i+\tfrac{1}{n}.
\end{equation*}
for all $i \in \{ 1,\ldots,n \}$. It follows that for any $k \geq k(n)$, any measure $\tilde{\mu}$ with $\tilde{\mu}([\tau|_k]) = \rho$ will satisfy $\int A_k\phi_i d \tilde{\mu} \in B_n(\alpha_i)$ for all $i \in \{ 1,\ldots,n \}$.

Since $s < s_{\infty}$ we have $\sum_{\omega \in \Sigma_k} \varphi^s(\omega) = \infty$ for all $k \in \N$. As such, for each $k \ge k(n)$ we choose a finite subset $C(k) \subset \Sigma_k \setminus \{ \tau|_k \}$ with
\begin{equation*}
\sum_{\omega \in C(k)} \varphi^s(\omega) > (\varphi^s(\tau|_k))^{-\rho/(1-\rho)}.
\end{equation*}
Let $\mu$ denote the unique $k$-th level Bernoulli measure satisfying
\begin{equation*}
\mu(\omega) =
\begin{cases}
(1-\rho) \varphi^s(\omega)/\sum_{\kappa \in C(k)} \varphi^s(\kappa), &\text{if } \omega \in C(k), \\
\rho, &\text{if } \omega = \tau|_k, \\
0, &\text{if } \omega \notin C(k) \cup \{ \tau|_k \}.
\end{cases}
\end{equation*}
Since $\mu([\tau|_k])=\rho$ we have $\int A_k\phi_i d\mu \in B_n(\alpha_i)$ for $i \in \{ 1,\ldots,n \}$. Moreover,
\begin{align*}
\sum_{\omega \in \Sigma_k}\mu([\omega])\log \frac{\varphi^s(\omega)}{\mu([\omega])}
&= \rho \log \frac{\varphi^s(\tau|_k)}{\rho} + \sum_{\omega \in C(k)} \frac{(1-\rho)\varphi^s(\omega)}{\sum_{\kappa \in C(k)}\varphi^s(\kappa)} \log \frac{\sum_{\kappa \in C(k)}\varphi^s(\kappa)}{(1-\rho)} \\
&\geq \rho \log \varphi^s(\tau|_k) +(1-\rho)\log \biggl(\sum_{\kappa \in C(k)}\varphi^s(\kappa)\biggr) > 0.
\end{align*}
Hence $D_k(\mu)>s$. This completes the proof of the lemma.
\end{proof}


\end{document}